\documentclass[a4paper]{amsart}

\usepackage[applemac]{inputenc}
\usepackage[T1]{fontenc}
\usepackage[english]{babel}
\usepackage{babelbib}
\usepackage{enumerate}
\usepackage{url}
\usepackage{amssymb}
\usepackage[hypertexnames=false,backref=page,pdftex,
 	pdfpagemode=UseNone,
 	breaklinks=true,
 	extension=pdf,
 	colorlinks=true,
 	linkcolor=blue,
 	citecolor=red,
 	urlcolor=blue,
 ]{hyperref}
\usepackage{color}
\usepackage{tkz-fct}
\usepackage{tikz}
\usepackage{tikz-cd}
\usetikzlibrary{matrix}
\usetikzlibrary{arrows}
\usepackage{graphicx}
\usepackage{mathrsfs}
\usepackage{cleveref}

\mathcode`A="7041 \mathcode`B="7042 \mathcode`C="7043 \mathcode`D="7044
\mathcode`E="7045 \mathcode`F="7046 \mathcode`G="7047 \mathcode`H="7048
\mathcode`I="7049 \mathcode`J="704A \mathcode`K="704B \mathcode`L="704C
\mathcode`M="704D \mathcode`N="704E \mathcode`O="704F \mathcode`P="7050
\mathcode`Q="7051 \mathcode`R="7052 \mathcode`S="7053 \mathcode`T="7054
\mathcode`U="7055 \mathcode`V="7056 \mathcode`W="7057 \mathcode`X="7058
\mathcode`Y="7059 \mathcode`Z="705A

\renewcommand{\mathcal}{\mathscr}

\newcommand{\bbA}{\mathbb{A}}

\newcommand{\bbC}{\mathbb{C}}

\newcommand{\bbG}{\mathbb{G}}

\newcommand{\bbN}{\mathbb{N}}

\newcommand{\bbP}{\mathbb{P}}
\newcommand{\bbQ}{\mathbb{Q}}

\newcommand{\bbV}{\mathbb{V}}

\newcommand{\bbZ}{\mathbb{Z}}

\newcommand{\Gm}{\mathbb{G}_m}

\newcommand{\cA}{\mathcal{A}}

\newcommand{\cE}{\mathcal{E}}
\newcommand{\cF}{\mathcal{F}}

\newcommand{\cH}{\mathcal{H}}
\newcommand{\cI}{\mathcal{I}}

\newcommand{\cK}{\mathcal{K}}
\newcommand{\cL}{\mathcal{L}}

\newcommand{\cO}{\mathcal{O}}

\newcommand{\cT}{\mathcal{T}}
\newcommand{\cU}{\mathcal{U}}

\newcommand{\cW}{\mathcal{W}}

\newcommand{\rG}{\textup{G}}
\newcommand{\rH}{\textup{H}}

\newcommand{\rR}{\textup{R}}

\newcommand{\rT}{\textup{T}}

\newcommand{\rd}{\textup{d}}

\newcommand{\frm}{\mathfrak{m}}

\newcommand{\an}{\textup{an}}

\newcommand{\aff}{\textup{aff}}
\newcommand{\ant}{\textup{ant}}

\newcommand{\et}{\textup{\'et}}

\newcommand{\too}{\longrightarrow}
\renewcommand{\phi}{\varphi}
\renewcommand{\epsilon}{\varepsilon}
\renewcommand{\ker}{\Ker}

\DeclareMathOperator{\Gal}{Gal}

\DeclareMathOperator{\pr}{pr}

\DeclareMathOperator{\Spec}{Spec}

\DeclareMathOperator{\Spf}{Spf}
\DeclareMathOperator{\cHom}{\mathcal{H}\hspace{-2.5pt}\textit{om}}

\DeclareMathOperator{\Hom}{Hom}

\DeclareMathOperator{\im}{Im}
\DeclareMathOperator{\Ker}{Ker}

\DeclareMathOperator{\coker}{Coker}

\DeclareMathOperator{\red}{red}

\DeclareMathOperator{\Pic}{Pic}

\DeclareMathOperator{\Sym}{Sym}
\DeclareMathOperator{\hSym}{\widehat{Sym}}
\DeclareMathOperator{\GL}{GL}

\DeclareMathOperator{\Lie}{Lie}

\DeclareMathOperator{\ev}{ev}
\DeclareMathOperator{\id}{id}

\DeclareMathOperator{\Res}{Res}

\DeclareMathOperator{\At}{At}
\DeclareMathOperator{\Fil}{Fil}

\DeclareMathOperator*{\hotimes}{\hat{\otimes}}

\DeclareMathOperator{\characteristic}{char}

\DeclareMathOperator{\tr}{tr}
\DeclareMathOperator*{\hbigoplus}{\hat{\bigoplus}}

\renewcommand{\le}{\leqslant}
\renewcommand{\ge}{\geqslant}

\renewcommand{\projlim}{\varprojlim}
\renewcommand{\injlim}{\varinjlim}

\theoremstyle{plain}
\newtheorem{theoremintro}{Theorem}
\newtheorem*{maintheorem}{Main Theorem}
\newtheorem{theorem}{Theorem}[section]
\newtheorem{lemma}[theorem]{Lemma}
\newtheorem{proposition}[theorem]{Proposition}
\newtheorem{corollary}[theorem]{Corollary}
\newtheorem{claim}[theorem]{Claim}

\theoremstyle{definition}
\newtheorem{definition}[theorem]{Definition}
\newtheorem{example}[theorem]{Example}

\theoremstyle{remark}
\newtheorem{remark}[theorem]{Remark}

\numberwithin{equation}{section}

\begin{document}

\title{Rigid analytic Stein algebraic groups are affine}

\author{Marco Maculan}
\email{marco.maculan@imj-prg.fr}
\address{Institut de Math\'ematiques de Jussieu, Sorbonne Universit\'e, 4 place Jussieu, F-75252 Paris}

\begin{abstract}
Let $K$ be a complete non-trivially valued  non-Archimedean field. Given an algebraic group over $K$ on which every regular function is constant, any rigid analytic function is shown to be constant too. It follows that an algebraic group over $K$ is affine if and only if the associated $K$-analytic space is Stein; that is, 
rigid analytic embeddings of it in an affine space may always be chosen to be given by algebraic functions. Arguably curiously, the corresponding statement over the complex numbers is false.
\end{abstract}

\maketitle

%\setcounter{tocdepth}{1}
%\tableofcontents

\setcounter{section}{-1}

\section{Introduction}

\subsection{Motivation}
Throughout this introduction, let $K$ be a complete non-trivially valued non-Archimedean field, \emph{e.g.} $K = \bbQ_p$, $\bbC_p$, $k(\!(t)\!)$ for a field $k$, \emph{etc.} For an algebraic variety  $X$ over $K$, that is, a finite type separated $K$-scheme, let $X^\an$ denote the $K$-analytic space\footnote{In this paper analytic spaces over $K$ are considered in the sense of Berkovich. Nonetheless, the main results are insensible to the choice of the language.} attached to it. The algebraic variety $X$ is \emph{Stein} if there is a closed embedding of $K$-analytic spaces $X^\an \hookrightarrow \bbA^{n, \an}_K$ for some $n \ge 0$. Of course, affine varieties are Stein and the frivolous question at the origin of this paper is whether the converse holds true. Counterexamples exist for the corresponding affirmation over the complex numbers, and they can be found already in the restrictive class of algebraic groups. Yet, such examples do not occur over $K$ (see \Cref{Cor:AlgebraicGroupSteinIFFAffine}):

\begin{maintheorem} \label{Thm:AffineVsStein} An algebraic group\footnote{Namely, a group $K$-scheme of finite type.} over $K$ is Stein if and only if it is affine.
\end{maintheorem}

The interest of such a statement, if any, resides in that it reveals yet another divergence between complex and rigid analysis. Before describing more in detail the content of this article, let me briefly recall what happens over $\bbC$. As pointed out by Serre, there are non-affine complex algebraic groups that admit a  closed holomorphic embedding in $\bbC^n$ (for some integer $n \ge 0$). The leading example is the universal vector extension $A^\natural$ of a (non-trivial) complex abelian variety $A$, that is, the moduli space of rank $1$ connections on the dual abelian variety $\check{A}$. The tensor product of line bundles equipped with a connection endows $A^\natural$ with the structure of a complex algebraic group, which sits in the following short exact sequence:
\[ 0 \too \bbV(\omega_{\check{A}}) \too A^\natural \too A \too 0,\]
where the map $A^\natural \to A$ forgets the connection on an algebraically trivial line bundle on $\check{A}$, and the map $\bbV(\omega_{\check{A}}) \to A^\natural$ associates to a differential form $\eta$ on $\check{A}$, the connection on the trivial line bundle $\cO_{\check{A}}$ given by the sum $\rd + \eta$ of the canonical derivation and $\eta$. (The space  of global differential forms on $\check{A}$ is identified with the dual $\omega_{\check{A}}$ of the Lie algebra of $\check{A}$.) On the one hand, the algebraic variety $A^\natural$ is not affine because the quotient of an affine algebraic group is always affine---and $A$ is certainly not so. On the other, a holomorphic connection on a line bundle on $A$ being integrable, the Riemann-Hilbert correspondence yields a biholomorphism 
\[ A^{\natural, \an} \cong \Hom_{\textup{Groups}}(\pi_1(\check{A}(\bbC), 0), \bbC^\times) \cong (\bbC^\times)^{2g},\]where $g := \dim A$, and $\pi_1(\check{A}(\bbC), 0)$ is the topological fundamental group of $\check{A}(\bbC)$ with $0$ as base-point. In particular, the complex manifold $A^{\natural, \an}$ admits a closed holomorphic embedding in $\bbC^{2g + 1}$. 

The difference between affine and Stein varieties is quite subtle. Neeman exhibited a quasi-affine (that is, admitting an open immersion in an affine variety) complex variety which is Stein but not affine.\footnote{The algebraic functions on $A^\natural$ are all constant thus $A^\natural$ is not quasi-affine.} Let me recall his construction. Let $H$ be an ample line bundle on $A$ and $P$ be the total space of the line bundle $p^\ast H^\vee$ deprived of its zero section, where $p \colon A^\natural \to A$ is the projection. The principal $\Gm$-bundle $P$ is quasi-affine by ampleness of $H$, but not affine: otherwise $P / \Gm = A^\natural$ would be so. Besides, a principal $\bbC^\times$-bundle on a Stein space is itself Stein,\footnote{This is true for principal bundle under a Stein complex Lie group, see \cite{MatsushimaMorimoto}.} hence so is $P$. A quasi-affine complex variety $X$ which is Stein but not affine forces its  $\bbC$-algebra of global sections $\Gamma(X, \cO_X)$ not to be finitely generated \cite[Proposition 5.5]{Neeman}. As $P$ is such a variety, it yields a counter-example of Hilbert's fourteenth problem, as extended by Zariski (\emph{op.cit.} Remark 8.2).

\subsection{Main results} Let us place ourselves over $K$ again. Needless to say, in order to prove the main theorem rigid analytic functions on algebraic groups have to be understood. Employing Brion's nomenclature an algebraic group $G$ over $K$ is said to be \emph{anti-affine} if $\Gamma(G, \cO_G) = K$. An anti-affine group is commutative, connected and smooth. In contrast with the complex situation, anti-affine algebraic groups do not admit non-constant analytic functions (\Cref{Cor:FunctionsOnAntiAffineGroupsText}):

\begin{theoremintro} \label{Thm:FunctionsOnAntiAffineGroups} Any analytic function on an anti-affine algebraic group is constant.
\end{theoremintro}

When the field $K$ is of positive characteristic the proof of \Cref{Thm:FunctionsOnAntiAffineGroups} is rather straightforward because anti-affine groups in positive characteristic are all semi-abelian varieties.  %For a semi-abelian variety, in arbitrary characteristic, analytic functions are expressed by means of `Fourier expansion' (\emph{cf.} \Cref{sec:FourierExpansion}) as series of global sections of (translation-invariant) lines bundles on the underlying abelian variety, and the GAGA principle permits to conclude (see \Cref{Cor:AnalyticFunctionsVectExtSemiAbelian}). Note that these considerations go through also in the complex case. 
Now, for an algebraic group $G$, the $K$-algebra $\Gamma(G, \cO_G)$ is of finite type and the natural structure of Hopf algebra on it endows  $G_\aff := \Spec \Gamma(G, \cO_G)$ with the structure of an algebraic group. The canonical map $\pi_G\colon G \to G_{\aff}$ is a morphism of algebraic groups with anti-affine. An easy descent argument shows the following (\Cref{Thm:FunctionsOnAnalyticGroupsText}):

\begin{theoremintro}  \label{Thm:FunctionsOnAnalyticGroups} All analytic functions on $G$ come from $G_\aff$ by precomposing with~$\pi_G$.
\end{theoremintro}

The main theorem is an immediate consequence of \Cref{Thm:FunctionsOnAnalyticGroups}. %Indeed, the latter implies that a morphism $f \colon G^\an \to \bbA^{n, \an}_K$ of $K$-analytic spaces factors through $G_{\aff}^\an$. If $f$ happens to be a closed embedding, then $\ker \pi$ is trivial. Also, \Cref{Thm:FunctionsOnAnalyticGroups} will be easily deduced from \Cref{Thm:FunctionsOnAntiAffineGroups} by a descent argument.
Brion has generalized Neeman's example in the following way. Let $G$ be a an anti-affine extension of an abelian variety $A$ and $H$ an ample line bundle on $A$. Let $P$ denote the total space of $p^\ast H^\vee$ deprived of its zero section where $p \colon G \to A$ is the projection. By ampleness the variety $P$ is quasi-affine and Brion shows that the ring $\Gamma(P, \cO_P)$ is not Noetherian \cite[Theorem 3.9]{BrionAntiAffine}. The equivalence between affine and Stein holds even in this more intricate example (\Cref{Thm:NeemanExampleIsNotSteinText}):
%Thanks to Chevalley's theorem (\cite{ConradChevalley}), a connected and reduced algebraic group $G$  over $K$ is an extension of an abelian variety $A$ by a linear algebraic group $L$:
%\[ 0 \too L \too G \stackrel{p}{\too} A \too 0.\]
%
%Let $H$ be a line bundle on the abelian variety $A$ and $P$ the total space of the line bundle $p^\ast H^\vee$ on $G$ deprived of its zero section. If $H$ is ample, then the algebraic variety $P$ is quasi-affine. Moreover, if $G$ is anti-affine and its linear part $L$ is non-trivial, Brion shows that the ring $\Gamma(P, \cO_P)$ is not Noetherian (), furnishing a counter-example to Zariski's version of Hilbert's fourteenth problem.\footnote{Such an algebraic group $G$ exists because $K$ is not an algebraic extension of $\bbF_p$: the only anti-affine groups on $\bar{\bbF}_p$ are abelian varieties (\cite[Proposition 2.2]{BrionAntiAffine}).}

\begin{theoremintro} Let $G$ be an algebraic group, $L$ a line bundle on $G$ and $P$ the total space of $L$ deprived of its zero section. Then the following are equivalent:
\[ \text{$G$ is affine} \iff \text{$G$ is Stein} \iff \text{$P$ is affine} \iff \text{$P$ is Stein}.  \]
\end{theoremintro}

%In particular, if $G$ is anti-affine and non-trivial, then $P$ is not Stein, showing that the equivalence between the notions of Stein and affine varieties hold even in more intricate examples.

\subsection{The role of the universal vector extension} \Cref{Thm:FunctionsOnAnalyticGroups} is most interesting when $\characteristic K = 0$ and for the universal vector extension of an abelian variety.   Besides, the algebraicity plays no role and the techniques employed go through for an abeloid variety (the rigid analytic analogue of a complex torus).  The remainder of this introduction will focus on the proof the following result (\Cref{Thm:AnalyticFunctionsOnVectorExt}):

\begin{theoremintro} \label{ThmIntro:FunctionsOnUniversalExtAreConstant} If $\characteristic K = 0$, then any analytic function on the the universal vector extension of an abeloid variety is constant.
\end{theoremintro}

The proof has a rather different flavour depending on the reduction behaviour of the abeloid variety $A$. The case of totally degenerate reduction is perhaps the more intuitive. Indeed, passing to finite extension of $K$, the universal cover  of $A$ is a $K$-torus $T$ and the topological fundamental group is identified with a free abelian group $\Lambda \subseteq T(K)$ of rank $\dim T$. Consider the torus $\check{T}$ with group of characters $\Lambda$ and $\omega_{\check{T}}$ the dual of its Lie algebra. Then the universal vector extension $A^\natural$ can be identified with the quotient $(T \times \bbV(\omega_{\check{T}}))/ \Lambda$. Here the action is given by
\[ \chi.(t, v) = (\chi t, v + \theta_{\Lambda}(\chi))\]
where $\theta_\Lambda(\chi) = \chi^\ast \tfrac{\rd z}{z}$ and $\chi \in \Lambda$ is seen as a character $\chi \colon \check{T} \to \Gm$.  Now an analytic function on $A^\natural$ is an analytic function on $T \times \bbV(\omega_{\check{T}})$ invariant under the action of~$\Lambda$. Since the image of $\theta_\Lambda \colon \Lambda \to \omega_{\check{T}}$ spans the $K$-vector space $\omega_{\check{T}}$ and accumulates to $0$, such an invariant function is necessarily constant.

Suppose that $A$ has good reduction, that is, it is the Raynaud generic fiber of a formal abelian scheme $\cA$. In this case the topology offers no information as the topological space underlying the universal vector extension $A^\natural$ is contractible. When $K$ has $0$  residue characteristic Coleman proved that all algebraic functions on each successive thickening are constant. Passing to the limit then permits to conclude. When $K$ is a valued extension of $\bbQ_p$ the situation is more interesting. There is no loss of generality in supposing $K$ algebraically closed. The idea is to replace the topological universal cover by the `perfectoid' one 
\[ \tilde{\cA} = \textstyle \varprojlim_{\times p} \cA,\] where the transition maps are the multiplication by $p$.\footnote{See \Cref{Sec:FormalProjectiveLimits} for the precise meaning of this projective limit.} Then $\cA^\natural$ can be identified with the quotient $(\tilde{\cA} \times \bbV(\omega_{\check{A}})) / \rT_p \cA$ where $\omega_{\check{A}}$ is the dual of the Lie algebra of the dual formal abelian scheme $\check{A}$ and
\[ \rT_p \cA = \varprojlim_{n \ge 1} \cA[p^n]\]
is the Tate module. (Actually the projective limit is taken scheme-theoretically so the result carries a non-trivial non-reduced structure.) The action of $\rT_p \cA$ is given by some morphism of (formal schemes in) groups $\theta_{\rT_p \cA} \colon \rT_p \cA \to \bbV(\omega_{\check{A}})$. Following an insight of Coleman \cite[p. 379]{ColemanHodgeTate}, \cite[\S4]{ColemanBiextension} and Faltings \cite[Theorem 4]{FaltingsHodgeTate}, the $K$-linear map 
\[ \theta_{\rT_p \cA} \colon \rT_p \cA(R) \otimes_{\bbZ_p} K \too \omega_{\check{\cA}} \otimes_R K = \omega_{\check{A}} \]
is surjective and leads to the Hodge-Tate decomposition of $\rH^1_{\et}(A, \bbQ_p)$. Moreover the analytic functions on $A^\natural$ are those on~$\bbV(\omega_{\check{A}})$ invariant under the translation by~$\rT_p \cA(R)$. Since the image of $\theta_{\rT_p \cA}$ spans the $K$-vector space $\omega_{\check{A}}$ and accumulates to $0$, such functions are necessarily constant. The above description of $\cA^\natural$ and its formal functions is carried out first at the algebraic level in \Cref{Sec:RepresentationCanonicalExtensionAlgebraic,,sec:AlgebraicFunctionsVectorExtensions}  for all successive thickenings and then  passed to the limit in \Cref{Sec:AlmostFiniteAndHodgeTateDecomposition}.

Finally, when $A$ has intermediate reduction, the proof mixes the two techniques. In this case the universal cover of $A$ is an extension
\[ 0 \too T \too E \too B \too 0 \]
where $T$ is a $K$-torus and $B$ an abeloid variety with good reduction. The topological fundamental group of $A$ is identified with a free abelian group $\Lambda \subseteq E(K)$ of rank $\dim T$. The universal cover of the dual abeloid variety is a similar extension
\[ 0 \too \check{T} \too \check{E} \too \check{B} \too 0 \]
where $\check{B}$ is the dual of $B$ and $\check{T}$ is the $K$-torus with group of characters $\Lambda$. Then the universal cover $E^\natural$  of $A^\natural$ is the push-out of $B^\natural \times_B E$ along the $K$-linear map $\omega_{\check{B}} \to \omega_{\check{E}}$ induced by the projection $\check{E} \to \check{B}$. Here $\omega_{\check{B}}$ is the dual of the Lie algebra of $\check{B}$, and similarly for $\omega_{\check{E}}$. In other words, the following diagram is exact and commutative:
\[
\begin{tikzcd}
0 \ar[r] & \bbV(\omega_{\check{B}}) \ar[r] \ar[d] & B^\natural \times_B E \ar[r] \ar[d] & E \ar[r] \ar[d, equal] & 0  \\
0 \ar[r] & \bbV(\omega_{\check{E}}) \ar[r] & E^\natural \ar[r] & E \ar[r] & 0  
\end{tikzcd}
\]
The topological fundamental group of $A^\natural$ is identified with a subgroup $\Lambda^\natural \subseteq E^\natural(K)$ and the projection $E^\natural \to E$ induces an isomorphism $\Lambda^\natural \cong \Lambda$. Once again analytic functions on $A^\natural$ correspond to $\Lambda^\natural$-invariant analytic functions on $E^\natural$. A careful analysis permits to show that the latter are all constant. This requires the knowledge of functions on \emph{all} commutative extensions of abeloid varieties with good reduction, and for this reason all extensions are treated at once.

\subsubsection*{Acknowledgements} This research was supported by ANR-18-CE40-0017.

\subsection{Reminders and conventions}Let $X$ be a locally ringed space and
\begin{equation} \tag{$F$}
\cdots \too F_{i - 1} \too F_i \too F_{i + 1} \too \cdots
\end{equation}
a sequence of $\cO_{X}$-modules indexed by integers. For morphisms of locally ringed spaces $f \colon Y \to X$ and $g \colon X \to Z$, and an $\cO_X$-module $M$, let $f^\ast (F)$, $g_\ast (F)$, $(F) \otimes M$ denote the sequences obtained from $(F)$ respectively pulling-back along $f$, pushing-forward along $g$, and taking the tensor product with $M$.

Let $K$ be a complete nontrivally valued non-Archimedean  field. In this paper $K$-analytic spaces are considered in the sense of Berkovich (\cite{BerkovichIHES}). By an abuse of notation, given a $K$-analytic space $X$, an $\cO_X$-module here is what is called an $\cO_{X_{\rG}}$-module in \emph{op.cit.}. As soon as the $K$-analytic space $X$ is good (that is, every point admits an affinoid neighbourhood) the  two notions coincide (\emph{op.cit.} Proposition 1.3.4). For a point $x$ in a $K$-analytic space $X$, let $\cH(x)$ be the completed residue field. For a $K$-analytic space $S$, an $S$-analytic space in groups will be called simply an $S$-analytic group. An \emph{abeloid variety} over $S$ is a proper, smooth $S$-analytic group with connected fibers.

\section{Applications to Stein algebraic groups} \label{sec:ApplicationToAlgebraicGroups}

Let $k$ be a field and $\bar{k}$ a separable closure of $k$. For a $k$-scheme $X$ let $\bar{X}$ be the $\bar{k}$-scheme obtained from $X$ by extending scalars to $\bar{k}$. An \emph{algebraic group} over $k$ is a finite type group $k$-scheme.

\subsection{Classification of algebraic groups} Let $A$ be an abelian variety over $k$, $\check{A}$ its dual and $\omega_{\check{A}} = (\Lie \check{A})^\vee$.   

Let  $T$ be a $k$-torus and $X^\ast(\bar{T})$ the group of characters of the split $\bar{k}$-torus $\bar{T}$. Suppose given a Galois-equivariant group homomorphism $c \colon X^\ast(\bar{T}) \to \check{A}(\bar{k})$. The datum of $T$ and $c$ corresponds to a semi-abelian variety
\[ 0 \too T \too E(T, c) \too A \too 0.\]
A point $g$ of $E(T, \phi)$ with values in a $\bar{k}$-scheme $S$ is a couple made of an $S$-valued point $x$ of $\bar{A}$ and, for $\chi \in \Lambda$, of an isomorphism $\langle g , \chi \rangle \colon \cO_S \to x^\ast L_\chi$ where $L_\chi$ is the pull-back along $(\id, c(\chi)) \colon \bar{A} \to \bar{A} \times \check{\bar{A}}$ of the Poincar\'e bundle $\cL$ on $A \times \check{A}$. For $\chi, \chi' \in \Lambda$ the trivializations $\langle g, \chi \rangle$ undergo the relation $ \langle g , \chi \rangle \otimes \langle g , \chi' \rangle = \langle g , \chi + \chi' \rangle$ which is meant to be understood via the isomorphism $L_\chi \otimes L_{\chi'} \cong L_{\chi + \chi'}$. 

Let $F$ a finite-dimensional $k$-vector space and $\phi \colon \omega_{\check{A}} \to F$ a $k$-linear map. The datum of $F$ and $\phi$ yields a vector extension
\[ 0 \too \bbV(F) \too V(F, \phi) \too A \too 0.\]
Its isomorphism $[V(F, \phi)]$ class as a principal $\bbV(F)$-bundle lives in the cohomology group $\rH^1_{\textup{fppf}}(A, \bbV(F)) = \rH^1(A, \cO_A) \otimes_k F \cong \Lie \check{A} \otimes_k F = \Hom(\omega_{\check{A}}, F)$
where the first equality is \cite[Exp. XI, Proposition 5.1]{SGA1}, the second is the usual expression for the Lie algebra of the dual abelian variety \cite[8.4, Theorem 1]{NeronModels} and the third holds by definition. The vector extension $V(F, \phi)$ is characterized by $[V(F, \phi)] = \phi$ via the preceding identification. 

Summing up the quintuple $(A, T, F, c, \phi)$ gives rise to the smooth, commutative, connected algebraic group \[E(T, c) \times_A V(F, \phi).\]

% and $F$ a finite-dimensional $k$-vector space. Let Suppose given a $k$-linear map $\phi \colon \omega_{\check{A}} \to F$ and  where $\bar{k}$ is a \emph{separable} closure of $k$ and $X^\ast(\bar{T})$ the group of characters of the split $\bar{k}$-torus $\bar{T} := T \times_k \bar{k}$. Out of the quintuple
%\[ (A, T, F, c, \phi)\]
%it is possible to construct a smooth commutative algebraic group as follows.

Conversely suppose that the field $k$ is perfect and let $G$ be a connected and reduced (or, equivalently, smooth) algebraic group. By Chevalley's theorem (\cite{ConradChevalley}), the algebraic group $G$ is an extension of an abelian variety $A$ by an affine algebraic group $L$. The algebraic groups $A$ and $L$ are respectively called the \emph{abelian} and \emph{linear} parts of $G$. When the algebraic group $G$ is moreover commutative, the linear part $L$ of $G$ is of the form $T \times_k U$ for a $k$-torus $T$ and a unipotent group $U$, called respectively the \emph{toric} and \emph{unipotent} parts of $G$. The algebraic group $G/U$ is a semi-abelian variety. When $\characteristic(k) = 0$  the unipotent group $U$ is of the form $\bbV(F)$ for a finite-dimensional $k$-vector space \cite[\S 2.7, p. 172]{SerreGroupesAlgebriques}, while in positive characteristic this may fail. If $U$ is a vector group, then $G/T$ is a vector extension of $A$. These considerations lead to the following classification:

\begin{lemma} Suppose $k$ perfect. The map $(A, T, F, c, \phi) \mapsto E(T, c) \times_A V(F, \phi)$ defined above sets up a bijection between the sets of isomorphism classes of
\begin{enumerate}
\item quintuples $(A, T, F, c, \phi)$ made of an abelian variety, a $k$-torus $T$, a finite-dimensional $k$-vector space $F$, a  $\Gal(\bar{k}/k)$-equivariant homomorphism of groups $c \colon X^\ast(\bar{T}) \to \check{A}(\bar{k})$ and a $k$-linear map $\phi \colon \omega_{\check{A}} \to F$;
\item commutative, connected and reduced algebraic groups over $k$ whose unipotent part is a vector group,
\end{enumerate}
the notion of isomorphism for these quintuples being defined in the evident manner.
\end{lemma}

Write $G \mapsto (A_G, T_G, c_G, \phi_G)$ for its inverse.

\subsection{Anti-affine algebraic groups} An algebraic group $G$ over a field $k$ is \emph{anti-affine} if $\Gamma(G, \cO_G) = k$. As recalled in the introduction, an anti-affine algebraic group is connected and reduced. %Moreover, for an extension $k'$ of $k$ and an algebraic group $G$, flat base change yields the equality $\Gamma(G', \cO_{G'}) =\Gamma(G, \cO_{G}) \otimes_k k'$, where $G'$ is the algebraic group over $k'$ obtained from $G$ by extending scalars. In particular, the algebraic group $G$ is anti-affine if and only if $G'$ is. 
With the notation in the preceding section:

\begin{theorem}[{\cite[Theorem 2.7]{BrionAntiAffine}}]
Let $G$ be a commutative, connected, reduced algebraic group over a perfect field $k$. Then, the algebraic group $G$ is anti-affine if and only if one of the following conditions is fulfilled:
\begin{enumerate}
\item $\characteristic k = 0$, $c_G$ is injective and $\phi_G$ surjective;
\item $\characteristic k >0$, $c_G$ is injective and the unipotent part of $G$ is trivial.
\end{enumerate}
\end{theorem}
In other words, anti-affine algebraic groups over a field of positive characteristic are anti-affine semi-abelian varieties. Moreover, if the field $k$ is an algebraic extension of a finite field, the only anti-affine semi-abelian varieties are abelian varieties (because, in this case, all $\bar{k}$-rational points of $\check{A}$ are torsion).

For an algebraic group $G$ over $k$, the $k$-algebra $\Gamma(G, \cO_G)$ is of finite type and inherits a structure of Hopf algebra, so $G_{\aff} := \Spec \Gamma(G, \cO_G)$ is an algebraic group over $k$. The canonical morphism $\pi_G \colon G \to G_{\aff}$ is faithfully flat and $\ker \pi_G$ is anti-affine (\cite[III.3.8.2 (c)]{DemazureGabriel}).

\subsection{Main results and applications} Let $K$ be a complete non-trivially valued non-Archimedean  field. For a $K$-scheme finite type $X$, let $X^\an$ be the $K$-analytic space deduced from $X$. Consider:
\begin{itemize}
\item $A$ an abelian variety over $K$,
\item $\Lambda$ a free abelian group of finite rank,
\item $F$ a finite-dimensional $K$-vector space,
\item $c \colon \Lambda \to \check{A}(K)$ a homomorphism of groups, and
\item $\phi \colon \omega_{\check{A}} \to F$ a $K$-linear map,
\end{itemize}
where $\check{A}$ is the dual abelian variety and $\omega_{\check{A}}$ the dual of $\Lie \check{A}$. Let $G$ be the reduced, connected, commutative algebraic group over $K$ corresponding to the quintuple $(A, T, F, c, \phi)$, where $T$ is the split $K$-torus with group of characters $\Lambda$.

\begin{theorem}[{\emph{infra} \Cref{Cor:AnalyticFunctionsVectExtSemiAbelian}}] \label{Thm:MainTheoremApplications} Suppose one of the following:
\begin{enumerate}
\item $\characteristic K = 0$, $c$ is injective and $\phi$ surjective;
\item $\characteristic K > 0$, $c$ is injective and $F = 0$.
\end{enumerate}
Then all $K$-analytic functions on $G^\an$ are constant.
\end{theorem}

\begin{corollary} \label{Cor:FunctionsOnAntiAffineGroupsText} Let $G$ be an anti-affine algebraic group over $K$ and $P$ a principal $G$-bundle, that is, a non-empty algebraic variety over $K$ endowed with a simply transitive action of $G$. Then any $K$-analytic function on $P^\an$ is constant.
\end{corollary}

\begin{proof} Let $\bar{K}$ be the completion of an algebraic closure of $K$. Consider the anti-affine algebraic group $\bar{G}$ over $\bar{K}$ obtained from $G$ by extending scalars to $\bar{K}$ and the principal $\bar{G}$-bundle $\bar{P}$. The choice of a point $\bar{K}$-point of $\bar{P}$ yields an isomorphism $\bar{P} \cong \bar{G}$. Since the toric part of $\bar{G}$ is split, \Cref{Thm:MainTheoremApplications} may be applied the algebraic group $\bar{G}$ over $\bar{K}$ to give the identity 
$ \Gamma(\bar{P}^\an, \cO_{\bar{P}^\an}) = \Gamma(\bar{G}^\an, \cO_{\bar{G}^\an}) = \bar{K}$. Rather generally, for a $K$-analytic space $X$, the homomorphism $\Gamma(X, \cO_{X}) \otimes_K \bar{K} \to \Gamma(\bar{X}, \cO_{\bar{X}})$ of $\bar{K}$-algebras  is injective, where $\bar{X}$ is the $\bar{K}$-analytic space deduced from $X$ by extending scalars.  The equality $\Gamma(P^\an, \cO_{P^\an}) = K$ follows.
\end{proof}

\begin{theorem} \label{Thm:FunctionsOnAnalyticGroupsText} For an algebraic group $G$ precomposing with $\pi^\an_G$ is an isomorphism
\[ \Gamma(G^\an_{\aff}, \cO_{G_{\aff}}^\an) \stackrel{\sim}{\too} \Gamma(G^\an, \cO_{G}^\an). \]
\end{theorem}

In order to prove it the following two facts are needed:

\begin{lemma} \label{Lemma:FaithfullyFlatDescentFunctionConstantOnTheFibers} Let $\pi \colon X \to S$ be a smooth surjective morphism of $K$-analytic spaces. Suppose $S$ reduced and, for $s \in S$, that all $\cH(s)$-analytic functions on $X_s = \pi^{-1}(s) $ are constant. Then precomposing with $\pi$ gives an isomorphism \[\Gamma(S, \cO_S) \stackrel{\sim}{\too} \Gamma(X, \cO_X). \]
\end{lemma}

\begin{proof} The morphism $\pi$ is faithfully flat, thus the map $\Gamma(S, \cO_S) \to \Gamma(X, \cO_X)$
is injective. Proving the statement amounts to exhibiting, for $f \in \Gamma(X, \cO_X)$, a necessarily unique $g \in \Gamma(S, \cO_S)$ such that $f = g \circ \pi$. Invoking faithfully flat descent \cite[Th\'eor\`eme 3.17]{DayliesDescent} this is equivalent to proving the equality 
\begin{equation}\label{Eq:FaithfullyFlatDescentFunctionConstantOnTheFibers} f \circ \pr_1 = f \circ \pr_2 \end{equation}
of $K$-analytic functions on $X \times_S X$, where $\pr_1, \pr_2 \colon X \times_S X \to X$ are the projections. Note that $\pi$  is properly surjective in Daylies' terminology because it is smooth and surjective, thus faithfully flat and without boundary; see \emph{op.cit.} \S 1.3. It suffices to check \eqref{Eq:FaithfullyFlatDescentFunctionConstantOnTheFibers} set-theoretically because $X \times_S X$ is reduced, the morphism $X \times_S X \to S$ being smooth and $S$ being reduced.  For $s \in S$ and check the wanted identity on the fiber $X_s$. Each $\cH(s)$-analytic function on the $\cH(s)$-analytic space $X_s$ is constant by hypothesis. Therefore $f_{\rvert X_s} \circ \pr_1 = f_{\rvert X_s} \circ \pr_2$ as desired.
\end{proof}

\begin{lemma} \label{FunctionsOnDeformation} Let $S$ be a finite $K$-analytic space whose underlying topological space is a singleton $\{ s \}$ and $\pi \colon Y \to X$ a morphism between flat $S$-analytic spaces. If $X$ is Stein in the sense of \Cref{sec:Stein} and the natural map $\Gamma(X_s, \cO_{X_s}) \to \Gamma(Y_s, \cO_{Y_s})$ is an isomorphism, then the natural map  $\Gamma(X, \cO_{X}) \to \Gamma(Y, \cO_{Y})$ is an isomorphism.
\end{lemma}

\begin{proof} The $K$-analytic space $S$ is the Banach spectrum of finite local $K$-algebra $A$. The ideal of nilpotent elements of $A$ coincides with its maximal ideal $\frm$. For $i \in \bbN$ consider the closed subspace $S_i \subseteq S$ defined by the ideal $\frm^{i + 1}$ and
\[ \pi_i \colon Y_i := Y \times_S S_i \too X_i := X \times_S S_i\]
the morphism induced by base-change. With this notation $S_0 = \{ s \}$ with its reduced structure and $\pi_0 = \pi_s$. For $i \ge 1$, consider the short exact sequence
\[ 0 \too \frm^{i} \cO_X / \frm^{i+1} \cO_X \too \cO_{X_i} \too \cO_{X_{i-1}} \too 0\]
of $\cO_X$-modules. The flatness of $X \to S$ implies $\frm^i \cO_X = \frm^i \otimes_{A} \cO_X$, thus
\[ \frm^{i} \cO_X / \frm^{i+1} \cO_X = \frm^i \otimes_{A} \cO_{X} / \frm \cO_X = (\frm^i \otimes_A K') \otimes_{K'} \cO_{X_0}\]
where $K' = A / \frm$ is the residue field at $s$. Taking global sections of the above short exact sequence yields the following exact sequence
\[ 0 \too E_i \otimes_{K'} \Gamma(X_0, \cO_{X_0}) \too \Gamma(X_i, \cO_{X_i}) \too \Gamma(X_{i-1}, \cO_{X_{i-1}})\]
where $E_i := \frm^i \otimes_A K'$. Moreover since $X$ is Stein the above exact sequence is short exact by \Cref{Lemma:ClosedOfSteinIsStein}. The same arguments give an analogous exact sequence for $Y$ (without its short-exactness)  and a commutative and exact diagram
\[ 
\begin{tikzcd}
0 \ar[r] &  E_i \otimes_{K'} \Gamma(X_0, \cO_{X_0}) \ar[r] \ar[d, "\id \otimes \pi_{0}^\ast", "\wr"'] &  \Gamma(X_i, \cO_{X_i}) \ar[r] \ar[d, "\pi_i^\ast"] & \Gamma(X_{i-1}, \cO_{X_{i-1}}) \ar[d, "\pi_{i-1}^\ast"] \ar[r] & 0 \\
0 \ar[r] &  E_i \otimes_{K'} \Gamma(Y_0, \cO_{Y_0}) \ar[r] &  \Gamma(Y_i, \cO_{Y_i}) \ar[r] & \Gamma(Y_{i-1}, \cO_{Y_{i-1}}) 
\end{tikzcd}
\]
where the leftmost vertical arrow is an isomorphism by hypothesis. Now the proof goes by induction on the smallest $N \in \bbN$ for which $\frm^{N + 1} = 0$. If $N = 0$ there is nothing to prove. If $N \ge 1$, then $\pi^\ast_{N - 1}$ is an isomorphism by induction hypothesis. The Snake Lemma applied to the above diagram with $i = N$ gives the result.
\end{proof}

\begin{proof}[{Proof of \Cref{Thm:FunctionsOnAnalyticGroupsText}}] Let $G_{\ant}$ be the kernel of $\pi_G$. Suppose first that $G$ is reduced. The morphism $\pi_G \colon G \to G_{\aff}$ is smooth, thus so is $\pi_G^\an$ by \cite[Proposition 3.5.8]{BerkovichIHES}. For $s\in G_{\aff}^\an$ the fiber $G^\an_s$ of $\pi^\an_G$ in $s$ is the analytification
\[ G_s := G \times_{G_{\ant}} \Spec\cH(s).\]
The anti-affine algebraic group $G_{\ant} \times_K \cH(s)$ acts simply transitively on $G_s$, thus \Cref{Cor:FunctionsOnAntiAffineGroupsText} implies $\Gamma(G^\an_s, \cO_{G^\an_s}) = \cH(s)$ and \Cref{Lemma:FaithfullyFlatDescentFunctionConstantOnTheFibers} gives the desired result.

Suppose that $G_{\red}$ is a smooth subgroup of $G$ (this is always the case if $K$ is perfect, but fails to be true in general) and let $H := G / G_{\red}$.  The algebraic group $S$ is finite (hence affine) and connected, thus made of a single $K$-rational point $e$. Thus the projection $p \colon G \to H$ through a surjective morphism $q \colon G_{\aff} \to H$ sitting in the following commutative and exact diagram of algebraic groups:
\[
\begin{tikzcd}
0 \ar[r] & G_{\red} \ar[r] \ar[d, "{\pi_{G_{\red}}}"] & G \ar[r, "p"] \ar[d, "\pi_G"] & H \ar[r] \ar[d, equal] & 0 \\
0 \ar[r] & (G_{\red})_{\aff} \ar[r] & G_{\aff} \ar[r, "q"] & H \ar[r] & 0.
\end{tikzcd}
\]
Note that  the equality $(G_{\red})_{\aff} = (G_{\aff})_{\red}$ holds because the subgroup $G_{\ant}$ is anti-affine, thus reduced. Now \Cref{FunctionsOnDeformation} applied with $Y = G^\an$, $X = G_\aff^\an$, $S= H^\an$ and $\pi = \pi_G$ permits to conclude. It is licit to apply it because the morphisms $p^\an$ and $q^\an$ are smooth by \emph{loc.cit.} thus flat; the algebraic group $G_{\aff}$ is affine hence $X$ is Stein (\Cref{EquivalentDefinitionSteinAlgebraicVariety}); the morphism induced between the fibers at $e$ is $\pi_{G_{\red}}^\an$, thus by reduced case all analytic functions on $G_{\red}$ come from $(G_{\red})_{\aff}$.

In the general case there is a finite extension $K'$ of $K$ for which $G'_{\red}$ is a smooth subgroup of $G' = G \times_K K'$. Quite generally, for a morphism $f \colon Y \to X$ of $K$-analytic spaces, the natural map $\Gamma(X, \cO_X) \to \Gamma(Y, \cO_Y)$ is an isomorphism if and only if the homomorphism $\Gamma(X', \cO_{X'}) \to \Gamma(Y', \cO_{Y'})$ obtained by extending scalars to $K'$ is an isomorphism. Indeed $\Gamma(X', \cO_{X'}) = \Gamma(X, \cO_X) \otimes_K K'$ without need to complete because $K'$ is a finite extension of $K$, and similarly for $Y$. This permits to reduce to the preceding case and thus concludes the proof.
\end{proof}

\begin{corollary} \label{Cor:AlgebraicGroupSteinIFFAffine} An algebraic group $G$ over $K$ is Stein if and only if it is affine.
\end{corollary}

\begin{proof} By \Cref{Thm:FunctionsOnAnalyticGroupsText} any morphism $f \colon G^\an \to \bbA^{n, \an}_K$ factors through $\pi_G^\an$. If $f$ is a closed immersion, then so is $\pi_G^\an$. Thus $\pi_G^\an$ is an isomorphism being smooth and surjective. In particular, if $G$ is Stein then it is affine; the converse is obvious.
\end{proof}

\begin{corollary} \label{Thm:NeemanExampleIsNotSteinText} Let $G$ be an algebraic group, $L$ a line bundle on $G$ and $P$ the total space of $L$ deprived of its zero section. Then the following are equivalent:
\[ \text{$G$ is affine} \iff \text{$G$ is Stein} \iff \text{$P$ is affine} \iff \text{$P$ is Stein}.  \]
\end{corollary}

%The proof of \Cref{Thm:NeemanExampleIsNotSteinText} makes use of the following (see  ):

%\begin{lemma} \label{Lemma:SteinGmPrincipalBundle} Let $X$ be a $K$-analytic space, $L$ a line bundle on $X$ and $Y$ the total space of the line bundle $L$ deprived of its zero section.  Then the $K$-analytic space $X$ is Stein if and only $Y$ is.
%\end{lemma}

\begin{proof} The morphism $P \to G$ being affine, the first condition implies the third. The equivalence of the first and the second is \Cref{Cor:AlgebraicGroupSteinIFFAffine} while the equivalence of the second and the fourth is \Cref{Prop:ToricBundleStein}.
\end{proof}

\section{Finite vector bundles on an abelian scheme}

Let $S$ be a scheme, $\alpha \colon A \to S$ an abelian scheme and $e \colon S \to A$ the zero section.

\subsection{Constant and unipotent bundles} 
Let $f \colon X \to S$ be a faithfully flat morphism of schemes such that $\cO_S \to f_\ast \cO_X$ is an isomorphism.

\begin{definition} \label{Def:UnipotentVectorBundle} A vector bundle $V$ on $X$ is 
\begin{enumerate}
\item \emph{constant (over $S$)} if it is isomorphic to $f^\ast E$ for some vector bundle $E$ on $S$;
\item \emph{unipotent} if there is an increasing filtration $0 = V_0 \subsetneq V_1 \subsetneq \cdots \subsetneq V_{d} = V$ of $V$ by $\cO_X$-submodules such that, for $i = 1, \dots, d$, the $\cO_X$-module $V_{i} / V_{i-1}$ is constant over $S$.
\end{enumerate}
\end{definition}

Note that the $\cO_X$-modules $V_i$ are automatically locally free. The direct sum, tensor product, internal homs, symmetric and exterior powers of constant (resp. unipotent) vector bundles are constant (resp. unipotent).

\begin{remark} \label{Prop:BeingAPullBack} A vector bundle $V$ on $X$ is constant if and only if $f^\ast f_\ast V \to V$ is an isomorphism.
In particular, for a constant vector bundle $V$, the $\cO_S$-module $f_\ast V$ is a vector bundle.
\end{remark}

%\begin{proof} (1) $\Rightarrow$ (2) Needless to say, it suffices to show the statement when $V$ is equal (as opposed to merely isomorphic) to $f^\ast E$ for some vector bundle $E$ on $S$. By the projection formula, the natural homomorphism 
%\[ E \otimes f_\ast \cO_X \too f_\ast f^\ast E = f_\ast V\]
%of $f_\ast \cO_X$-modules is an isomorphism. Composing it with the natural homomorphism $\cO_S \to f_\ast \cO_X$ of $\cO_S$-modules, which is an isomorphism by hypothesis, yields an isomorphism of $\cO_S$-modules $E \to f_\ast V$. Identifying $E$ with $f_\ast V$ by means of the latter, the natural homomorphism $f^\ast E \to V$ is nothing but the identity.
%
%(2) $\Rightarrow$ (1) It suffices to set $E := f_\ast V$ as soon as it is established that $f_\ast V$ is a vector bundle on $S$. The $\cO_X$-module $f^\ast f_\ast V$ is finitely presented and flat, as it is isomorphic to the vector bundle $V$. Since the morphism $f$ is faithfully flat (hence universally injective at the level of local rings \cite[\href{https://stacks.math.columbia.edu/tag/08WP}{Lemma 08WP}]{stacks-project}), the $\cO_S$-module $f_\ast V$ is finitely presented and flat (\cite[\href{https://stacks.math.columbia.edu/tag/08XD}{Theorem 08XD}]{stacks-project}), that is, a vector bundle over $S$. This also proves the last statement.
%\end{proof}

\begin{remark} \label{Prop:PushForwardOfConstantIsConstant} Let $p \colon S' \to S$ be a finite and locally free morphism of schemes and $q \colon X' = X \times_S S' \to X$ the second projection. For a vector bundle $V$ on $X'$ constant over $S'$, the vector bundle $q_\ast V$ on $X$ is constant over $S$.
\end{remark}

%\begin{proof} According to \Cref{Prop:BeingAPullBack}, it suffices to show that the homomorphism of $\cO_X$-modules \[\alpha \colon f^\ast f_\ast q_\ast V \too q_\ast V\]  is an isomorphism. The morphism $f'$ is faithfully flat because being faithful flatness is a property stable under base change. Moreover, by flatness of $p$, the homomorphism of $\cO_{S'}$-modules $p^\ast f_\ast \cO_X \to f'_\ast \cO_{X'}$ obtained by base change is an isomorphism. It follows, that the natural homomorphism $\cO_{S'} \to f'_\ast \cO_{X'}$ is an isomorphism.
%
%\Cref{Prop:BeingAPullBack}, applied to the constant vector bundle $V$ on $X'$, states that the natural homomorphism of $\cO_{X'}$-modules \[\beta \colon f'^\ast f'_\ast V \too V\] is an isomorphism. Now, by flat base change along $f$, the natural homomorphism of $\cO_{X}$-modules $f^\ast p_\ast f'_\ast V \to q_\ast f'^\ast f'_\ast V$ is an isomorphism. The equality $f^\ast p_\ast f'_\ast V  = f^\ast f_\ast q_\ast V $ implies that the map $\alpha$ is the push-forward of $\beta$ along the morphism $q$, thus an isomorphism.
%\end{proof}

\begin{lemma} \label{Prop:GlobalSectionUnipotentBundleTwistedByAnIdeal} Let $V$ be a unipotent vector bundle on $X$ and $F$ an $\cO_X$-module such that $f_\ast F$ vanishes. Then $ f_\ast (F \otimes V) = 0$.
\end{lemma}

\begin{proof} Consider a short exact sequence of $\cO_X$-modules $0 \to V' \to V \to V'' \to 0$ where $V'$ and $V''$ are vector bundles such that $f_\ast(F \otimes V')$ and $f_\ast (F \otimes V'')$ vanish. Then, the sequence of $\cO_X$-modules $0 \to F \otimes V' \to F \otimes V \to F \otimes V'' \to 0$ is exact and, by left exactness of the push-forward along $f$, the $\cO_S$-module $f_\ast V$ vanish. By induction on the rank of $V$ and by definition of a unipotent vector bundle, one reduces to the case $V = f^\ast E$ for some vector bundle $E$ on $S$. The projection formula then yields $f_\ast (F \otimes f^\ast E) = f_\ast F \otimes E = 0$ which concludes the proof.
\end{proof}

\subsection{Finite vector bundles}

\begin{definition} Let $m \ge 1$ be an integer. A vector bundle $V$ on $A$ is \emph{$m$-finite} if the vector bundle $[m]^\ast V$ is constant, \emph{i.e.} it is isomorphic to $\alpha^\ast E$ for some vector bundle $E$ on $S$. It is \emph{finite} if it is $n$-finite for some integer $n \ge 1$.
\end{definition}

\begin{remark}
The direct sum, tensor product, internal homs, symmetric and exterior powers of ($m$-)finite vector bundles are ($m$-)finite.
\end{remark}

\begin{lemma} \label{Lemma:SplittingExtensionVectorBundles} Let $E, E'$ be vector bundles on $S$ and $V$ an extension of $\alpha^\ast E$ by $\alpha^\ast E'$. Assume there is an integer $m \ge 1$ vanishing on $S$.  Then, the vector bundle $V$ is $m$-finite and, if $S$ is moreover affine, then $[m]^\ast V$ is the trivial extension. 
\end{lemma}

\begin{proof} Condition (2) in \Cref{Prop:BeingAPullBack} is local on the base, hence the scheme $S$ may be assumed to be affine. The pull-back by $[m]$ acts on the cohomology group
\[ \rH^1(A, \cHom(\alpha^\ast E, \alpha^\ast E')) = \rH^0(S, \rR^1\alpha_\ast  \cHom(\alpha^\ast E, \alpha^\ast E'))\] as the multiplication by the integer $m$. Since $m = 0$ on $S$, the vector bundle $[m]^\ast V$ is isomorphic to the trivial extension of $\alpha^\ast E$ by $\alpha^\ast E'$.
\end{proof}

\begin{lemma} \label{Lemma:TrivialPullbackUnipotentBundle} Assume there is an integer $m \ge 1$ vanishing $S$.  Then, a unipotent vector bundle $V$ on $A$ of rank $r$ is $m^r$-finite.
\end{lemma}

\begin{proof} Let $0 = V_0 \subsetneq V_1 \subsetneq \cdots \subsetneq V_{d} = V$ a filtration as in the definition of a unipotent vector bundle. The vector bundles $V_i$ are themselves unipotent, thus, arguing by induction on the rank of $V$, one reduces to the case where $V$ is an extension of constant vector bundles. \Cref{Lemma:SplittingExtensionVectorBundles} then permits to conclude.
\end{proof}

\begin{lemma} Let $0 \to V' \to V \to V'' \to 0$ be an exact sequence of vector bundles on $A$ with $V'$ and $V''$ finite. If some integer $m \ge 1$ vanishes on $S$, then $V$ is finite.
\end{lemma}

\begin{proof} Let $n \ge 1$ be an integer such that the vector bundles $[n]^\ast V'$ and $[n]^\ast V''$ are constant. Then $[n]^\ast V$ is unipotent and \Cref{Lemma:TrivialPullbackUnipotentBundle} permits to conclude.
\end{proof}

\begin{lemma} \label{Prop:PushForwardFiniteIsFinite} Let $f \colon S' \to S$ be a finite and locally free morphism of schemes and $V'$ an $m$-finite vector bundle $A' := A \times_S S'$ for some integer $m \ge 1$. Then the vector bundle $g_\ast V'$ on $A$ is $m$-finite, where $g \colon A' \to A$ is the base-change morphism.
\end{lemma}

\begin{proof} Let $[m]_A$ and $[m]_{A'}$ temporarily denote the multiplication-by-$m$ map respectively on $A$ and $A'$. By flat base change along the morphism $[m]_A$, the natural homomorphism of $\cO_A$-modules $[m]_A^\ast g_\ast V' \to g_\ast [m]_{A'}^\ast V'$ is an isomorphism. Applying \Cref{Prop:PushForwardOfConstantIsConstant} to the constant vector bundle $[m]_{A'}^\ast V'$, the vector bundle $[m]_A^\ast g_\ast V'$ is constant, thus the vector bundle $g_\ast V'$ is $m$-finite.
\end{proof}

\subsection{Universal cover and Tate group scheme} \label{Def:pAdicGroupScheme} Let $p$ be a prime number.

\begin{itemize}
\item The \emph{universal cover} of $A$ is the projective limit of the projective system $( [m] \colon A \to A)_{m \in \bbN \smallsetminus \{ 0 \}}$, the partial order on $\bbN \smallsetminus \{ 0 \}$ being divisibility. 

\item The \emph{Tate group scheme} $\rT A$ is the projective limit of the finite flat group $S$-schemes $A[m]$, for  $m \in \bbN \smallsetminus \{ 0 \}$, the transition maps $A[n] \to A[m]$ being multiplication by $\frac{n}{m}$ whenever $m$ divides $n$.

\item The \emph{$p$-adic universal cover of $A$} is the projective limit of the projective system $( [p] \colon A \to A)_{i \in \bbN}$ with the usual order on $\bbN$. 

\item The \emph{$p$-adic Tate group scheme $\rT_p A$ of $A$} is the limit of the projective system of group $S$-schemes $(A[p^i], [p] \colon A[p^{i+1}] \to A[p^i])_{i \in \bbN}$. 
\end{itemize}

The existence of such projective limits is granted by the finiteness  of the transition maps ({\cite[\href{https://stacks.math.columbia.edu/tag/01YX}{Lemma 01YX}]{stacks-project}}). The canonical projections $\rT A \to \rT_p A$ induce an isomorphism $\rT A \cong \prod_{\ell \textup{ prime}} \rT_\ell A$ of group $S$-schemes, always treated as understood in what follows. Let $(T_i, \tau_{ji} \colon T_j \to T_i)_{i \in I}$ be an inverse system of affine $S$-schemes, $T = \projlim_{i \in I} T_i$, and, for $i \in I$, $\pr_i \colon T \to T_i$ the canonical projection.

\begin{lemma}[{\cite[\href{https://stacks.math.columbia.edu/tag/01ZC}{Proposition 01ZC}]{stacks-project}}]\label{Prop:MorphismToFiniteTypeFactorsAtFiniteLevel} Suppose $S$ quasi-compact. Let $X$ be a finitely presented quasi-compact $S$-scheme and $f \colon T \to X$ a morphism of $S$-schemes. Then $f= f_i \circ \pr_i$ for some $i \in I$ and some morphism of $S$-schemes $f_i \colon T_i \to X$.
\end{lemma}

\begin{lemma}[{\cite[\href{https://stacks.math.columbia.edu/tag/01YZ}{Lemma 01YZ}]{stacks-project}}]  \label{Prop:BaseChangeProjLimit} For an $S$-scheme $S'$, $T_{S'} = \projlim_{i \in I} (T_i)_{S'}$.
\end{lemma}

In particular, the formation of the ($p$-adic) universal cover and the ($p$-adic) Tate group scheme of $A$ is compatible with base change.

\begin{lemma} \label{Prop:SESUniversalCoverAbelianScheme} Let $\nu \colon \tilde{A} \to A$ the be the projection onto the first factor. Then, the following sequence of group $S$-schemes is short exact:
\[ 0 \too \rT A \too \tilde{A} \stackrel{\nu}{\too} A \too 0.\]
\end{lemma}

\begin{proof} For an $S$-scheme $X$ let $h_X$ be the functor of points of $X$. By faithfully flat descent $h_X$ is a fppf sheaf on the category of $S$-schemes. The statement means that the sequence  $0 \to h_{\rT A} \to h_{\tilde{A}} \to h_{A} \to 0$ of fppf sheaves of abelian groups is short exact. For $m \ge 1$, the sequence $ 0 \to h_{A[m]} \to h_{A} \to h_A \to 0$ of fppf sheaves of abelian groups  is short exact, where the arrow $h_A \to h_A$ is the multiplication by~$m$. Note that the Mittag-Leffler condition is satisfied as, for $n \ge 1$, the fppf sheaves homomorphisms $n \colon h_{A} \to h_A$ and $n \colon h_{A[mn]} \to h_{A[m]}$ are surjective. Passing to the limit gives the desired result.
\end{proof}

\subsection{Finite vector bundles as representations}\label{Sec:RepresentationUnipotentModPn}  For an integer $m \ge 1$, the map $[m] \colon A \to A$ makes $A$ an fppf principal $A[m]$-bundle over itself. It is  an \emph{\'etale} principal $A[m]$-bundle as soon as $m$ is invertible on $S$. For an $m$-finite vector bundle $V$ on $A$, the vector bundle $[m]^\ast V$ is constant and naturally endowed with an $A[m]$-linearization. By \cite[Exp. I, 6.6]{SGA3} the vector bundle $\alpha_\ast [m]^\ast V$ on $S$ inherits an $A[m]$-linearization by push-forward along $\alpha$. Endow the vector $e^\ast V$ on $S$ with the $A[m]$-linearization deduced from the one of $\alpha_\ast [m]^\ast V$ via evaluation at the zero section
\[ \alpha_\ast [m]^\ast V \stackrel{\sim}{\too} e^\ast [m]^\ast V = e^\ast V.\]
Since the action on $A[m]$ on $S$ is trivial, this corresponds to a representation 
\[ \rho_{V, m} \colon A[m] \too \GL(e^\ast V).\]
The composite map
\[ \rho_V \colon \begin{tikzcd}  \rT A \ar[r, "{\pr_m}"] & A[m] \ar[r, "\rho_{V, m}"] & \GL(e^\ast V). \end{tikzcd}  \]
is independent of the chosen integer $m$, because $\rho_{V, mn} = \rho_{V, m} \circ [n]$ for $n \in \bbN \smallsetminus \{ 0 \}$. The construction of $\rho_V$ is functorial: given a homomorphism $\phi \colon V \to W$ between finite vector bundles on $A$ the homomorphism $e^\ast \phi \colon e^\ast V \to e^\ast W$ is $\rT A$-equivariant.

\begin{proposition} \label{Prop:RepresentationEssentiallyTrivialVectorBundle} Suppose that the scheme $S$ is quasi-compact. Then, the functor
\[
\left\{  \textup{finite vector bundles on $A$}  \right\} \too \left\{  \textup{representations of $\rT A$} \right\}, \; V \longmapsto \rho_V,
\]
%\[
%\begin{array}{rcl}
%\rho \colon \left\{  \textup{finite vector bundles on $A$}  \right\} &\too& \left\{  \textup{representations of $\rT A$} \right\} \\
%V &\longmapsto& \rho_V
%\end{array}
%\]
is an equivalence of categories. Furthermore,

\begin{enumerate}
\item it preserves direct sums, tensor products, internal homs, symmetric, exterior powers, and is compatible with arbitrary base change; 
\item  for an $m$-finite vector bundle $V$, the representation $\rho_V$ factors through $A[m]$;
\item  for a finite vector bundle $V$, the homomorphism $\alpha_\ast V \to e^\ast V$ given by evaluation at $e$ induces an isomorphism of $\cO_S$-modules
\[ \alpha_\ast V \cong (e^\ast V)^{\rT A}. \]
\end{enumerate}
\end{proposition}

\begin{proof} Property (2) holds by design. By \Cref{Prop:MorphismToFiniteTypeFactorsAtFiniteLevel} a representation of $\rT A$ factors through a representation of $A[m]$ for some integer $m \ge 1$. Therefore it suffices to prove that, for each $m \ge 1$, the functor
\[
\left\{  \textup{$m$-finite vector bundles on $A$}  \right\} \too \left\{  \textup{representations of $A[m]$} \right\}, \; V \longmapsto \rho_{V,m},
\]
is an equivalence of categories satisfying properties (1) and (3). This is essentially faithfully flat descent for modules on the the fppf principal $A[m]$-bundle $[m] \colon A \to A$. Indeed it states that $V \mapsto [m]^\ast V$ is an equivalence of categories between quasi-coherent $\cO_A$-modules and quasi-coherent $\cO_A$-modules endowed with an $A[m]$-linearization \cite[Theorem 4.46]{FGAExplained}. When $V$ is an $m$-finite vector bundle such a linearization can be read of from $\alpha_\ast [m]^\ast V \cong e^\ast V$ because $[m]^\ast V$ is constant. The linearization on $e^\ast V$ corresponds by definition to the representation $\rho_{V, m}$ whence the result.
\end{proof}

\begin{example}[Weil's pairing redux] Let $\check{A}$ be the dual abelian scheme and, for an integer $m \ge 1$,  $\cL_m$  the restriction to $A \times_S \check{A}[m]$ of the Poincar\'e line bundle $\cL$ on $A \times_S \check{A}$. The line bundle $\cL_m$ on the abelian scheme $A \times_S \check{A}[m]$ over $\check{A}[m]$ is $m$-torsion, therefore its pull back $([m], \id)^\ast \cL_m$ on $A \times_S \check{A}[m]$ is constant. In other words, the line bundle $\cL_m$ is $m$-finite and the corresponding representation $\rho_{\cL_m, m}$ is a morphism of group $\check{A}[m]$-schemes
\[ \rho_{\cL_m, m} \colon A[m] \times_S \check{A}[m] \too \bbG_{m, S} \times_S \check{A}[m].\]
The \emph{Weil pairing} can be interpreted as the composite map
\[ \langle -, - \rangle_{A[m]} := \pr_1 \circ \rho_{\cL_m, m} \colon A[m] \times_S \check{A}[m] \too \bbG_{m, S}, \]
where $\pr_1 \colon \bbG_{m, S} \times_S \check{A}[m] \to \bbG_{m, S}$ is the first projection. Let $A[m]^\ast$ be the Cartier dual of the finite locally free group $S$-scheme $A[m]$. The Weil pairing induces an isomorphism of group $S$-schemes  $\check{A}[m] \to A[m]^\ast$ (see \cite[Theorem 1.1]{OdaASENS}).

\end{example}

\begin{definition} \label{Def:UnipotentRepresentation} 
A representation $\rho \colon G \to \GL(E)$ is \emph{unipotent} if there is an increasing filtration $0 = E_0 \subsetneq E_1 \subsetneq \cdots \subsetneq E_d = E$ by $\cO_S$-submodules of $E$ such that, for $i = 1, \dots, d$,
\begin{itemize}
\item $E_i$ and $E_i / E_{i-1}$ are vector bundle;
\item the vector bundle $E_i$ is stable under the action of $G$;
\item the induced representation $\rho \colon G \to \GL(E_i / E_{i-1})$ is trivial.
\end{itemize}
\end{definition}

\begin{corollary} \label{Cor:UnipotentAsRepresentationTateGroupScheme} Suppose that the scheme $S$ quasi-compact and that there is an integer $m \ge 1$ vanishing on $S$. Then, the functor
\[
\left\{  \textup{unipotent vector bundles on $A$}  \right\} \to \left\{  \textup{unipotent representations of $\rT A$} \right\}, \; V \mapsto \rho_V,
\]
%\[
%\begin{tikzcd}[ row sep=0pt]
%\rho \colon \left\{ \begin{array}{c} \textup{unipotent} \\ \textup{vector bundles on $A$} \end{array} \right\} \ar[r]  & \left\{ \begin{array}{c} \textup{unipotent} \\\textup{representations of $\rT A$} \end{array} \right\}\\
%\hspace{119pt} V  \ar[r, mapsto]& \rho_V \hspace{107pt}
%\end{tikzcd}
%\]
is an equivalence of categories. Furthermore,

\begin{enumerate}
\item it preserves direct sums, tensor products, internal homs, symmetric and exterior powers, and is compatible with arbitrary base change; 
\item if $V$ is a unipotent vector bundle of rank $r$, then $\rho_V$ factors through $A[m^{r}]$;
\item if $V$ is an extension of constant vector bundles, then $\rho_V$ factors via $A[m]$.
\end{enumerate}
\end{corollary}

\begin{proof} This is an immediate consequence of \Cref{Prop:RepresentationEssentiallyTrivialVectorBundle}.
\end{proof}

Let $f \colon S' \to S$ be a  finite and locally free morphism, $m \ge 1$ an integer and $V'$ an $m$-finite vector bundle on $A' := A \times_S S'$. By \Cref{Prop:PushForwardFiniteIsFinite} the vector bundle $V:= g_\ast V'$ on $A$ is $m$-finite where $g \colon A' \to A$ is the base-change morphism. It will be useful to have an expression for $\rho_{V, m}$ in terms of $\rho_{V', m}$. For, consider the Weil restriction $\Res_{S'/S} A'[m]$ of $A'[m]$ along $f$ \cite[7.6]{NeronModels}. The `adjunction formula' \cite[Lemma 7.6.1]{NeronModels} gives a closed immersion
\[\epsilon_m \colon A[m] \too \Res_{S'/S} A'[m]\] 
which is also a morphism of group $S$-schemes. Pushing-forward along $f$ a linear automorphism of $e'^\ast V'$, where $e'$ is the zero section of $A'$, yields a morphism of group $S$-schemes
\[ j \colon \Res_{S'/S} \GL(e'^\ast V') \too \GL(e^\ast V),\]
where $f_\ast e'^\ast V' \cong e^\ast V$ by flat base change, which is a closed immersion.

%and $\alpha' \colon A' \to S'$, $g \colon A' \to A$ the natural morphisms.
%Note that  where  $e'$ is the zero section of $A'$. 

\begin{proposition} \label{Prop:RepresentationWeilRestriction} If $\Res_{S'/S} A[m]$ is flat over $S$, then
\[ \rho_{V, m} = j \circ \Res_{S'/S} \rho_{V, m} \circ \epsilon_m.\]
\end{proposition}

\begin{proof} Let $\lambda$ and $\lambda'$ be the natural linearization respectively of $[m]^\ast V$ and $[m]^\ast V'$ under the action of $A[m]$ and $A'[m]$. Then $\lambda =  (g_m \times g)_\ast \lambda'$ where $g_m \colon A'[m] \to A[m]$ is the restriction of $g$. Now the representation $\rho_{V, m}$ corresponds to the natural $A[m]$-linearization of $\alpha_\ast [m]^\ast V$ given by
\[ (\id_{A[m]} \times \alpha)_\ast \lambda = (\id_{A[m]} \times \alpha)_\ast (g_m \times g)_\ast \lambda' = (g_m \times f)_\ast (\id_{A'[m]}, \alpha')_\ast \lambda'\]
where $\alpha' \colon A' \to S'$ is the structural morphism. The $A'[m]$-linearization $(\id, \alpha')_\ast \lambda'$ of $\alpha'_\ast [m]^\ast V'$ is the one corresponding to the representation $\rho_{V', m}$. Rather generally, consider a finite flat group $S$-scheme $G$ together with a representation \[r \colon G' := G \times_S S' \too \GL(E)\] where $E$ is a vector bundle on $S'$. Provided that $\Res_{S'/S}G'$ is flat over $S$ \cite[Exp. I, 6.6]{SGA3}, pushing-forward along $h \times f$ the corresponding $G'$-linearization of $E$ is permitted and defines a representation of $G$ given by the composite map
\[ G \too \Res_{S'/S} G' \stackrel{\Res r}{\too} \Res_{S'/S} \GL(E) \too \GL(f_\ast E),\]
where the first arrow is given by adjunction, the second by functoriality of the Weil restriction and the third by pushing-forward linear automorphisms of $E$.
\end{proof}

The hypothesis of \Cref{Prop:RepresentationWeilRestriction} is fullfilled when $f$ is \'etale \cite[Proposition 7.6.5]{NeronModels}. In this paper it will be applied when $F$ is a vector bundle on $S$ and $S'$ is the first order thickening of the $S$-scheme $\bbV(F)$ along its zero section. Of course $S' \to S$ is not \'etale in this case; however, for a finitely presented $S$-scheme $X$ for which the relative tangent bundle $\cT_{X/S}$ is a vector bundle on $X$, 
\[ \Res_{S'/S} X = \bbV(\cT_{X/S} \otimes F).\]

\section{The representation of the canonical extension} \label{Sec:RepresentationCanonicalExtensionAlgebraic}

Let $S$ be a scheme, $\alpha \colon A \to S$ an abelian scheme $e \colon S \to A$ the zero section.

\subsection{The canonical extension} \label{Sec:CanonicalExtension} A \emph{homogeneous line bundle} on $A$ is a line bundle together with an isomorphism $\pr_1^\ast L \otimes \pr_2^\ast L \cong \mu^\ast L$ where $\mu$ is the law group on $A$. Homogenous line bundles are parameterized by the dual abelian scheme $\check{A}$ and the Poincar\'e bundle $\cL$ is the universal homogeneous line bundle on $A \times_S \check{A}$. Consider the Atiyah extension of the line bundle $\cL$
\[ 0 \too \Omega^1_{A \times \check{A}/A} \too \At_{A \times \check{A}/A}(\cL) \too \cO_{A \times \check{A}} \to 0 \]
where $A \times_S \check{A}$ is seen as an abelian scheme over $A$ via the first projection $p$. Pushing-forward the above short exact sequence along $p$ yields a sequence
\begin{equation} \tag{$\cU_A$}
 0 \too \alpha^\ast \omega_{\check{A}} \too \cU_A := p_\ast \At_{A \times \check{A}/A}(\cL) \too \cO_A \too 0 
\end{equation}
which is still short exact by \cite[Proposition 1.3 (2)]{RigidAnalyticUniversalVectorExtension}, where $\omega_{\check{A}}$ is the dual of the Lie algebra of $\check{A}$ and $p_\ast \Omega^1_{A \times \check{A}/A} = \alpha^\ast \omega_{\check{A}}$.

\begin{definition} The unipotent bundle $\cU_A$ is called the \emph{canonical extension} of $\cO_A$.
\end{definition}

It earns the name canonical because any other extension of $\cO_A$ by some constant vector bundle is obtained by push-out of the sequence $(\cU_A)$ \cite[Theorem 1.12]{RigidAnalyticUniversalVectorExtension}. Here it will crucial to have another description for $\cU_A$. For, let $\check{A}_1$ the first-order thickening of $\check{A}$ along its zero section, $\iota \colon A \times \check{A}_1  \to A \times_S \check{A}$ the the closed immersion and $\pi \colon A \times_S \check{A}_1 \to A$ the first projection. Then $\pi_\ast \iota^\ast \cL$ sits in the following short exact sequence
\[ 0 \too \alpha^\ast \omega_{\check{A}} \too \pi_\ast \iota^\ast \cL \too \cO_A \too 0 \]
which is seen to be canonically isomorphism to $(\cU_A)$ \cite[Remark 1.7]{RigidAnalyticUniversalVectorExtension}. The main interest of the canonical extension is that the associated affine bundle
\[ A^\natural := \bbP(\cU_A) \smallsetminus \bbP(\alpha^\ast \omega_{\check{A}})\]
parameterizes homogeneous line bundles on $\check{A}$ endowed with a connection \cite[Theorem 1.9]{RigidAnalyticUniversalVectorExtension}. The tensor product of line bundles with connections endows $A^\natural$ with the structure of a vector extension of $\check{A}$. As such it is canonically isomorphic to the universal vector extension of $\check{A}$ \cite[Theorem 1.20]{RigidAnalyticUniversalVectorExtension} and thus henceforth identified with it.

\subsection{The universal vector hull of a finite group scheme} Let $G$ be a finite and locally free commutative group $S$-scheme. Consider the  the \emph{Cartier dual} of $G$, namely the $S$-group scheme representing the functor 
\[ (\textup{Sch}/S) \too (\textup{Groups}), \qquad S'\longmapsto \Hom_{S'\textup{-groups}}(G_{S'}, \bbG_{m, S'}).\] Let  $\omega_{G^\ast}$ the dual of the $\Lie G^\ast$ and and $G_1^\ast$ the first order thickening of $G^\ast$ along its zero section $e^\ast$. Consider the morphism of $G^\ast$-schemes
\[
f  \colon G \times_S G^\ast \too  \bbG_{m, S} \times_S G^\ast, \qquad (g, \chi) \longmapsto (\chi(g), \chi),
\]
and let $f_1 \colon G \times_S G_1^\ast \to \bbG_{m, S} \times_S G^\ast_1$ be its restriction to $G_1^\ast$. The $S$-scheme $G_1^\ast$ is identified with the first-order thickening of $\bbV(\omega_{G^\ast})$ along its zero section. Therefore,
\begin{align*}
\Res_{G_1^\ast / S} (G \times_S G_1^\ast) &= \bbV(\cT_{G/S} \otimes \omega_{G^\ast}), \\
\Res_{G_1^\ast / S} (\bbG_{m, S} \times_S G_1^\ast) &= \bbV(\cT_{\Gm/S}  \otimes  \omega_{G^\ast}),
\end{align*}
where $\cT_{G/S}$ and $\cT_{\Gm/S}$ denote the tangent bundles respectively of $G$ and $\Gm$. The restriction of $f$ to the fiber at $e^\ast$ is by definition the trivial character on $G$. Thus the morphism $\Res_{G_1^\ast / S} f_1 \colon \bbV(\cT_{G/S} \otimes  \omega_{G^\ast}) \to \bbV(\cT_{\Gm/S}  \otimes  \omega_{G^\ast})$ factors through a morphism $S$-schemes
\[ \theta \colon G \too \bbV(\Lie \Gm \otimes  \omega_{G^\ast}) = \bbV(\cHom(\omega_{\Gm}, \omega_{G^\ast})). \]

\begin{definition} The \emph{universal vector hull} of $G$ is the composite morphism
\[ \theta_{G} \colon G \stackrel{\theta}{\too} \bbV(\cHom(\omega_{\Gm}, \omega_{G^\ast})) \too \bbV(\omega_{G^\ast}),  \]
where the latter arrow is the evaluation at the $1$-form $\frac{\rd z}{z}$.
\end{definition}
Concretely for an $S$-scheme $S'$ and $g \in G(S')$ the identity $\theta_G(g) = \ev_g^\ast \tfrac{\rd z}{z}$ holds where $\ev_g \colon G^\ast_{S'} \to \bbG_{m, S'}$ is the evaluation at $g$ of a character of $G$. The map $\theta_G$ owes its name to the following property (see \cite[Proposition 1.4]{MazurMessing}):

\begin{proposition} \label{Prop:UniversalPropertyUniversalVectorHull} Let $f \colon G \to \bbV(E)$ be a morphism of group $S$-schemes for a vector bundle $E$ on $S$. Then $f = \phi \circ \theta_{G}$ for a unique homomorphism $\phi \colon \omega_{G^\ast} \to E$.
\end{proposition}

\subsection{The representation of the canonical extension} \label{sec:RepresentationCanonicalExtensionScheme}  Let $m \ge 1$ be an integer vanishing on $S$.  Consider the canonical extension
\begin{equation*} \tag{$\cU_A$}
0 \too \alpha^\ast \omega_{\check{A}} \too \cU_A \too \cO_A \too 0.
\end{equation*}
By \Cref{Cor:UnipotentAsRepresentationTateGroupScheme} the vector bundle corresponds to a unipotent representation
\[ \rho_{\cU_A, m} \colon A[m] \too \GL(e^\ast \cU_{A}).\]
The aim of this section is to compute it. In order to do so, for an $S$-scheme $f \colon S' \to S$ and an $S'$-valued point of $t$ of $A[m]$, note that $\rho_{\cU_A,m}(t)$ is an automorphism the short exact sequence $f^\ast e^\ast (\cU_A)$ inducing the identity on $f^\ast \omega_A$ and $\cO_{S'}$:
\[
\begin{tikzcd}
0 \ar[r] & f^\ast \omega_A \ar[r] \ar[d, equal] & f^\ast e^\ast \cU_A \ar[r] \ar[d, "\rho_{\cU_A,m}(t)"] & \cO_{S'} \ar[d, equal] \ar[r] & 0\ \\
0 \ar[r] & f^\ast \omega_A \ar[r] & f^\ast e^\ast \cU_A \ar[r] &  \cO_{S'} \ar[r] & 0.
\end{tikzcd}
\]
 Therefore the difference $\rho_{\cU_A, m} - \id$ factors through a morphism of group $S$-schemes 
\[ \theta_{\cU_A} \colon A[m]\too \bbV(\omega_{\check{A}}).\] 
Identify the Cartier dual $A[m]^\ast$ of $A[m]$ with $\check{A}[m]$ via the Weil pairing. Since $m$ vanishes on $S$, the monomorphism $\omega_{\check{A}[m]} \to \omega_{\check{A}}$ is an isomorphism which will be implicit from now on. Keeping these identifications in mind, let
\[\theta_{A[m]} \colon A[m] \too \bbV(\omega_{\check{A}})\] 
be the universal vector hull of the finite locally free group $S$-scheme $A[m]$.

\begin{theorem} \label{Thm:RepresentationUniversalVectorExtension} With the notation above,
\[ \theta_{\cU_A} = \theta_{A[m]}. \]
\end{theorem}

\begin{proof} Let $\cL_m$ be the restriction to $A \times_S \check{A}[m]$ of the Poincar\'e bundle $\cL$. The line bundle $\cL$ on the abelian scheme $A \times_S \check{A}[m]$ on $\check{A}[m]$ is $m$-torsion and the resulting morphism of group $S$-schemes
\[ \rho := \rho_{\cL_m, m} \colon A[m] \times_S \check{A}[m] \to \bbG_{m, S} \times_S \check{A}[m]\]
is the one defining the Weil pairing. Let $\check{A}[m]_1$ be the first order thickening of $\check{A}[m]$ along its zero section and 
\[ \rho_{1} \colon A[m] \times_S \check{A}[m]_1 \too \bbG_{m,S} \times_S \check{A}[m]_1 \]
the morphism of group $\check{A}[m]_1$-schemes deduced from $\rho$ by base change along the closed immersion $\check{A}[m]_1 \to \check{A}[m]$.  Allowing for the isomorphism $ \check{A}[m] \cong A[m]^\ast$ induced by Weil's pairing and the trivialization $\cO_S \cong \omega_{\Gm}$ given $\frac{\rd z}{z}$, by definition of the universal vector hull of the finite group scheme $A[m]$,
\[ \Res_{\check{A}[m]_1/S} \rho_1 = \theta_{A[m]} \circ p,\]
where $p \colon \bbV(\cT_{A[m]/S} \otimes \omega_{\check{A}[m]}) \to A[m]$ is the projection. Now $m = 0$ on $S$, thus
\[
\omega_{\check{A}[m]} = \omega_{\check{A}}, \qquad \check{A}[m]_1 = \check{A}_1,
\]
where $\check{A}_1$ is the first order thickening of $\check{A}$ along the zero section. Let $\iota \colon \check{A}_1 \to \check{A}$ be the closed immersion and $\pi \colon \check{A}_1 \to S$ the structural morphism. As explained in \Cref{Sec:CanonicalExtension} the canonical extension $\cU_{A}$ can be identified with the vector bundle
\[ \cU_A \cong (\id_A, \pi)_\ast (\id_A,\iota)^\ast \cL.\]
Because of the equality $\check{A}_1 = \check{A}[m]_1$ the line bundle $(\id_A, \iota)^\ast \cL$ on the abelian scheme $A \times_S \check{A}_1$ over $\check{A}_1$ is the restriction of $\cL_m$ to $A \times_S \check{A}[m]_1$ and thus $m$-torsion.  Therefore with the notation of \Cref{Prop:RepresentationWeilRestriction} the representation $\rho_{\cU_A}$ is 
\[ \rho_{\cU_A} = j \circ \Res_\pi \rho_1 \circ \epsilon = j \circ \theta_{A[m]}. \]
Indeed, via the identification $\Res_{\pi} (A\times_S \check{A}_1)= \bbV(\cT_{A/S} \otimes \alpha^\ast \omega_{\check{A}})$, the map $\epsilon$ is the zero section of the vector bundle $\cT_{A/S} \otimes \alpha^\ast \omega_{\check{A}}$, thus the second equality is the very definition of the universal vector hull. It remains to give an expression for the morphism
\[ j \colon \Res_{\pi} \GL((e, \iota)^\ast \cL) \too \GL(\pi_\ast (e, \iota)^\ast \cL).\]
The canonical rigidification of the Poincar\'e bundle furnishes an isomorphism of line bundles $\cO_{\check{A}} \cong (e, \id_{\check{A}})^\ast \cL$ by mean of which the $\cO_S$-module $\pi_\ast (e, \iota)^\ast \cL$ is identified with the $\cO_S$-module $\pi_{\ast} \cO_{\check{A}_1}$. Then, for an $S$-scheme $f \colon S' \to S$ and an $t \in A[m](S')$ the morphism $j$ sends $\theta_{A[m]}(t)$ to the automorphism of the short exact sequence
\[ 0 \too f^\ast \omega_{\check{A}} \too f^\ast \pi_{\ast} \cO_{\check{A}_1} \stackrel{q}{\too} \cO_{S'} \too 0 \]
given by $\id{} + q.\theta_{A[m]}(t)$. In particular, by definition of the $\theta_{\cU_A}$,
\[ \rho_{\cU_A}(t) = \id{} + q. \theta_{\cU_A(t)} = \id{} + q.\theta_{A[m]}(t),\]
whence the result.
\end{proof}

Let $F$ be a vector bundle on $S$, $\phi \colon \omega_{\check{A}} \to F$ a homomorphism of $\cO_A$-modules and $(\cF)$ the short exact sequence of $\cO_A$-modules obtained as the push-out of the extension $(\cU_A)$ along $\phi$. The vector bundle $\cF$ is unipotent and, since the integer $m$ vanishes on $S$, corresponds to a representation $\rho_{\cF, m}$. Arguing similarly to above, the morphism $\rho_{\cF, m} - \id$ factors through a morphism of group $S$-schemes
\[ \theta_\cF \colon A[m] \too \bbV(F).\]

\begin{corollary} \label{Cor:RepresentationAlgebraicUnipotentBundle} With the notation introduced above, 
\[ \theta_\cF = \phi \circ \theta_{A[m]}.\]
\end{corollary}

\begin{proof} Since the extension $(V)$ is obtained as the push-out of $(\cU_A)$ along $\phi$, this follows immediately from \Cref{Thm:RepresentationUniversalVectorExtension}.
\end{proof}

\subsection{Universal vector hull of the Tate group scheme} \label{sec:UniversalVectorHullTateGroupScheme} Suppose there is an integer $m \ge 1$ vanishing on $S$.

\begin{definition} The \emph{universal vector hull} of the Tate group scheme $\rT A$ is the morphism of group $S$-schemes 
\[ \theta_{\rT A} = \theta_{A[n]} \circ \pr_n \colon \rT A \too A[n] \too \bbV(\omega_{\check{A}}).\]
\end{definition}

It does not depend on the integer $m$ vanishing on $S$ as $\theta_{A[mn]} = \theta_{A[m]} \circ [n]$ for any integer $n \ge 1$. The kernel of $\theta_{\rT A}$ contains $T_\ell A$ for all primes $\ell$ invertible on $S$.

\begin{proposition} \label{Prop:UniversalPropertyUniversalVectorHullTateModule} Suppose $S$ quasi-compact and that there is an integer $m \ge 1$ vanishing on $S$. Let $E$ be a vector bundle on $S$ and $f \colon \rT A \to \bbV(E)$ a morphism of group $S$-schemes.  Then $f = \phi \circ \theta_{\rT A}$ for a unique homomorphism $\phi \colon \omega_{\check{A}} \to E$.
\end{proposition}

\begin{proof} This is the combination of  \Cref{Prop:MorphismToFiniteTypeFactorsAtFiniteLevel} \Cref{Prop:UniversalPropertyUniversalVectorHull}.
\end{proof}

Multiplying by $m$ on $\omega_{\check{A}}$ is the zero map, thus there is a unique map $\sigma_m \colon A \to A^\natural$ such that the following diagram is commutative:
\[ 
\begin{tikzcd}
0 \ar[r] & A[m] \ar[d, "\sigma_{m \rvert A[m]}"] \ar[r] & A \ar[d, "\sigma_m"] \ar[r, "{[m]}"] & A \ar[r] \ar[d, equal] & 0 \\
0 \ar[r] & \bbV(\omega_{\check{A}}) \ar[r] & A^\natural \ar[r] & A \ar[r] & 0
\end{tikzcd}
\]
%\[ 
%\begin{tikzcd}
%0 \ar[r] & \bbV(\omega_{\check{A}}) \ar[d, "0"] \ar[r] & A^\natural \ar[r, "\pi"] \ar[d, "{[m]}"'] & A \ar[r] \ar[d, "{[m]}"] \ar[dl, "\sigma_m"] & 0 \\
%0 \ar[r] & \bbV(\omega_{\check{A}}) \ar[r] & A^\natural \ar[r, "\pi"] & A \ar[r] & 0
%\end{tikzcd}
%\]
%where $\pi \colon A^\natural \to A$ is the structural morphism. Let $\tilde{A}$ be the universal cover of $A$ and $\pr_n \colon \tilde{A} \to A$ the projection onto the $n$-th factor for an integer $n \ge 1$.  

%The morphism of group $S$-schemes $\sigma_{A} := \sigma_m \circ \pr_m \colon \tilde{A} \to A^\natural$ does not depend on the chosen vanishing integer $m$ and sits in the following commutative and exact diagram of group $S$-schemes:

%the first row being exact by \Cref{Prop:SESUniversalCoverAbelianScheme}. For a prime number $\ell$ invertible on $S$, the $\ell$-adic universal cover of $A$ is contained in the kernel of $\sigma_{A}$.
The following fact is proved in \cite[2.6.2]{MazurMessing} via the theory of universal vector extension of finite flat group schemes. An alternative argument goes as follows.

\begin{proposition} \label{Prop:SplittingUniversalExtensionModM} With the notation above,
\[\sigma_{m \rvert A[m]} = - \theta_{A[m]}.\]
\end{proposition}

\begin{proof} Let $ \rho:= \rho_{\cU_A} \colon A[m] \to \GL(e^\ast \cU_A)$ the representation associated with the unipotent bundle $\cU_A$ by \Cref{Cor:UnipotentAsRepresentationTateGroupScheme}.
%The vector bundle $\cU_A$ on $A$ is extension of constant vector bundles. Since the integer $m$ vanishes on $S$, it is $m$-finite, yielding an isomorphism $[m]^\ast \cU_A \cong \alpha^\ast e^\ast \cU_A$.  By definition, the natural $A[m]$-linearization of $[m]^\ast \cU_A$ corresponds via the preceding isomorphism to the representation 
The base change along the multiplication-by-$m$ of the universal vector extension $A^\natural=\bbP(\cU_{A}) \smallsetminus \bbP(\alpha^\ast \omega_{\check{A}})$ is
\[ \bbP([m]^\ast \cU_A) \smallsetminus \bbP([m]^\ast \alpha^\ast \omega_{\check{A}}) \cong  \bbA(e^\ast \cU_A) \times_S A, \]
where $\bbA(e^\ast \cU_A) = \bbP(e^\ast \cU_A) \smallsetminus \bbP(\omega_{\check{A}})$.  By construction of $\rho$  the universal vector extension $A^\natural$ is the quotient of $\bbA(e^\ast \cU_A) \times_S A$ under the action of $A[m]$ defined by
\[ t.([v], x) = ([\rho(t).v], x + t). \]
Note that $\bbP(\omega_{\check{A}}) \subseteq \bbP(e^\ast \cU_A)$ is stable under action of $A[m]$ hence so is its complement $\bbA(e^\ast \cU_A)$. The zero section of $A^\natural$ permits to identify $\bbA(e^\ast \cU_A)$ with $\bbV(\omega_{A})$. By means of this, \Cref{Thm:RepresentationUniversalVectorExtension} implies that the action of $A[m]$ is $(t, v) \mapsto v + \theta_{A[m]}(t)$.
Summing up $A^\natural = (\bbV(\omega_{\check{A}}) \times_S A) / A[m]$ where the action on the right-hand side is \[t.(v, x) = ( v + \theta_{A[m]}(t), x + t).\] 
%That is to say, the sequence of group $S$-schemes
%\[
%\begin{tikzcd}[row sep=0pt]
%0 \ar[r] & A[m] \ar[r] & \bbV(\omega_{\check{A}}) \times_S A \ar[r, "p"] & A^\natural \ar[r] & 0 \\
%& t \ar[r, mapsto] & (\theta_{A[m]}(t), t)
%\end{tikzcd}
%\]
%is short exact. 
Consider the map $\tau \colon A \to A^\natural = (\bbV(\omega_{\check{A}}) \times_S A) / A[m]$, $x \mapsto [0, x]$ where the brackets stay for the point in the quotient. The morphism obtain composing $\tau$ with $A^\natural \to A$ coincides with the quotient map $A \to A/A[m]$, that is, the multiplication by $m$. Therefore $\tau = \sigma_m$ by uniqueness of $\sigma_m$. Since $\tau(x) = [0, x] = [-\theta_{A[m]}(x), 0]$ when $x$ is $m$-torsion, this concludes the proof.
\end{proof}

\begin{corollary} \label{SectionUniversalCoverAbelianScheme} Let $\tilde{A}$ be the universal cover of $A$ and $\sigma_A := \sigma_m \circ  \pr_m \colon \tilde{A} \to A^\natural$. Then $\sigma_A$ does not depend on $m$ and the following diagram is exact and commutative:
\[
\begin{tikzcd}
0  \ar[r] & \rT A \ar[r]  \ar[d, "-\theta_{\rT A}"] & \tilde{A} \ar[r, "{\pr_1}"]  \ar[d, "\sigma_A"] & A \ar[r]  \ar[d, equal]& 0 \\
0 \ar[r] & \bbV(\omega_{\check{A}}) \ar[r] & A^\natural \ar[r] & A \ar[r] & 0
\end{tikzcd}
\]
\end{corollary}

\section{Algebraic functions on vector extensions} \label{sec:AlgebraicFunctionsVectorExtensions}

\subsection{Statement} \label{sec:AlgebraicFunctionsVectorExtensionsStatements} Let $S$ be a scheme, $\alpha \colon A \to S$ an abelian scheme, $\check{A}$ its dual and $\omega_{\check{A}}$ the dual of the Lie algebra of $\check{A}$. Consider  $F$ a vector bundle on $S$, $\phi \colon \omega_{\check{A}} \to F$ a homomorphism of $\cO_S$-modules, and $(\cF)$ the short exact sequence of $\cO_A$-modules obtained as the push-out along $\phi$ of the short exact sequence $(\cU_A)$:
\[
\begin{tikzcd}
0 \ar[r] & \alpha^\ast \omega_{\check{A}} \ar[r] \ar[d, "\alpha^\ast \phi"] & \cU_A \ar[r] \ar[d] & \cO_A \ar[r] \ar[d, equal] & 0 \\
0 \ar[r] & \alpha^\ast F \ar[r]  & \cF \ar[r, "s"]  & \cO_A \ar[r] & 0.
\end{tikzcd}
\]
Upon setting $C:= \coker \phi$, this yields to an isomorphism
\[ \alpha^\ast C= \alpha^\ast \coker \phi \cong \cF / \cU_A. \]
Let $\bbP(\cF)$ be the projective bundle of line subbundles of the vector bundle $\cF$. The projective bundle $\bbP(\alpha^\ast F)$ is a relative Cartier divisor in $\bbP(\cF)$. Its complement 
\[ \bbA(\cF) := \bbP(\cF) \smallsetminus \bbP(\alpha^\ast F)\] can be seen as the closed subscheme of the total space $\bbV(\cF)$ with equation $s = 1$.  Let $\pi \colon \bbP(\cF) \to A$ and $\eta \colon \bbA(\cF) \to S$ denote the structural morphisms. Assume that $C$  is a vector bundle on $S$ and consider the morphism of $S$-schemes
\[ \pr_{u} \colon \bbA(\cF) \too \bbV(\cF) \too \bbV(\cF / \cU_A) \cong \bbV(\alpha^\ast C) = \bbV(C) \times_S A \stackrel{\pr_1}{\too} \bbV(C). \]
Let $\bbA(\cF)_e$ be the fiber of $\bbA(\cF) \to A$ at $e$. By restriction $\pr_u$ induces a morphism of $S$-schemes $\bbA(\cF)_e \to \bbV(C)$ corresponding to an injective homomorphism of $\cO_S$-algebras $ \Sym C^\vee \to \eta_{e \ast } \cO_{\bbA(\cF)_e}$ where $\eta_e \colon \bbA(\cF)_e \to S$ is the structural morphism. In what follows $\Sym C^\vee$ is identified with its image in $\eta_{e \ast } \cO_{\bbA(\cF)_e}$.

Suppose some integer $m\ge1$ vanishes on $S$. The vector bundle $\cF$ on $A$ is an extension of constant bundles thus  by \Cref{Cor:UnipotentAsRepresentationTateGroupScheme} it corresponds to a representation $\rho_{\cF, m} \colon A[m] \to \GL(e^\ast \cF)$. Moreover \Cref{Cor:RepresentationAlgebraicUnipotentBundle} states that
%, for an $S$-scheme $f \colon S' \to S$ and  $t \in A[m](S')$, %$\rho_{\cF,m}(t)$ is an automorphism the %short exact sequence $f^\ast e^\ast (\cF)$ inducing the identity on $f^\ast F$ and $\cO_{S'}$:
%\[
%\begin{tikzcd}
%0 \ar[r] & f^\ast F \ar[r] \ar[d, equal] & f^\ast e^\ast \cF \ar[r] \ar[d, "\rho_{\cF,m}(t)"] & \cO_{S'} \ar[d, equal] \ar[r] & 0\ \\
%0 \ar[r] & f^\ast F\ar[r] & f^\ast e^\ast \cF \ar[r] &  \cO_{S'} \ar[r] & 0.
%\end{tikzcd}
%\]
the difference $\rho_{\cF, m} - \id$ factors through the morphism of group $S$-schemes 
\[ \theta_{\cF} := \phi \circ \theta_{A[m]}   \colon A[m]\too \bbV(F),\] 
where $\theta_{A[m]} \colon A[m] \to \bbV(\omega_{\check{A}})$ is the universal vector hull of $A[m]$. The group $S$-scheme $A[m]$ acts on $\bbP(e^\ast \cF)$ via $\rho_{\cF, m}$ and leaves $\bbP(F)$ stable. The action on \[\bbA(\cF)_e = \bbP(e^\ast \cF) \smallsetminus \bbP(F)\] is described as follows. For an $S$-scheme $f \colon S' \to S$ a point $x \in \bbA(\cF)_e(S')$ corresponds to a splitting of the short exact sequence $0 \to f^\ast F \to f^\ast e^\ast \cF \to \cO_{S'} \to 0$ and  $t \in A[m](S')$ acts as
\[ x \longmapsto x + \theta_\cF(t). \]
Let $(\eta_{e \ast} \cO_{\bbA(\cF)_e})^{A[m]}$ denote the $A[m]$-invariants of $\eta_{e \ast} \cO_{\bbA(\cF)_e}$ under this action.

\begin{theorem} \label{Thm:AlgebraicFunctionsOnUniversalExt} Suppose that $C$ is a vector bundle on $S$. Then,
\begin{enumerate}
\item if $V$ is a vector bundle on $A$ such that $\alpha_\ast V = 0$, then $\eta_\ast ( \pi^\ast V_{\rvert \bbA(\cF)}) = 0$;
\item the restriction map $\eta_\ast \cO_{\bbA(\cF)} \to \eta_{e \ast} \cO_{\bbA(\cF)_e}$ is injective and has image
\[
\begin{cases}
\Sym C^\vee & \text{if $S$ is flat over $\bbZ$}, \\
(\eta_{e \ast} \cO_{\bbA(\cF)_e})^{A[m]} & \text{if there is an integer $m \ge 1$ vanishing on $S$}.
\end{cases}
\]
\end{enumerate}
\end{theorem}

\Cref{Thm:AlgebraicFunctionsOnUniversalExt} (1) will be applied when $V $ is a homogeneous line bundle. Over a field such a line bundle admits a non-zero global section if and only if it is trivial. More precisely:

\begin{lemma} \label{Prop:NonTrivialHomogeneousLineBundlesHaveNoNonzeroSections} Let $L$ be a homogeneous line bundle on $A$ and $s \in \Gamma(A, L)$. Then $s$ is a trivialization if and only if there is a faithfully flat $S$-scheme $S'$ and $x \in A(S')$ such that $x^\ast s \in \Gamma(S', x^\ast L)$ is a trivialization.
\end{lemma}

\begin{proof} If $s$ is trivialization, then take $S' = S$ and $x = e$. Conversely $[-1]^\ast L \cong L^\vee$ thus $t:= [-1]^\ast s \in \Gamma(A, L^\vee)$. Moreover $f := t(s) \in \Gamma(A, \cO_A) = \Gamma(S, \cO_S)$. The hypothesis implies that $f$ is invertible when base-changing to $S'$. But $f$ must be invertible already on $S$ by faithful flatness of $S' \to S$, thus $s$ is a trivialization.
 \end{proof}

The rest of this chapter is devoted to the proof of \Cref{Thm:AlgebraicFunctionsOnUniversalExt}.

\subsection{A result on affine bundles} Reset the notation and let $f \colon X \to S$ be a quasi-compact, quasi-separated, faithfully flat morphism of schemes such that the natural homomorphism $\cO_S \to f_\ast \cO_X$ is an isomorphism. Suppose the following  commutative and exact diagram of short exact sequence of $\cO_X$-modules is given:
\[ 
\begin{tikzcd}
0 \ar[r] & f^\ast F \ar[d, "f^\ast \psi"] \ar[r] & \cF \ar[r] \ar[d, "\Psi"]  & \cO_X \ar[r] \ar[d, equal]  & 0 \\
0 \ar[r] & f^\ast F' \ar[r] & \cF' \ar[r]  & \cO_X \ar[r]  & 0
\end{tikzcd}
\]
where $F$ and $F'$ are vector bundles on $S$ and $\psi \colon F \to F'$ a homomorphism. The homomorphism $\Psi$ induces a morphism between affine bundles
\[ \pi \colon \bbA(\cF) := \bbP(\cF) \smallsetminus \bbP(f^\ast F) \too \bbA(\cF') := \bbP(\cF') \smallsetminus \bbP(f^\ast F').\]
Arguing similarly to \Cref{sec:AlgebraicFunctionsVectorExtensionsStatements} the cokernel of $\Psi$ is naturally identified with $f^\ast F''$ where $F'' = \coker \psi$. As soon as $F''$ is a vector bundle this permits to define a projection $p \colon \bbA(\cF') \to \bbV(F'')$. Let $g$ and $g'$ be  the structural morphisms of the $S$-schemes $\bbA(\cF)$ and $\bbA(\cF')$.
\begin{proposition} \label{ReductionToUnivExt} If the natural map $\cO_S \to g_\ast \cO_{\bbA(\cF)}$ is an isomorphism, then
\[ g'_\ast \cO_{\bbA(\cF')} = \Sym (F''^\vee). \]
\end{proposition}

Let $t$ be the section of $\cF^\vee$ defined by the projection $\cF \to \cO_X$. The homomorphism  $\cO_S \to g_\ast \cO_{\bbA(\cF)}$ being an isomorphism can be reformulated by saying that, for any integer $d \ge 0$, the multiplication by $t$ induces an isomorphism  \[f_\ast \Sym^d \cF^\vee \stackrel{\sim}{\too} f_\ast \Sym^{d + 1} \cF^\vee.\] \Cref{ReductionToUnivExt} follows by duality from the next lemma:

\begin{lemma} \label{Lemma:InductionSymmetricPowers}
Consider the commutative and exact diagram of $\cO_X$-modules
\[
\begin{tikzcd}
 & & 0 \ar[d] & 0 \ar[d]\\ 
& 0 \ar[r] \ar[d] & \cE'' \ar[r, "\sim"] \ar[d] & f^\ast E'' \ar[r] \ar[d] & 0 \\
0 \ar[r] & \cO_X \ar[r, "s'"] \ar[d, equal] & \cE' \ar[r] \ar[d] & f^\ast E' \ar[r] \ar[d, "f^\ast \phi"] & 0 \\
0 \ar[r] & \cO_X \ar[r, "s"] & \cE \ar[r] & f^\ast E \ar[r] & 0 
\end{tikzcd}
\]
where $E$, $E'$ and $E''$ are vector bundles over $S$ and $\phi \colon E' \to E$ a homomorphism of $\cO_S$-modules. Assume, for each integer $d \ge 0$, that the multiplication by the section $s$ induces an isomorphism of $\cO_S$-modules
\[ f_\ast \Sym^d \cE \stackrel{\sim}{\too} f_\ast \Sym^{d + 1} \cE.\]
Then, the homomorphism of $\cO_X$-algebras $\Sym \cE'' \to \Sym \cE'$ induces an isomorphism of $\cO_S$-algebras
\[ \Sym E'' \cong f_\ast \Sym \cE''   \stackrel{\sim}{\too} \injlim_{d \in \bbN} f_\ast \Sym^d \cE',\]
the transition maps in the direct limit being the multiplication by the section $s'$.
\end{lemma}

%Just to state the obvious, apply the previous Lemma with $X = A$, $f = \alpha$,
%\begin{align*}
%\cE''' &= \cF^\vee, &\cE'''' &= \cU_A^\vee, &\cE'' &= (\cF/\cU_A)^\vee,\\
%E''' &= F^\vee, &E'''' &= \omega_{\check{A}}^\vee, &E'' &= C^\vee.
%\end{align*}
%Then, the image of the restriction map $ \eta_\ast \cO_{\bbA(\cF)} \to \eta_{e \ast }\cO_{\bbA(\cF)_e}$, already known to be injective, is the $\cO_S$-algebra $\Sym C^\vee$, as to be proven.

 \begin{proof}[{Proof of \Cref{Lemma:InductionSymmetricPowers}}] For an integer $d \ge 0$, consider the following commutative and exact diagram of $\cO_X$-modules
\[
\begin{tikzcd}
0 \ar[r]  & \cK_d \ar[r] \ar[d] & \Sym^d \cE' \ar[r] \ar[d] & \Sym^d \cE \ar[r] \ar[d] & 0  \\
0 \ar[r] & f^\ast K_d \ar[r] &   \Sym^d f^\ast E' \ar[r] &   \Sym^d f^\ast E \ar[r] & 0.
\end{tikzcd}
\]
By design the vector bundles $\cK_0$ and $K_0$ vanish.

\begin{claim} For an integer $d \ge0$, 
\[ f_\ast \Sym^d \cE' = s'^d \cO_S \oplus f_\ast \cK_d, \qquad f_\ast \Sym^d \cE = s^d \cO_S.\]
\end{claim}

\begin{proof}[Proof of the Claim] The proof goes by induction on $d$. For $d = 0$ the statement is clear. Suppose $d \ge 1$ and the statement true for $d - 1$. By hypothesis, multiplication by the section $s$ induces an isomorphism
\[ s \colon f_\ast \Sym^{d-1} \cE \stackrel{\sim}{\too} f_\ast \Sym^{d} \cE,\]
yielding the equality $f_\ast \Sym^{d} \cE = s^{d}\cO_S$. In order to prove the remaining identity, consider the following exact sequence of $\cO_S$-modules
\[ 0 \too f_\ast \cK_{d} \too f_\ast \Sym^{d} \cE' \stackrel{p}{\too} f_\ast \Sym^{d} \cE = s^{d}\cO_S. \]
The composite map
\[
\begin{tikzcd}
 f_\ast \Sym^{d} \cE \ar[r, "(s^{d})^{-1}"] & \cO_S \ar[r,"s'^{d}"] & f_\ast \Sym^{d} \cE' 
\end{tikzcd}
\]
defines a section of the homomorphism $p$. The Claim follows.
\end{proof}
The vector bundles $\cK_d$ and $K_d$ admit filtrations
\begin{align*}
\Fil^i \cK_d &: && 0 &&=& \Fil^0 \cK_d &&\subset&& \Fil^1 \cK_d&& \subset&& \cdots &&\subset&& \Fil^d \cK_d &&=&& \cK_d, \\
\Fil^i K_d &: && 0 &&=& \Fil^0 K_d &&\subset&& \Fil^1 K_d &&\subset &&\cdots &&\subset&& \Fil^d K_d &&=&& K_d,
\end{align*}
whose $i$-th graded pieces are
\begin{align*}
\Fil^{i+1} \cK_d / \Fil^i \cK_d = \Sym^i  \cE \otimes \Sym^{d - i} \cE'', \\
\Fil^{i+1} K_d / \Fil^i K_d = \Sym^i  E \otimes \Sym^{d - i} E''.
\end{align*}

\begin{claim} The image of the homomorphism $f_\ast \cK_d \to K_d$ is $\Fil^1 K_d = \Sym^d E''$.
\end{claim}

\begin{proof}[Proof of the Claim] The homomorphism $\Sym^d \cE' \to \Sym^d f^\ast E'$ preserves the preceding  filtrations on $\cK_d$ and $K_d$, meaning that, for $i = 0, \dots, d$, the following diagram of $\cO_X$-modules is commutative and exact:
\[
\begin{tikzcd}[column sep=20pt]
0 \ar[r] & \Fil^i \cK_d \ar[r] \ar[d] & \Fil^{i+1} \cK_d \ar[r] \ar[d] &\Sym^i  \cE \otimes \Sym^{d - i} \cE'' \ar[r] \ar[d, "\alpha"]  &  0 \\
0 \ar[r] & f^\ast \Fil^i K_d \ar[r] & f^\ast \Fil^{i+1} K_d \ar[r] & \Sym^i f^\ast E \otimes  \Sym^{d - i} f^\ast E''  \ar[r] &  0,
\end{tikzcd}
\]
where $\alpha$ is induced by $\cE \to f^\ast E$ and $\cE'' \to f^\ast E''$ via tensor constructions.  By the projection formula, pushing-forward the homomorphism $\alpha$ along $f$ defines a homomorphism of $\cO_S$-modules
\[ f_\ast \alpha \colon (f_\ast \Sym^i  \cE) \otimes \Sym^{d - i} E'' \too \Sym^i E \otimes  \Sym^{d - i} E'', \]
where the isomorphism $\Sym^{d-i} \cE'' \cong f^\ast \Sym^{d-i} E''$ has been taking into account. Note that the map $f_\ast \alpha$ is obtained by taking the tensor product with $\Sym^{d - i} E''$ of the homomorphism $\beta \colon f_\ast \Sym^i  \cE \to  \Sym^i E$ induced by $\cE \to f^\ast E$. Now, for an integer $i \ge 1$, pushing forward along $f$ the short exact sequence of $\cO_X$-modules 
\[ 0 \too \Sym^{i-1} \cE \stackrel{s}{\too} \Sym^i \cE \too f^\ast \Sym^i E \too 0,\]
yields the following exact sequence of $\cO_S$-modules:
\[ 0 \too f_\ast \Sym^{i-1} \cE \stackrel{s}{\too} f_\ast \Sym^i \cE \stackrel{\beta}{\too} \Sym^i E. \]
By assumption, multiplication by $s$ is an isomorphism, thus $\beta = 0$. In particular, the homomorphism $f_\ast \alpha$ vanishes for $i \ge 1$ and the Claim follows.
\end{proof}
For $d \ge 0$ we have $\ker(f_\ast \cK_{d + 1} \to K_{d + 1})  = f_\ast \cK_d$ because the square
\[ 
\begin{tikzcd}
\Sym^d \cE' \ar[d, "s'"] \ar[r] & \Sym^d E' \ar[d, "0"] \\ 
\Sym^{d+1} \cE' \ar[r] & \Sym^{d+1} E'
\end{tikzcd}
\]
is commutative. Therefore, the second Claim implies that the sequence
\[ 0 \too f_\ast \cK_d \stackrel{s'}{\too} f_\ast \cK_{d+1} \too \Sym^{d+1} E'' \too 0 \]
is short exact. Coupling this with the first Claim yields
\[ \injlim_{d \in \bbN} f_\ast \Sym^d \cE' = \injlim_{d \in \bbN} (s'^d \cO_S \oplus f_\ast \cK_d) \cong \cO_S \oplus \bigoplus_{d \ge 1} \Sym^d E'' = \Sym E'',
\]
where the transition maps are the multiplication by $s'$.
 \end{proof}

\subsection{Proof of \Cref{Thm:AlgebraicFunctionsOnUniversalExt}} Go back to the notation in \Cref{sec:AlgebraicFunctionsVectorExtensionsStatements}. Let $j$ denote the open immersion of $\bbA(\cF)$ in $\bbP(\cF)$. For $d \in \bbZ$, consider the line bundle $\cO_\cF(d)$ on $\bbP(\cF)$. The projection $s \colon \cF \to \cO_A$ in the datum of $(\cF)$ as an extension can be seen as a global section of $\cF^\vee$. The latter being identified with $\pi_\ast \cO_{\cF}(1)$, the section $s$ defines also a global section on the line bundle $\cO_{\cF}(1)$.  The hyperplane bundle $\bbP(\alpha^\ast F)$ in $\bbP(\cF)$ is the zero locus of the section $s$. Therefore, for a vector bundle $E$ on $\bbP(\cF)$, 
\[ j_\ast j^\ast E = \injlim_{d \in \bbN} E(d),\]
where $E(d) := E \otimes \cO_\cF(d)$ for $d \in \bbZ$ and the transition maps are the multiplication by $s$. For a vector bundle $V$ on $A$ and $d \in \bbN$, the projection formula reads
\[ \pi_\ast (\pi^\ast V(d)) = V \otimes \pi_\ast \cO_{\cF}(d) = V \otimes \Sym^d \cF^\vee\]
because of the equality $\pi_\ast \cO_{\cF}(d) = \Sym^d \cF^\vee$. Thus, 
\[ \pi_\ast j_\ast j^\ast \pi^\ast V = \pi_\ast (\injlim_{d \in \bbN} \pi^\ast V(d)) = \injlim_{d \in \bbN} (V \otimes \pi_\ast \cO_\cF(d)). \]
Note that pushing-forward along $\pi$ and the direct limit commute because of the quasi-compactness and quasi-separation of $\pi$ (see \cite[\href{https://stacks.math.columbia.edu/tag/009F}{Lemma 009F}]{stacks-project}). For the same reason, pushing-forward along $\alpha$ finally gives
\begin{equation} \label{Eq:GlobalSectionsAffineBundleAlgebraic} \eta_\ast j^\ast \pi^\ast V = \injlim_{d \in \bbN} \alpha_\ast (V \otimes \Sym^d \cF^\vee).
\end{equation}

\begin{proof}[Proof of (1)] Let $i \colon \bbP(\alpha^\ast F) \to \bbP(\cF)$ the closed immersion. For $d \in \bbN$ consider the short exact sequence of $\cO_\bbP(\cF)$-modules
\[ 0 \too \cO_\cF(d) \stackrel{s}{\too} \cO_{\cF}(d+1) \too i_\ast i^\ast \cO_{\cF}(d+1) \too 0.\]
With the conventions above, its tensor product with the vector bundle $\pi^\ast V$ reads
\[ 0 \too \pi^\ast V (d) \stackrel{s}{\too} \pi^\ast V (d+1) \too i_\ast i^\ast  \pi^\ast V (d + 1) \too 0.\]
The push-forward of the previous short exact sequence along the morphism $\pi$ gives the following exact sequence of $\cO_{A}$-modules:
\[ 0 \too V \otimes \Sym^d \cF^\vee \stackrel{s}{\too} V \otimes \Sym^{d+1} \cF^\vee \too V \otimes \alpha^\ast \Sym^{d+1} F^\vee,\]
because of the projection formula and the equality, for an integer $d \in \bbN$, 
\[ \pi_\ast i_\ast i^\ast \cO_\cF(d+1) = \alpha^\ast \Sym^{d+1} F^\vee. \]
Consider the exact sequence of $\cO_S$-modules
\[ 0 \too \alpha_\ast (V \otimes \Sym^d \cF^\vee) \stackrel{s}{\too} \alpha_\ast (V \otimes \Sym^{d+1} \cF^\vee) \too \alpha_\ast V \otimes \Sym^{d+1} F^\vee.\]
obtained by pushing forward the previous one along the morphism $\alpha$. By hypothesis, the $\cO_S$-module $\alpha_\ast V$ vanishes, thus multiplication by $s$ induces an isomorphism
\[ \alpha_\ast (V \otimes \Sym^d \cF^\vee) \stackrel{\sim}{\too} \alpha_\ast (V \otimes \Sym^{d+1} \cF^\vee). \]
In particular, according to \eqref{Eq:GlobalSectionsAffineBundleAlgebraic},
\[ \eta_\ast  j^\ast \pi^\ast V = \injlim_{d \in \bbN} \alpha_\ast (V \otimes \Sym^d \cF^\vee) = \alpha_\ast V = 0, \]
which concludes the proof.
\end{proof}
Let $I$ be the sheaf of ideals of $\cO_A$ vanishing on the section $e$. Consider the short exact sequence of $\cO_A$-modules
\[ 0 \too I \otimes \Sym^d \cF^\vee \too \Sym^d \cF^\vee \stackrel{\rho_d}{\too} e_\ast e^\ast \Sym^d \cF^\vee \too 0.\]
The $\cO_S$-module $\alpha_\ast I$ vanishes, thus \Cref{Prop:GlobalSectionUnipotentBundleTwistedByAnIdeal} implies $\alpha_\ast (I \otimes \Sym^d \cF^\vee) = 0$
because the vector bundle $\Sym^d \cF^\vee$ is unipotent. In particular, pushing forward the previous short exact sequence along the morphism $\alpha$ yields an injective homomorphism of $\cO_S$-modules
\[ r_d \colon \alpha_\ast \Sym^d \cF^\vee \too e^\ast \Sym^d \cF^\vee. \]
Passing to the limit and applying \eqref{Eq:GlobalSectionsAffineBundleAlgebraic} with $V = \cO_A$ shows that the restriction map
\[ r \colon \eta_\ast \cO_{\bbA(\cF)} = \injlim_{d \in \bbN} \alpha_\ast \Sym^d \cF^\vee \too \injlim_{d \in \bbN} e^\ast \Sym^d \cF^\vee = \eta_{e \ast} \cO_{\bbA(\cF)_e}
\]
is injective.

\begin{proof}[Proof of (2) when $S$ is flat over $\bbZ$] Suppose first that the vector bundle $F$ coincides with $\omega_A$ and the homomorphism $\phi$ is the identity, so that $\cF = \cU_A$ and  $\bbA(\cF) = A^\natural$ is the universal vector extension.  Then, the statement is \cite[Corollary 2.4]{ColemanFourier}. In the locally Noetherian and characteristic $0$ case (to which one may reduce to for our ultimate goal), this can be found also in \cite[Th\'eor\`eme 2.4.1]{LaumonFourierGen}. In both cases, this is where the hypothesis of flatness over $\bbZ$ is used. The general case the follows by \Cref{ReductionToUnivExt}.
 \end{proof}

\begin{proof}[Proof of (2) when an integer $m \ge 1$ vanishes on $S$]  The vector bundle $\cF$ on $A$ is extension of constant vector bundles, thus it is $m$-finite by \Cref{Lemma:SplittingExtensionVectorBundles}. It follows that, for an integer $d \ge 0$, the vector bundle $\Sym^d \cF^\vee$ is $m$-finite. According to \Cref{Prop:RepresentationEssentiallyTrivialVectorBundle}, evaluation on $e$ yields an isomorphism of $\cO_S$-modules
\[ \alpha_\ast \Sym^d \cF^\vee \stackrel{\sim}{\too} (\Sym^d e^\ast \cF^\vee)^{A[m]}.\]
The section $e^\ast s$ of the vector bundle $e^\ast \cF^\vee$ is invariant under the action of $A[m]$, for it is the restriction of the global section $s$ of $\cF^\vee$. Therefore, multiplication by $e^\ast s$ is an $A[m]$-equivariant homomorphism. Passing to the limit and applying \eqref{Eq:GlobalSectionsAffineBundleAlgebraic} with $V = \cO_A$ shows that the restriction homomorphism defines an isomorphism
\[ \eta_\ast \cO_{\bbA(\cF)} = \injlim_{d \in \bbN} \alpha_\ast \Sym^d \cF^\vee \stackrel{\sim}{\too} (\injlim_{d \in \bbN}  \Sym^d e^\ast \cF^\vee)^{A[m]} = (\eta_{e \ast} \cO_{\bbA(\cF)_e})^{A[m]},\]
whence the statement.
\end{proof}

\section{Almost finite vector bundles and the Hodge-Tate decomposition} \label{Sec:AlmostFiniteAndHodgeTateDecomposition}

Let $K$ be a non-trivially valued complete non-Archimedean field, $R$ the ring of integers of $K$, $\varpi \in R$ a non-zero topologically nilpotent element and, for a formal $R$-scheme $X$ and an integer $n \ge 1$, $X_n$ the closed subscheme defined by the equation $\varpi^n = 0$.  Let $S$ be a formal $R$-scheme, $\alpha \colon A \to S$ a formal abelian scheme and $e \colon S \to A$ the zero section. The dual formal abelian scheme $\check{A}$ is the formal scheme \[\check{A} = \injlim_{n \ge 1} \check{A}_n\] where, for an integer $n \ge 1$, $\check{A}_n$ is the dual abelian scheme of $A_n$. %The arguments in \Cref{Sec:UniversalExtensionAbelianScheme} can then be used to define the canonical extension $(\cU_A)$ of $\cO_A$ and the universal vector extension $A^\natural$ of $A$.

\subsection{The Tate formal group scheme}  \label{Sec:FormalProjectiveLimits} Let $(X^{(i)}, f^{(i)} \colon X^{(i+1)} \to X^{(i)})_{i \in \bbN}$ be a projective system of formal $S$-schemes with affine transition maps. For an integer $n \ge 1$, let $(X^{(i)}_n, f^{(i)}_n)_{i \in \bbN}$ be the projective system of $S_n$-schemes obtained by base change to $S_n$. According to \cite[\href{https://stacks.math.columbia.edu/tag/01YX}{Lemma 01YX}]{stacks-project} and owing to the affineness of the transition maps, the projective limit \[X_n := \projlim_{i \in \bbN} X^{(i)}_n\]
exists and represents the functor associating to a $S_n$-scheme $T$ the projective limit of the sets $X^{(i)}_n(T)$. By \Cref{Prop:BaseChangeProjLimit}, the formation of the projective limit is compatible to base change, thus, for an integer $0 \le m \le n$, 
\[ X_m = X_n \times_{S_n} S_m.\]
The very definition of a morphism between formal schemes and the universal property of projective limits of schemes yield:

\begin{proposition} The functor $S' \mapsto \projlim_{i \in \bbN} X^{(i)}(S')$ on formal $S$-schemes $S'$
is representable by the formal $S$-scheme $\injlim_{n \in \bbN} X_n$, called the \emph{projective limit} of the formal schemes $X^{(i)}$.
\end{proposition}

The definition of universal cover, Tate group scheme and, for a prime $p$, the $p$-adic universal cover and $p$-adic Tate scheme is the analogous of Definition \Cref{Def:pAdicGroupScheme}. 

\subsection{Hodge-Tate decomposition} \label{Sec:HodgeTateDecomposition} Suppose that $K$ is a valued extension of $\bbQ_p$ for some prime $p$ and take $\varpi = p$. 

\begin{definition}
For each integer $n \ge 1$ let $\theta_{\rT A_n} \colon \rT A_n \to \bbV(\omega_{\check{A}_n})$ be the universal vector hull of the Tate group scheme of $A_n$. The morphism of group formal $S$-schemes $\theta_{\rT A} \colon \rT A \to \bbV(\omega_{\check{A}})$ defined by the morphisms $\theta_{\rT A_n}$ factors through a morphism of group formal $S$-schemes
\[ \theta_{\rT_p A} \colon \rT_p A \too \bbV(\omega_{\check{A}}) \]
called the \emph{universal vector hull} of the $p$-adic Tate group scheme.
\end{definition}

%\begin{proposition} Suppose $S$ quasi-compact. Let $E$ be a vector bundle on $S$ and $f \colon \rT_p A \to \bbV(E)$ a morphism of formal group $S$-schemes. Then, there is a unique homomorphism of $\cO_S$-modules $\phi \colon \omega_{\check{A}} \to E$ such that $f = \phi \circ \theta_{\rT_p A}$.
%\end{proposition}

\begin{remark} \label{Rmk:NonReducedStructureTateGroupScheme} Suppose $K$ algebraically closed and $S= \Spf(R)$. Then, for any free $R$-module $E$ of finite rank $r$,
\[ \Hom_{R\textup{-groups}}(\rT_p A, \bbV(E)) \neq \Hom_{\bbZ_p\textup{-mod}}(\rT_p A(R), E) \]
simply because these two $\bbZ_p$-modules have different ranks, respectively $gr$ and $2gr$. Indeed the left-hand side coincides with $\Hom_{R\textup{-mod}}(\omega_{\check{A}}, E)$ by the universal property of the universal vector hull (the analogous of \Cref{Prop:UniversalPropertyUniversalVectorHullTateModule}), while on the right-hand side $\rT_p A(R)$ is the usual $p$-adic Tate module.
\end{remark}

Let $\tilde{A}$ be the $p$-adic universal cover of the formal abelian scheme $A$. For $n \ge 1$ the morphism $\sigma_{A_n}$ introduced in \Cref{sec:UniversalVectorHullTateGroupScheme} factors through a morphism of $S_n$-schemes $\sigma_{A_n, p} \colon \tilde{A}_n \to A^\natural_n$ because each prime different from $p$ is invertible of $S_n$. By \Cref{SectionUniversalCoverAbelianScheme} the morphism of formal $S$-schemes $\sigma_{A, p} \colon \tilde{A} \to A^\natural$ defined in this way fits in the following commutative and exact diagram of group formal $S$-schemes
\[ 
\begin{tikzcd}
0 \ar[r] & \rT_p A \ar[r] \ar[d, "-\theta_{\rT_p A}"] & \tilde{A} \ar[r, "\pr_1"] \ar[d, "\sigma_{A, p}"] & A \ar[r] \ar[d, equal] & 0 \\ 
0 \ar[r]  & \bbV(\omega_{\check{A}}) \ar[r]  & A^\natural \ar[r]  & A \ar[r]  & 0
\end{tikzcd}
\]

\begin{theorem} \label{Thm:HodgeTateDecompositionGoodReduction} Suppose $S = \Spf(R)$ and $K$ algebraically closed. Then, the following $K$-linear map is surjective:
\[ \theta_{\rT_p A}\colon \rT_pA (R) \otimes_{\bbZ_p} K \too \omega_{\check{A}} \otimes_R K\]
\end{theorem}

\begin{proof} The map $\alpha_G$ for the $p$-divisible group $G = A[p^\infty]$ in \cite[Definition 3.2.3]{ScholzePDivisibleGroups} agrees with $-\theta_{\rT_p A}$ by \Cref{Prop:LimitFormulaUniversalVectorHull} below and \emph{op.cit.} Lemma 3.5.1. The statement then is Proposition 5.1.6 in \emph{op.cit.}.
\end{proof}

\begin{lemma} \label{Prop:LimitFormulaUniversalVectorHull} For $x = (x_i)_{i \in \bbN} \in \tilde{A}(S)$ and $x_i^\natural \in A^\natural(S)$ mapping to $x_i \in A(S)$,
\[ \sigma_{A, p}(x) = \lim_{i \to \infty} p^i x^\natural_i.\]
\end{lemma}

\begin{proof} For $n \ge 1$ the base change to $S_n$ of the morphism $\sigma_{A, p}$ can be written for any integer $i \ge n$ as $\sigma_{n, i} \circ \pr_{n, i}$ where $\pr_{n, i} \colon \tilde{A}_{n} \to A_n$ is the projection onto the $i$-th factor and $\sigma_{n, i} \colon A_n \to A^\natural_n$ is the unique for which its composition with $A^\natural_n \to A_n$ is $[p^i] \colon A_n \to A_n$. In particular, the chain of congruences
\[ \sigma_{A, p}(x) \equiv \sigma_{n, i}(x_i) \equiv p^i x_i^\natural \pmod{p^n} \]
holds. The statement follows by letting $n \to \infty$.
\end{proof}

\subsection{Almost finite vector bundles as representations} The notions of $m$-finite, constant and unipotent vector bundle on $A$, and that of a (unipotent) representation is the obvious analogue of the scheme-theoretic one.

\begin{definition} A vector bundle $V$ on $A$ is \emph{almost finite} if $V_{\rvert A_n}$ is finite for each~$n$.
\end{definition}

When $K$ is the completion of an algebraic closure of $\bbQ_p$ for some prime $p$ and $S = \Spf(R)$ a line bundle  on $A$ is almost finite if and only if it is homogeneous. For an almost finite bundle $V$ on $A$, the representations $\rho_{V, n} \colon \rT A_n \to \GL(e^\ast V)_n$, for each integer $n \ge 1$, define a representation 
\[ \rho_V \colon \rT A \too \GL(e^\ast V).\]

\begin{proposition}\label{Prop:RepresentationUnipotentFormalVectorBundle}  Suppose $S$ quasi-compact. Then, the functor
\[
\{ \textup{almost finite vector bundles on $A$}\} \too \{\textup{representations of $\rT A$} \}, \quad V \longmapsto \rho_V
\]
is an equivalence of categories, preserves direct sums, tensor products, internal homs, symmetric and exterior powers, and is compatible with arbitrary base change. Moreover, for an almost finite vector bundle $V$, the evaluation $\alpha_\ast V \to e^\ast V$ at the zero section induces an isomorphism
\[ \alpha_\ast V \cong (e^\ast V)^{\rT A}. \]
\end{proposition}

\begin{proof} Apply \Cref{Prop:RepresentationEssentiallyTrivialVectorBundle} to each successive thickening $A_n$.
\end{proof}

%Similarly to \Cref{Cor:UnipotentAsRepresentationTateGroupScheme}, the preceding statement 
In the case of unipotent bundles this reads as:

\begin{corollary} Suppose $S$ quasi-compact and $K$ a valued extension of $\bbQ_p$. Then,
\[
\{ \textup{unipotent vector bundles on $A$} \} \to \{ \textup{unipotent representations of $\rT_p A$} \}, V \mapsto \rho_V
\]
%\[
%\begin{tikzcd}[ row sep=0pt]
%\rho \colon \left\{ \begin{array}{c} \textup{unipotent} \\ \textup{vector bundles on $A$} \end{array} \right\} \ar[r]  & \left\{ \begin{array}{c} \textup{unipotent} \\\textup{representations of $\rT_p A$} \end{array} \right\}\\
%\hspace{119pt} V  \ar[r, mapsto]& \rho_V \hspace{107pt}
%\end{tikzcd}
%\]
is an equivalence of categories, preserves direct sums, tensor products, internal homs, symmetric, exterior powers, and is compatible with arbitrary base change. Moreover, for a unipotent vector bundle $V$ on $A$, the evaluation $\alpha_\ast V \to e^\ast V$ at the zero section   induces an isomorphism 
\[ \alpha_\ast V \cong (e^\ast V)^{\rT_p A}. \]
\end{corollary}

\subsection{Representation of the canonical extension} \label{Sec:RepresentationCanonicalExtension} Let $K$ be a valued extension of $\bbQ_p$, $F$ a vector bundle on $S$, $\phi \colon \omega_{\check{A}} \to F$ be a homomorphism of $\cO_A$-modules and $(\cF)$ the short exact sequence of $\cO_A$-modules obtained as the push-out of the extension $(\cU_A)$ along $\phi$. The vector bundle $\cF$ is unipotent and corresponds to a representation $\rho_{\cF}$ of the $p$-adic Tate group scheme $\rT_p A$. Moreover, for a formal $S$-scheme $f \colon S' \to S$ and an $S'$-valued point $t$ of $\rT_p A$, the automorphism $\rho_\cF(t)$ fits into the following commutative and exact diagram of $\cO_{S'}$-modules:
\[
\begin{tikzcd}
0 \ar[r] & f^\ast F \ar[r] \ar[d, equal] & f^\ast \cF \ar[r] \ar[d, "\rho_{\cF}(t)"] & \cO_{S'} \ar[r] \ar[d, equal] & 0\ \\
0 \ar[r] & f^\ast F \ar[r] & f^\ast \cF \ar[r] & \cO_{S'} \ar[r] & 0.
\end{tikzcd}
\]
Therefore, the morphism $\rho_{\cF} - \id$ factors through a morphism of group $S$-schemes
\[ \theta_\cF \colon \rT_p A \too \bbV(F).\]
\Cref{Cor:RepresentationAlgebraicUnipotentBundle} gives immediately:

\begin{corollary} \label{Cor:RepresentationFormalUnipotentBundle} With the notation introduced above, 
\[ \theta_\cF = \phi \circ \theta_{\rT_p A}.\]
\end{corollary}

\section{Formal functions on the universal vector extension}

\subsection{Statements} \label{sec:FormalFunctionsVectorExtensions} Let $K$ be a complete non-trivially valued non-Archimedean field and $R$ its ring of integers. Let $\alpha \colon A \to \Spf(R)$ be a formal abelian scheme, $F$ a free $R$-module of finite rank, $\phi \colon \omega_{\check{A}} \to F$ a homomorphism of $R$-modules, and $(\cF)$ the short exact sequence of $\cO_A$-modules obtained as the push-out of $(\cU_A)$ along $\phi$. Let $\bbP(\cF)$ be the projective bundle of line subbundles of the vector bundle $\cF$. The projective bundle $\bbP(\alpha^\ast F)$ is a relative Cartier divisor in $\bbP(\cF)$. Let 
\[\bbA(\cF) := \bbP(\cF) \smallsetminus \bbP(\alpha^\ast F)\] be its complement, suppose that cokernel $C$ of $\phi$ is a torsion-free (hence locally free of finite rank) and define a morphism of formal $R$-schemes 
\[ \pr_{u} \colon \bbA(\cF) \too \bbV(\cF) \too \bbV(\cF / \cU_A) \cong \bbV(\alpha^\ast C) = \bbV(C) \times_S A \stackrel{\pr_1}{\too} \bbV(C). \]
by arguing as in \Cref{sec:AlgebraicFunctionsVectorExtensions}. Let $\pi \colon \bbP(\cF) \to A$ denote the structural morphism.

\begin{theorem} \label{Thm:FormalFunctionsOnVectorExt} Suppose that the $R$-module $C$ is torsion-free. Then, with the notation introduced above,
\begin{enumerate}
  \item for a vector bundle $V$ on $A$ such that $\Gamma(A, V) = 0$, 
\[ \Gamma(\bbA(\cF), \pi^\ast V) = 0;\]
\item if $\characteristic K = 0$, then precomposing with $\pr_{u}$ induces an isomorphism
\[ \Gamma(\bbV(C), \cO_{\bbV(C)}) \stackrel{\sim}{\too} \Gamma(\bbA(\cF), \cO_{\bbA(\cF)}).\]
\end{enumerate}
\end{theorem}

Let $\Lambda$ be a finitely generated free abelian group and $T$ the formal split torus with group of characters $\Lambda$. Let $c \colon \Lambda \to \check{A}(R)$ be a group homomorphism and 
\[ 0 \too T \too G \too A \too 0 \]
the extension of the formal abelian scheme $A$ by the formal split torus $T$ determined by the homomorphism $c$. Namely, the formal scheme $G$ represents the functor associating to a formal $R$-scheme $S$ the datum of an $S$-valued point $x$ of $A$ and, for each $\chi \in \Lambda$, a trivialization $\langle x, \chi \rangle$ of the line bundle 
\[ \cL_{x, \chi} := (x, c(\chi))^\ast \cL\]
on $S$ where $\cL$ is the Poincar\'e bundle on $A \times_R \check{A}$. Moreover, the trivializations $\langle x, \chi \rangle$ undergo the compatibility
\[  \langle x, \chi \rangle \otimes \langle x, \chi \rangle = \langle x, \chi + \chi' \rangle,\]
meant to be understood via the isomorphism $\cL_{x, \chi} \otimes \cL_{x, \chi'} \cong \cL_{x, \chi + \chi'}$. Let $T_0$ be the formal split torus with group of characters $\ker c$. Consider the morphism \[\pr_t \colon G \too T_0\] defined as follows. For $\chi \in \ker c$, the line bundle $(\id, c(\chi))^\ast \cL$ on $A$ is by definition the trivial one. Therefore, for a formal $R$-scheme $S$ and an $S$-valued point $x$ of $G$, the trivialization $\langle x, \chi \rangle$ is an invertible function on $A \times_R S$. Because of the equality $\Gamma(A \times_R S, \cO_{A \times_R S}) = \Gamma(S, \cO_S)$ the homomorphism $\ker c \to \Gamma(S, \cO_S)^\times$, $\chi \mapsto \langle x, \chi \rangle$, defines an $S$-point $\pr_t(x)$.

\begin{corollary} \label{Cor:FormalFunctionsVectExtSemiAbelian} If $\characteristic(K) = 0$, then precomposing  with $(\pr_{t}, \pr_{u})$ induces an isomorphism of $R$-algebras
\[ \Gamma(T_0 \times_R \bbV(C), \cO_{T_0 \times \bbV(C)}) \stackrel{\sim}{\too} \Gamma(G \times_A \bbA(\cF), \cO_{G \times_A \bbA(\cF)}).\]
\end{corollary}

The proof of these two results will occupy the remained of this chapter.

\subsection{Notation} Fix notation as follows. Let $\varpi \in R$ be a non-zero topologically nilpotent element, $X$ a formal $R$-scheme and $W$ an $\cO_X$-module. For an integer $n \ge $1, consider the ring $R_n:= R / (\varpi^n)$, the closed subscheme $X_n$ of $X$ given by the equation $\varpi^n = 0$ and the $\cO_{X_n}$-module $W_{n} := W_{\rvert X_n}$. For a morphism $f \colon X \to Y$ of formal $R$-schemes, let $f_n \colon X_n \to Y_n$ denote its reduction modulo $\varpi^n$. Also, for a free $R$-module $E$ of finite rank, set
\[\hSym E := \projlim_{n \ge1} \Sym E / (\varpi^n).\]
With this notation,
\[ \Gamma(\bbV(E), \cO_{\bbV(E)}) = \hSym E^\vee. \]
For $d \in \bbN$ consider the line bundle $\cO_{\cF}(d)$ on the projective bundle $\bbP(\cF)$, so that
\[ \pi_\ast \cO_{\cF}(d) = \Sym^d \cF^\vee.\]
In particular, the projection $s \colon \cF \to \cO_A$ in the datum of the extension $(\cF)$ defines a global section of the line bundle $\cO_{\cF}(1)$. Upon writing $j$ for the open immersion of $\bbA(\cF)$ in $\bbP(\cF)$, for a vector bundle $L$ on $\bbP(\cF)$ the identity
\[ j_\ast j^\ast L = \projlim_{n \ge 1} j_{n \ast} j^\ast_n L_n = \projlim_{n \ge 1} \injlim_{d \in \bbN} L(d)_n\]
holds, where, for an integer $d\ge 0$, $L(d) := L \otimes \cO_\cF(d)$, and the transition maps in the direct limit are the multiplication by the section $s$. For a vector bundle $V$ on $A$ and $L = \pi^\ast V$, the projection formula gives
\begin{equation} \label{Eq:FormalSectionsVectorBundleVectorExtension} \pi_\ast j_\ast j^\ast \pi^\ast V = \projlim_{n \ge 1} \injlim_{d \in \bbN} \pi_{n \ast} \pi_n^\ast V(d)_n = \projlim_{n \ge 1} \injlim_{d \in \bbN} V_n \otimes \Sym^{d} \cF^\vee_n.
\end{equation}

\subsection{Proof of \Cref{Thm:FormalFunctionsOnVectorExt} (1)} It would tempting to apply  \Cref{Thm:AlgebraicFunctionsOnUniversalExt} (1) to the restriction $V_n$ of $V$ to the $n$-th thickening $A_n$. Alas, even though the vector bundle $V$ on $A$ has no non-zero sections, the vector bundle $V_n$ on $A_n$ might very well have, preventing the hypothesis of \emph{loc.cit.} to be fulfilled. Nonetheless, the Mittag-Leffler condition in this case says that, given an integer $n \ge 1$, there is $r \ge 0$ such that, for each integer $n' \ge n + r$, the image of restriction map
\[ \Gamma(A_{n'}, V_{n'}) \too \Gamma(A_n, V_n)\]
vanishes. The argument will ultimately rely on this fact. In order to get started with, for integers $n \ge 1$ and $d \ge 0$, set
\begin{align*}
\cW^d &:= V \otimes \Sym^d \cF^\vee, & W^d &:= V \otimes \alpha^\ast \Sym^d F^\vee.
\end{align*}
Given integers $n' \ge n$ and $d' \ge d$, consider the homomorphism of $R$-modules
\[ s^{d, d'}_{n, n'} \colon \Gamma(A_n, \cW^d) / \Gamma(A_{n'}, \cW^d) \too \Gamma(A_n, \cW^{d'}) / \Gamma(A_{n'}, \cW^{d'}).
 \]
induced by the multiplication by the section $s^{d' - d}$.

\begin{claim} The homomorphism $s^{d, d'}_{n, n'}$ is injective.
\end{claim}

\begin{proof}[Proof of the Claim] Since the composite of injective maps is injective, it suffices to prove for $d' = d + 1$. Let $i \colon \bbP(\alpha^\ast F) \to \bbP(\cF)$ be the closed immersion and consider the short exact sequence of $\cO_{\bbP(\cF)}$-modules
\[ 0 \too \pi^\ast V(d) \stackrel{s}{\too} \pi^\ast V(d + 1) \too i_\ast i^\ast \pi^\ast V(d + 1) \too 0.\]
The short exact sequence of $\cO_A$-modules 
\[ 0 \too V\otimes \Sym^d \cF^\vee \stackrel{s}{\too} V \otimes \Sym^{d + 1} \cF^\vee \too V \otimes \Sym^{d + 1} \alpha^\ast F^\vee \too 0,\]
obtained by pushing forward the previous one along $\pi$, yields the following commutative and exact diagram of $R$-modules:
\[
\begin{tikzcd}
0 \ar[r] & \displaystyle \frac{\Gamma(A_{n'}, \cW^d)}{\Gamma(A_{n'}, \varpi^n \cW^d)} \ar[d, hook] \ar[r, "s"] & \displaystyle
\frac{\Gamma(A_{n'}, \cW^{d+1})}{\Gamma(A_{n'}, \varpi^n \cW^{d+1})} \ar[d, hook] \ar[r] & \displaystyle \frac{\Gamma(A_{n'}, W^{d+1})}{\Gamma(A_{n'}, \varpi^n W^{d+1})} \ar[d, hook] \\
0 \ar[r] & \Gamma(A_n, \cW^d) \ar[r, "s"]&  \Gamma(A_n, \cW^{d+1}) \ar[r]   & \Gamma(A_n, W^{d+1}).
\end{tikzcd} 
\]
Note that the vertical arrows are injective because, for an $R$-scheme $X$ and an $\cO_X$-module $E$, the sequence of $R$-modules
\[ 0 \too \Gamma(X, \varpi^n E) \too \Gamma(X, E) \too \Gamma(X, E/\varpi^n E),\]
deduced by taking global sections from the short exact sequence of $\cO_X$-modules
\[ 0 \too \varpi^n E \too E \too E / \varpi^n E \too 0,\]
is exact (by left-exactness of taking global sections).  Now, the homomorphism $s_{n, n'}^{d, d+1}$ is nothing but the one induced on the cokernels of the two leftmost vertical arrows of the above diagram. The Snake Lemma hence permits to conclude.
\end{proof}

For convenience set $\cW^{-1} = 0$. Let $f$ be a section of $\pi^\ast V$ on $\bbA(\cF)$ and $n \ge 1$ an integer. Then, according to \eqref{Eq:FormalSectionsVectorBundleVectorExtension}, the reduction of $f$ modulo $\varpi^n$ comes from a section of the vector bundle $\cW^d_n = V_n \otimes \Sym^d \cF_n^\vee$
on $A_n$ for some integer $d \ge -1$. Let $d_n(f)$ the smallest such integer and 
\[ f_n \in \Gamma(A_n, \cW^{d_n(f)})\]
the (necessarily unique) section mapping to the residue class of $f$ modulo $\varpi^n$. Note that the equality $d_n(f) = -1$ holds if and only if $f = 0$.

\begin{claim} For each integer $n \ge 1$, the section $f_n$ vanishes identically.
\end{claim}

\begin{proof}[Proof of the Claim] For brevity, write $d_n$ instead of $d_n(f)$. Arguing by contradiction, suppose that the section $f_n$ is non-zero, thus $d_n \ge 0$. Unwinding the definitions, the section 
\[ g_{n, n'} := s^{d_{n'} - d_n} f_n \in \Gamma(A_n, \cW^{d_{n'}}) \]
is the residue class modulo $\varpi^n$ of the section $f_{n'}$. In particular, the image of the section $g_{n, n'}$ in the quotient
\[\Gamma(A_n, \cW^{d_{n'}}) / \Gamma(A_{n'}, \cW^{d_{n'}})\] 
vanishes. On other hand, according to the first Claim, the map
\[ s^{d_n, d_{n'}}_{n, n'} \colon \Gamma(A_n, \cW^{d_{n}}) / \Gamma(A_{n'}, \cW^{d_{n}}) \too \Gamma(A_n, \cW^{d_{n'}}) / \Gamma(A_{n'}, \cW^{d_{n'}})\]
is injective, meaning that, there is a section $h_{n'}$ of $\cW^{d_n}$ on $A_{n'}$ whose residue class modulo $\varpi^n$ is $f_n$. The Mittag-Leffler condition for the projective system $(\Gamma(A_n, V))_{n \ge 1}$ leads to a contradiction. Before seeing why, remark that the afore-mentioned projective system satisfies the Mittag-Leffler condition thanks to \cite[Corollaire 4.1.7]{EGAIII1} when $K$ is discretely valued (see alternatively \cite[8.2.7, p. 191]{FGAExplained}), and \cite[Corollaire 2.11.7]{AbbesEGR} in the general case (see also \cite[Proposition 11.3.3]{FujiwaraKato}).  Now, the vector bundle $V$ has no non-zero global sections on $A$ by hypothesis. The Mittag-Leffler condition ensures that there is an integer $n' \ge n$ such that the residue classes modulo $\varpi^n$ of sections of $V_{n'}$ vanish altogether. In particular, the image of the section $h_{n'}$ in 
\[  \Gamma(A_{n}, W^{d_n})  = \Gamma(A_n, V) \otimes \Sym^{d_n} F_n^\vee \]
vanishes. By definition, the residue class of $h_{n'}$ modulo $\varpi^n$ is $f_{n}$, thus the image of the section $f_n$ in $\Gamma(A_{n}, W^{d_n})$ is $0$. This means that there is a global section $f'_n$ of $\cW^{d_n - 1}$ on $A_n$ such that $f_n = s f'_{n}$, contradicting the minimality of $d_n$.
\end{proof}

The claim implies that, for each $n \ge 1$, the section $f_n$ vanishes thus $f = 0$. \qed

\subsection{Proof of \Cref{Thm:FormalFunctionsOnVectorExt} (2)}  Suppose first that the residue field of $R$ is of characteristic $0$. Then, for an integer $n \ge 1$, the ring $R_n$ is flat over $\bbZ$. This permits to apply  \Cref{Thm:AlgebraicFunctionsOnUniversalExt} (2) and obtain $\Gamma(\bbA(\cF_n), \cO_{\bbA(\cF_n)}) = \Gamma(\bbV(C_n), \cO_{\bbV(C_n)})$. Passing to the projective yields $\Gamma(\bbA(\cF), \cO_{\bbA(\cF)}) = \Gamma(\bbV(C), \cO_{\bbV(C)})$ concluding the proof in this case. From now on, assume that the characteristic of the residue field of $R$ is a prime number $p$. In other words, the field $K$, being of characteristic $0$ by hypothesis, is a complete valued extension of $\bbQ_p$. As usual, in this case, the topological nilpotent element $\varpi$ is understood to be $p$. By compatibility of global sections to extension of scalars, there is no harm in assuming the field $K$ to be algebraically closed. The vector bundle $\cF$ is unipotent, therefore it corresponds to a unipotent representation $\rho_{\cF} \colon \rT_p A \to \GL(e^\ast \cF)$ of the $p$-adic Tate group scheme of $A$. With the notation of \Cref{Sec:RepresentationCanonicalExtension}, the difference $\rho_{\cF} - \id$ factors through a morphism of formal $S$-schemes in groups
\[ \theta_\cF \colon \rT_p A \too \bbV(F)\] 
and \Cref{Cor:RepresentationFormalUnipotentBundle} states the equality $\theta_\cF = \phi \circ \theta_{\rT_p A}$
where $\theta_{\rT_p A} \colon \rT_p A \to \bbV(\omega_{\check{A}})$ is the universal vector hull of the $p$-adic Tate group scheme of $A$. The $p$-adic Tate group scheme $\rT_p A$ acts on the projective space $\bbP(e^\ast \cF)$ via the representation $\rho_{\cF}$ leaving the hyperplane $\bbP(F)$ stable. On the complementary affine space \[\bbA(\cF)_e = \bbP(e^\ast \cF) \smallsetminus \bbP(F)\] the formal scheme in groups $\rT_p A$ acts, for a formal $R$-scheme $f \colon S \to \Spf(R)$ and $S$-valued points $x$ of $\bbA(\cF)_e$ and $t$ ot $\rT_p A$, by
\[ x \longmapsto x + \theta_\cF(t) \]
where the point $x$ is seen as a splitting of the short exact sequence $f^\ast e^\ast (\cF)$. By applying \Cref{Thm:AlgebraicFunctionsOnUniversalExt} (2) to each infinitesimal thickening of $\bbA(\cF)$ and then passing to the projective limit, the restriction homomorphism 
\[ \Gamma(\bbA(\cF), \cO_{\bbA(\cF)}) \too \Gamma(\bbA(\cF)_e, \cO_{\bbA(\cF)_e})\]
is seen to be injective and its image be the $R$-subalgebra of $\rT_p A$-invariant functions:
\[ \Gamma(\bbA(\cF), \cO_{\bbA(\cF)}) = \Gamma(\bbA(\cF)_e, \cO_{\bbA(\cF)_e})^{\rT_p A}. \]
In a concrete fashion, the choice of a splitting of the short exact sequence of $R$-modules $e^\ast (\cF)$ permits to identify $\bbA(\cF)_e$ with $\bbV(F)$. By doing so, $\rT_p A$ acts by translation on $\bbA(\cF)_e \cong \bbV(F)$ and formal functions on $\bbA(\cF)$ can be seen as the set formal functions on $\bbV(F)$ invariant under translation by $\rT_p A$.  Now, the Tate formal group scheme $\rT_p A$ is seen to be nonreduced if the formal abelian scheme $A$ is non-trivial. Therefore, invariance under the formal group scheme $\rT_p A$ and invariance under the group $\rT_p A(R)$ are not \emph{a priori} equivalent notions. Nonetheless:
\begin{claim} For $f \in \Gamma(\bbA(\cF)_e, \cO_{\bbA(\cF)_e})$ the following are equivalent:
\begin{enumerate}
\item $f$ is $\rT_p A$-invariant;
\item $f$ is $\rT_p A(R)$-invariant;
\item $f$ comes from $\bbV(C)$ by pull-back along the map $\pr_u \colon \bbA(\cF)_e \to \bbV(C)$.
\end{enumerate}
\end{claim}
The implication (1) $\Rightarrow$ (2) is clear. For (3) $\Rightarrow$ (1) note that the composite map
\[ \rT_p A \stackrel{\theta_{\cF}}{\too} \bbV(F) \too \bbV(C) \]
vanishes identically because of the equality $\theta_\cF = \phi \circ \theta_{\rT_p A}$ and of the definition of $C$ as the cokernel of $\phi$. This shows the invariance under $\rT_p A$ of the functions on $\bbA(\cF)_e$ obtained precomposing a formal function on $\bbV(C)$ with the morphism $\pr_u$.  In order to show (2) $\Rightarrow$ (3), notice that the Hodge-Tate decomposition (as stated in \Cref{Thm:HodgeTateDecompositionGoodReduction}) implies that the $K$-linear map
\[ \theta_{\cF} \otimes \id_K \colon \rT_p A(R)  \otimes_{\bbZ_p} K \too F \otimes_R K,\] 
obtained by extending $K$-linearly $\theta_\cF$, surjects onto the image of $\phi \otimes \id_K$.  The key point of the whole story, displaying unequivocally the difference with the complex case, is that lattices in $p$-adic vector spaces accumulate to $0$, forcing lattice-invariant analytic functions to be constant. More formally:

\begin{lemma} \label{Lemma:FormalFunctionsInvariantUnderLattice} Suppose $K$ is a complete valued field extension of $\bbQ_p$. Let $B$ a $p$-adically complete flat $R$-algebra, $E$ a free $R$-module of finite rank and $\Gamma$ a finitely generated $\bbZ_p$-submodule of $E$ such that the $K$-linear map
\[ \Gamma \otimes_{\bbZ_p} K \too E \otimes_R K\]
is an isomorphism. Then, an element $f$ of $B \hotimes_R \hSym E^\vee$ belongs to $R$ if and only if it is $\Gamma$-invariant, that is, for $\gamma \in \Gamma$,
\[ f(x + \gamma) = f(x).\]
\end{lemma}

Before proving the Lemma, let us see how it permits to the conclude the argument. Since the $R$-module $C$ is torsion-free (thus free of finite rank), the $R$-module $F$ decomposes as $E \oplus C$, where $E$ is the image of $\phi$. This leads to an isomorphism of $R$-algebras
\[ \hSym F^\vee \cong \hSym E^\vee {\hotimes}_R B,\]
where $B = \hSym C^\vee$.  The image $\Gamma$ of the map $\theta_{\cF} \colon \rT_p A (R) \to F$ generates the $K$-vector space $E \otimes_R K$, therefore the previous lemma implies that the only functions of $\bbA(\cF)_e \cong \bbV(F)$ invariant under $\rT_p A(R)$ are those coming from $\bbV(C)$. \qed

\begin{proof}[{Proof of \Cref{Lemma:FormalFunctionsInvariantUnderLattice}}] By induction one reduces to the case where $E$ has rank $1$. Moreover, the choice of a generator of $E$ induces an isomorphism of $p$-adically complete $B$-algebras
\[ \hSym E^\vee {\hotimes}_R B \cong B \{ x \} := \projlim_{n \in \bbN} B/ p^n B [x].\]
Now, expand $f$ in powers series $f(x) = \sum_{k = 0}^\infty f_k x^k$, with $f_k \in B$ such that $f_k \to 0$ as $k \to \infty$.  For $\gamma \in \Gamma$ non-zero, the Taylor expansion of $f(x+\gamma)$ is 
\[
\sum_{k = 0}^\infty f_k (x + \gamma)^k = \sum_{i = 0}^\infty  \sum_{k \ge i} \binom{k}{i} f_k \gamma^{k-i} x^i.
\]
Comparing it with the Taylor expansion of $f$, yields, for an integer $i \ge 0$,
\[ f_i =  \sum_{k \ge i} \binom{k}{i} f_k \gamma^{k - i} = f_i + \sum_{k \ge i + 1} \binom{k}{i} f_k \gamma^{k - i}.\]
Canceling $f_i$ on both sides of the previous equality and dividing by $\gamma$ the result (a licit operation owing to the flatness of $B$) gives, for an integer $i \ge 1$,
\[ i f_i = - \sum_{k \ge i+1} \binom{k}{i - 1} f_k \gamma^{k - i}.\]
The right-hand side of the previous equality tends to $0$ as soon as $\gamma$ does, implying the equality $i f_i = 0$ and, again by flatness of $B$, the vanishing of $f_i$ for $i \ge 1$.
\end{proof}

\subsection{Proof of \Cref{Cor:FormalFunctionsVectExtSemiAbelian}} Consider the following cartesian square:
\begin{center}
\begin{tikzcd}
X:= G \times_A \bbA(\cF) \ar[r, "f"] \ar[d, "g"]  & \bbA(\cF) \ar[d, "p"] \\
G \ar[r, "q"]& A
\end{tikzcd}
\end{center}
where the map $p$ is the composition of the morphism $\pi \colon \bbP(\cF) \to A$ with the open immersion $\bbA(\cF) \to \bbP(\cF)$. For an integer $n \ge 1$, the morphism $p_n \colon \bbA(\cF)_n \to A_n$ is flat, thus, by flat base change,
\[ f_{n \ast} \cO_{X_n} = f_{n \ast}  g_{n}^\ast \cO_{G_n} =  p_{n}^\ast q_{n \ast} \cO_{G_n}. \]
For $\chi \in \Lambda$, let $\cL_\chi$ denote the line bundle $(\id_A, c(\chi))^\ast \cL$ on $A$, where $\cL$ is the Poincar\'e bundle on $A \times_R \check{A}$. By definition, the affine $A_n$-scheme $G_n$ is the relative spectrum of the quasi-coherent $\cO_{A_n}$-algebra
\[ q_{n \ast} \cO_{G_n} = \bigoplus_{\chi \in \Lambda}  \cL_{\chi, n}^\vee.\]
Combining the previous equalities yields
\[ p_{n}^\ast q_{n \ast} \cO_{G_n} =  \bigoplus_{\chi \in \Lambda}  p_n^\ast \cL_{\chi, n}^\vee. \]
The $R_n$-scheme $\bbA(\cF_n)$ is quasi-separated and quasi-compact, thus taking global sections commute with direct limits and, in particular,
\[ \Gamma(X_n,  \cO_{X}) = \bigoplus_{\chi \in \Lambda} \Gamma(\bbA(\cF_n), p^\ast \cL_{\chi}^\vee).\]
Taking to projective limit for all integers $n \ge 1$ gives the equality
\[ \Gamma(X,  \cO_{X}) = \hbigoplus_{\chi \in \Lambda} \Gamma(\bbA(\cF), p^\ast \cL_{\chi}^\vee).\]
According to \Cref{Prop:NonTrivialHomogeneousLineBundlesHaveNoNonzeroSections}, or rather its formal scheme-theoretical analogue, a non-trivial homogeneous line bundle has no non-zero sections. Therefore, \Cref{Thm:FormalFunctionsOnVectorExt} implies
\[ 
\Gamma(\bbA(\cF), p^\ast \cL_{\chi}^\vee) =
\begin{cases}
\Gamma(\bbV(C), \cO_{\bbV(C)}) & \textup{if $c(\chi) =0$}, \\
0 & \textup{otherwise},
\end{cases}
\]
which, combined with the previous expression, gives
\[ \Gamma(X,  \cO_{X}) = \hbigoplus_{\chi \in \Ker c} \Gamma(\bbV(C), \cO_{\bbV(C)}) \chi = \Gamma(\bbV(C), \cO_{\bbV(C)}){\hotimes}_R \Gamma(T_0, \cO_{T_0}),\]
concluding the proof. \qed

\section{Rigid analytic functions on the universal vector extension} 

Let $K$ be a complete non-trivially valued non-Archimedean field.

\subsection{Statements} Let $A$ denote an abeloid variety, that is, a proper smooth connected $K$-analytic group, with dual $\check{A}$ (\cite[Corollary 7.6.5]{LutkebohmertCurves}). The arguments in \Cref{Sec:CanonicalExtension} (which rely on \cite[1.2]{RigidAnalyticUniversalVectorExtension}) are transliterated right-away in rigid-analytic jargon, permitting to the define the canonical extension $\cU_A$ of $\cO_A$. Let $F$ be a $K$-vector space, $\phi \colon \omega_{\check{A}} \to F$ a $K$-linear map, and $(\cF)$ the short exact sequence of $\cO_A$-modules obtained as the push-out of the extension $(\cU_A)$ along $\phi$. Consider the affine bundle $\pi \colon \bbA(\cF) \to A$ associated with $(\cF)$ and the projection $\pr_u \colon \bbA(\cF) \to \bbV(C)$
where $C$ is the cokernel of $\phi$. The base being a point here, note that $C$ is a finite-dimensional $K$-vector space, so no further assumption is needed.

\begin{theorem} \label{Thm:AnalyticFunctionsOnVectorExt} Suppose that the field $K$ is of characteristic $0$. With the notation introduced above,
\begin{enumerate}
  \item if $L$ is a non-trivial homogeneous line bundle on $A$, then
\[ \Gamma(\bbA(\cF), \pi^\ast L) = 0;\]

\item precomposing with $\pr_{u}$ induces an isomorphism
\[ \Gamma(\bbV(C), \cO_{\bbV(C)}) \stackrel{\sim}{\too} \Gamma(\bbA(\cF), \cO_{\bbA(\cF)}).\]
\end{enumerate}
\end{theorem}

Let $\Lambda$ be a finitely generated free abelian group, $T$ the split $K$-torus with group of characters $\Lambda$, $c \colon \Lambda \to \check{A}(K)$ a group homomorphism, $G$ the extension of $A$ by $T$ determined by $c$, and $T_0$ the split $K$-torus with group of characters $\ker c$. Arguing as in \Cref{sec:FormalFunctionsVectorExtensions} furnishes a morphism of $K$-analytic spaces $\pr_{t} \colon G \to T_0$, whence the map
\[ (\pr_{t}, \pr_{u}) \colon G \times_A \bbA(\cF) \too T_0 \times \bbV(C). \]

\begin{corollary} \label{Cor:AnalyticFunctionsVectExtSemiAbelian} With the notation above,
\begin{enumerate}
\item precomposing with $\pr_{t}$ induces an isomorphism of $K$-algebras
\[ \Gamma(T_0, \cO_{T_0}) \stackrel{\sim}{\too} \Gamma(G, \cO_{G});\]
\item if the characteristic of the field $K$ is $0$, then precomposition with $(\pr_{t}, \pr_{u})$ gives rise to an isomorphism of $K$-algebras
\[ \Gamma(T_0 \times \bbV(C), \cO_{T_0 \times \bbV(C)}) \stackrel{\sim}{\too} \Gamma(G \times_A \bbA(\cF), \cO_{G \times_A \bbA(\cF)}).\]
\end{enumerate}
\end{corollary}

The rest of this chapter deals with the proof of these two statements.

\subsection{Proof in the good reduction case} Suppose that the abeloid variety $A$ has good reduction, that is, it is Raynaud's generic fiber of a formal abelian scheme $\cA$ over $R$. Let $\check{\cA}$ denote the dual formal abelian scheme of $\cA$. Let $F_0$ be a finitely generated $R$-submodule of the $K$-vector space $F$ containing the image of $\omega_{\check{\cA}}$ under the linear map $\phi$ and generating $F$ as $K$-vector space. The $R$-module $F_0$ can be chosen in such a way that the induced homomorphism of $R$-modules $\phi \colon \omega_{\check{\cA}} \to F_0$ has a torsion-free cokernel $C_0$. Fix a non-zero topologically nilpotent element $\varpi$ of $R$. For an integer $n \ge 1$ define inductively a finitely generated $R$-submodule $F_n$ of $F$ containing $F_{n-1}$ and such that the cokernel of the map $\phi \colon \omega_{\check{\cA}} \to F_n$ is \[C_n :=\varpi^{-n} C_0.\] Let $(\cF_n)$ denote the short exact sequence of $\cO_{\cA}$-modules obtained as the push-out of the canonical extension $(\cU_\cA)$ on $\cA$ along the map $\phi \colon \omega_{\check{\cA}} \to F_n$.  Raynaud's generic fiber $X_n$ of the affine bundle $\bbA(\cF_n)$ over $\cA$ is a compact analytic domain of the $K$-analytic space $\bbA(\cF)$. The compact analytic domains $X_n$ are  nested inside each other (increasingly, with respect to $n$) and their union $X$ is an analytic domain of $\bbA(\cF)$. For a vector bundle $V$ on $\bbA(\cF)$, the restriction map
\[ \Gamma(\bbA(\cF), V) \too \Gamma(X, V) = \projlim_{n \in \bbN} \Gamma(X_n, V), \]
is injective, for the analytic domain $X$ is dense for the (analytic) Zariski topology on $\bbA(\cF)$. And last, the formal $R$-scheme $\bbA(\cF_n)$ being quasi-compact, for a coherent $\cO_{\bbA(\cF_n)}$-module $\cW$, the identity
\[ \Gamma(X_n, W) = \Gamma(\bbA(\cF_n), \cW) \otimes_R K \]
holds, where $W$ is the coherent $\cO_{X_n}$-module induced by $\cW$.

\begin{proof}[{Proof of \Cref{Thm:AnalyticFunctionsOnVectorExt} when $A$ has good reduction}] (1) The properness of the dual formal abelian scheme $\check{\cA}$ implies that the line bundle $L$ comes from a line bundle $\cL$ on $\cA$. Moreover, the hypothesis of $L$ being non-trivial implies that $L$ has no non-zero global sections (again by the rigid-analytic version of \Cref{Prop:NonTrivialHomogeneousLineBundlesHaveNoNonzeroSections}), thus the same holds for $\cL$. Now, \Cref{Thm:FormalFunctionsOnVectorExt} (1) states the vanishing of global  sections of the line bundle $\pi^\ast \cL$ on $\bbA(\cF_n)$, thus
\[\Gamma(X_n, \pi^\ast L) = \Gamma(\bbA(\cF_n), \pi^\ast \cL) \otimes_R K = 0.\]
It follows that any global section of the line bundle $\pi^\ast L$ on $X$ vanishes identically and in turn the same thing holds true over $\bbA(\cF)$, the restriction map \[\Gamma(\bbA(\cF), \pi^\ast L) \too \Gamma(X, \pi^\ast L)\] being injective. (Note that this argument does not make use of characteristic $0$.)

(2) Since the field $K$ is of characteristic $0$, by \Cref{Thm:FormalFunctionsOnVectorExt} (1) precomposing with $\pr_u$ gives an isomorphism of $K$-algebras
\[ \Gamma(\bbV(C_n)_\eta, \cO_{\bbV(C)}) \stackrel{\sim}{\too} \Gamma(X_n, \cO_{\bbA(\cF)}),\]
where $\bbV(C_n)_\eta$ is Raynaud's generic fiber of the formal scheme $\bbV(C_n)$. The family of compact analytic domains $\bbV(C_n)_\eta$ is an increasing exhaustion and a G-cover of the $K$-analytic space $\bbV(C)$, yielding the identity
\[ \Gamma(\bbV(C), \cO_{\bbV(C)})= \projlim_{n \in \bbN} \Gamma(\bbV(C_n)_\eta, \cO_{\bbV(C)}). \]
Passing the above isomorphism to the limit, it follows that the composite homomorphism of $K$-algebras
\[ \Gamma(\bbV(C), \cO_{\bbV(C)}) \stackrel{\iota}{\too} \Gamma(\bbA(\cF), \cO_{\bbA(\cF)}) \stackrel{\rho}{\too} \Gamma(X, \cO_{\bbA(\cF)})\]
is bijective, where $\iota$ is the precomposing a function on $\bbV(C)$ with $\pr_u$ and $\rho$ is the restriction to $X$. Both the maps $\iota$ and $\rho$ are injective: the first because the morphism $\pr_u$ is smooth and surjective, the second because $X$ is Zariski-dense for the analytic topology on $\bbA(\cF)$. Their composite $\rho \circ \iota$ being bijective forces $\iota$ (and $\rho$) to be so, as desired.
\end{proof}

\subsection{Reminder on Fourier expansion} The main tool to deal with bad reduction will be Fourier transform for section of coherent sheaves. Reset momentarily the notation and let $X$ be a separated $K$-analytic space, $\Lambda$ a free abelian group of finite rank $n$, $\Lambda \to \Pic(X)$, $\chi \mapsto L_\chi$ a group homomorphism and $P$ the $X$-analytic space whose points $s$ with values on a $X$-analytic space $f \colon X' \to X$ is the set of data, for $\chi \in \Lambda$, of a trivialization $\langle \chi, s \rangle$ of the line bundle $f^\ast L_\chi$. Moreover, the trivializations above satisfy, for $\chi, \chi' \in \Lambda$,\footnote{\label{Footnote:AbuseNotationToricBundle}Several abuses of notation have been perpretated here. Rather than isomorphism classes of line bundles one should fix, for $\chi \in \Lambda$, a line bundle $L_\chi$ and, for $\chi, \chi' \in \Lambda$, isomorphisms $L_\chi \otimes L_{\chi'} \cong L_{\chi + \chi'}$ through which the formula $\langle \chi, s \rangle \otimes \langle \chi, s \rangle = \langle \chi + \chi', s \rangle$ ought to be understood.}
\begin{equation} \label{Eq:RelationTrivializationLineBundles} \langle \chi, s \rangle \otimes \langle \chi, s \rangle = \langle \chi + \chi', s \rangle.
\end{equation}
Let $p \colon P \to X$ be the projection. The split torus $T$ over $S$ with group of characters $\Lambda$ acts naturally on $P$ by the rule defined, for an $S$-analytic space $S'$ and $S'$-valued points $s$ of $P$ and $t$ of $T$, and a character $\chi \in \Lambda$,
\[ \langle \chi, ts \rangle = \chi(t) \langle \chi, s \rangle. \]
\label{sec:FourierExpansion}The identity morphism of $P$ corresponds, by definition, to the datum of `universal' trivializations, for $\chi \in \Lambda$,
\[ \langle \chi, \id \rangle \colon \cO_P \stackrel{\sim}{\too} p^\ast L_\chi \]
of the line bundle $p^\ast L_\chi$, undergoing the usual relation \eqref{Eq:RelationTrivializationLineBundles}. For a coherent $\cO_X$-module $F$, these trivializations combined with adjunction along $p$ furnish an injective homomorphism of $\cO_X$-modules
\[ \epsilon_{F, \chi} \colon F \otimes L_\chi \too p_\ast p^\ast F.\]
The goal of this subsection is to prove that $p_\ast p^\ast F$ is the topological direct sum (whose meaning is to be specified below) of the coherent $\cO_X$-modules $F \otimes L_\chi$.

\begin{definition}
Let $F$ be a coherent sheaf of $\cO_X$-modules. Consider the $\cO_X$-module
\[ \hbigoplus_{\chi \in \Lambda} F \otimes L_\chi \subseteq \prod_{\chi \in \Lambda} F \otimes L_\chi\]
whose sections on an analytic domain $Y \subseteq X$ is the subset of sequences $(f_\chi)_{\chi \in \Lambda}$ satisfying the following condition. Let $Y' \subseteq Y$ be an affinoid domain, $s \colon Y' \to P$ a morphism of $X$-analytic spaces,  $\| \cdot \|$ a norm on $\Gamma(Y', F)$ defining the topology,\footnote{Given an epimorphism of $\cO_{Y'}$-modules $\phi \colon \cO_{Y'}^n \to F_{\rvert Y'}$  for some  $n \in \bbN$, an example of such a norm is $\| v\| = \inf_{\phi(w) = v} \max (\| w_1 \|_{Y'}, \dots, \| w_n\|_{Y'})$
where $w = (w_1, \dots, w_n)$ and $\|\cdot \|_{Y'}$ is a norm defining the topology of the affinoid $K$-algebra $\Gamma(Y', \cO_{Y'})$. Note that, by vanishing of higher coherent cohomology on affinoid spaces, the map induced by $\phi$ on global subsections is surjective.} a basis $\chi_1, \dots, \chi_n$ of the free abelian group $\Lambda$ (where $n$ is the rank of $\Lambda$), and $r \ge 1$, $\epsilon > 0$ real numbers. Then there is a finite subset of $\Lambda$ such that, for $\chi \in \Lambda$ outside of which, the  inequality $\| v_\chi\| r^{|\chi|} < \epsilon$ holds, where $v_\chi \in \Gamma(Y, F)$ satifies $f_\chi = v_\chi \otimes \langle s, \chi \rangle$ and where $|\chi| := |a_1| + \cdots + |a_n|$ with $\chi = a_1 \chi_1 + \cdots + a_n \chi_n$. Consider the homomorphism of $\cO_X$-modules
\[ \epsilon_F \colon \hbigoplus_{\chi \in \Lambda} F \otimes L_\chi \too p_\ast p^\ast F, \qquad (f_\chi)_{\chi \in\Lambda} \longmapsto \sum_{\chi \in \Lambda} \epsilon_{F, \chi}(f_\chi),\]
which is well-defined because the series above converges.

\end{definition}

\begin{proposition} \label{Prop:FourierExpansionAnalyticFunctionToricBundle} The homomorphism $\epsilon_F $ is an isomorphism.
\end{proposition}

\begin{proof} The question is local, therefore $X$ may be assumed to be the Banach spectrum of a $K$-affinoid algebra $A$ and the line bundles $L_\chi$ to be trivial. There is also no harm in fixing once for all a basis of the free abelian group $\Lambda$ (two such choices lead to `length' functions $|\cdot|$ on $\Lambda$  whose ratio is bounded) and  a norm $\| \cdot \|$ defining the topology on the finite $A$-module $M := \Gamma(X, F)$ (all such norms are equivalent).  By doing so, one is eventually led back to the case $P = \Gm^n \times X $, where $n$ is the rank of $\Lambda$. For a real number $r \ge 1$, consider the affinoid $A$-algebra
\[ A_r := A \{ r^{-1} t_1, r u_1, \dots, r^{-1} t_n, r u_n \} / (t_i u_i - 1, i = 1, \dots, n). \]
The Banach spectra $P_r$ of the affinoid $K$-algebras $A_r$ form an increasing (with respect to $r$) exhaustion of $P$. Therefore,
\[ \Gamma(P, p^\ast F) = \projlim_{r \ge 1} M \otimes_A A_r \]
because of the equality $ \Gamma(P_r, p^\ast F) = M \otimes_A A_r$. In particular, a global subsection $f$ of $p^\ast F$ can be written uniquely as a series $f = \sum_{a \in \bbZ^n} f_a t^a$ where $t^a = t_1^{a_1} \cdots t_n^{a_n}$ for $a = (a_1, \dots, a_n) \in \bbZ^n$. The convergence of $f$ is expressed as the existence, for real numbers $r \ge 1$ and $\epsilon > 0$, of a finite subset of $\bbZ^n$ outside of which the inequality
\[ \| f_a\| r^{|a_1| + \cdots + |a_n|} < \epsilon \]
holds, as desired.
\end{proof}

\begin{definition}
A section $f \in \Gamma(P, p^\ast F)$ can be written in a unique manner as a series $f = \sum_{\chi \in \Lambda} p^\ast f_\chi$ with $f_\chi \in \Gamma(X, F \otimes L_\chi)$ called the \emph{Fourier expansion} of $f$ .
\end{definition}

\begin{proposition} Let $F$ and $F'$ be coherent $\cO_X$-modules and $\phi \colon F \to F'$ a homomorphism of $\cO_X$-modules. Then, the following  diagram is commutative:
\[
\begin{tikzcd}
\hbigoplus_{\chi \in \Lambda} F \otimes L_\chi  \ar[d, "\phi \otimes \id"'] \ar[r, "\epsilon_F"]&  p_\ast p^\ast F \ar[d, "p_\ast p^\ast \phi"] \\
\hbigoplus_{\chi \in \Lambda} F' \otimes L_\chi \ar[r, "\epsilon_{F'}"] &  p_\ast p^\ast F' 
\end{tikzcd}
\]
\end{proposition}

\begin{proof} It suffices to prove the commutativity summand by summand, that is,
\[ p_\ast p^\ast \phi \circ \epsilon_{\chi, F} = \epsilon_{\chi, F'} \circ (\phi \otimes \id), \qquad \chi \in \Lambda. \]
 But this is clear as $\epsilon_{\chi, F}$ is the composition of the map $F \otimes L_\chi \to p_\ast p^\ast (F \otimes L_\chi)$ given by adjunction and the inverse of the isomorphism $p_\ast p^\ast F \cong p_\ast p^\ast (F \otimes L_\chi)$ induced by the `universal' trivialization $\langle \chi, \id \rangle$; similarly for $\epsilon_{\chi, F'}$.
\end{proof}

An automorphism $\tau$ of the $K$-analytic space $P$ is $T$-equivariant if and only it is induced by an automorphism $\sigma$ of the $K$-analytic space $X$ such that, for $\chi \in \Lambda$, the line bundles $L_\chi$ and $\sigma^\ast L_\chi$ on $X$ are isomorphic.\footnote{Persevering with the original abuse of notation (see \Cref{Footnote:AbuseNotationToricBundle}), the isomorphisms $L_\chi \cong \sigma^\ast L_\chi$ are meant to be compatible with the implied isomorphisms $L_\chi \otimes L_{\chi'} \cong L_{\chi + \chi'}$.} The preceding Proposition immediately implies the following fact, which will be useful later on:

\begin{corollary} \label{Cor:FonctorialDecompositionWrtAutomorphisms} Let $\tau$ be a $T$-equivariant automorphism of the $K$-analytic space $P$ and $\sigma$ the induced automorphism of $X$. Then, the diagram of $\cO_X$-modules
\[
\begin{tikzcd}
F \ar[r, "\epsilon_F"] \ar[d] & \hbigoplus_{\chi \in \Lambda} F \otimes L_\chi \ar[d] \\
\sigma_\ast \sigma^\ast F \ar[r, "\epsilon_{\sigma^\ast F}"] & \hbigoplus_{\chi \in \Lambda} \sigma_\ast \sigma^\ast F \otimes L_\chi,
\end{tikzcd}
\]
where the vertical arrows are given by adjunction along $\sigma$, is commutative.
\end{corollary}

\begin{proposition} \label{Prop:CohomologyPrincipalBundle} Let $F$ be a coherent $\cO_X$-module and $q \ge 0$ an integer. Then,
\begin{enumerate}
\item for $q \ge 1$, the higher direct image $\rR^q p_\ast p^\ast F$ vanishes;
\item for $\chi \in \Lambda$, the homomorphism $\rH^q(X, F \otimes L_\chi) \to \rH^q (P, p^\ast F)$ induced by $\epsilon_{F, \chi}$ is injective.
\end{enumerate}
\end{proposition}

\begin{proof} (1) Upon considering a G-cover of $X$ made of affinoid domains on which the line bundles $L_\chi$ are trivial, one reduces to prove that the $q$th cohomology group of $p^\ast F$ vanishes when $X$ is such an affinoid space.  In this case, the toric bundle $P$ is isomorphic to $X \times \Gm^n$, where $n$ is the rank of $\Lambda$. The $K$-analytic space $X \times \Gm^n$ is Stein, thus higher coherent cohomology vanishes, and so in particular does $\rH^q(X \times \Gm, p^\ast F)$.

(2) The vanishing of $\rR^i p_\ast p^\ast F$ for each $i \ge 1$ coupled with Grothendieck spectral sequence $\rH^i(X, \rR^j p_\ast p^\ast F) \Rightarrow \rH^{i + j} (P, p^\ast F)$ yields $\rH^i(P, p^\ast F) = \rH^i(X, p_\ast p^\ast F)$ for each $i \in \bbN$. Identifying $p_\ast p^\ast F$ with $\hbigoplus_{\chi \in \Lambda} F \otimes L_\chi$ via $\epsilon_F$, the $\cO_S$-module $F \otimes L_\chi$ is seen to be a direct factor of $p_\ast p^\ast F$, whence the conclusion.
\end{proof}

\subsection{Uniformization} \label{sec:TateRaynaudUniformization} Up to passing to finite extension of $K$, by \cite[Corollary 7.6.2]{LutkebohmertCurves} the topological universal cover of $A$ is an extension
\[ 0 \too T \too E \stackrel{p}{\too} B \too 0\]
where $B$ is an abeloid variety with good reduction and $T$ is a split $K$-torus with group of characters $\check{\Lambda}$. The above extension is determined by a group homomorphism $\check{c} \colon \check{\Lambda} \to \check{B}(K)$ where $\check{B}$ is the dual abeloid variety. For $\check{\chi} \in \check{\Lambda}$ set
\[ \cL_{\check{\chi}} := (\id_B, \check{c}(\check{\chi}))^\ast \cL\]
where $\cL$ is the Poincar\'e bundle on $B \times \check{B}$. A $K$-point $g$ of $E$ corresponds to the datum for $\check{\chi} \in \check{\Lambda}$ of a non-zero element $\langle g, \check{\chi} \rangle_E$ of the fiber of $\cL_{\check{\chi}}$ at $p(g) \in B(K)$. Moreover, the identity $\langle g, \check{\chi}  +  \check{\chi}' \rangle_E = \langle g, \check{\chi} \rangle_E \otimes \langle g, \check{\chi}' \rangle_E$ holds for $\check{\chi}, \check{\chi}' \in \check{\Lambda}$ via the isomorphism of line bundles $\cL_{\check{\chi} + \check{\chi}'} \cong \cL_{\check{\chi}}  \otimes \cL_{\check{\chi}'} $. The topological fundamental group $\Lambda := \pi_1(A, e)$ is then identified with a subgroup of $E(K)$. Let $u \colon E \to A = E / \Lambda$ be the quotient map. 

The dual abeloid variety admits a similar description. Let $\check{T}$ be the split $K$-torus with group of characters $\Lambda$. The inclusion $\Lambda \subseteq E(K)$ determines an extension 
\[ 0 \too \check{T} \too \check{E} \stackrel{\check{p}}{\too} \check{B} \too 0\]
together with an inclusion $\check{\Lambda} \subseteq \check{E}(K)$ such that $\check{A} = \check{E} / \check{\Lambda}$ and $ \langle \chi , \check{\chi}\rangle_E = \langle \chi , \check{\chi}\rangle_{\check{E}}$ for all $\chi \in \Lambda$ and $\check{\chi} \in \check{\Lambda}$ with the obvious notation. For this reason from now on the subscripts $E$ and $\check{E}$ will be dropped. The properness of $\check{E} / \check{\Lambda}$ implies that so-defined pairing $\langle -, - \rangle$ is non-degenerate.  Let $\check{u} \colon \check{E} \to \check{A}$ be the quotient map. 

By the usual descent arguments, coherent sheaves on $A$ are identified with coherent sheaves on $E$ together with a $\Lambda$-linearization. For homogeneous line bundles this takes the particularly simple form \cite[Theorem 6.7]{BoschLutkebohmertDegenerating}. A homogeneous line bundle $L$ on $A$ corresponds to a homogeneous line bundle $M$ on $B$ and a $\Lambda$-linearization on $p^\ast B$. Such a $\Lambda$-linearization is the datum of isomorphisms, for $\chi \in \Lambda$,
\[
p^\ast M \too \tr_\chi^\ast p^\ast M \cong p^\ast M \otimes \chi^\ast p^\ast M, \qquad s \longmapsto s \otimes r(\chi)
\]
where $\tr_\chi$ is the translation by $\chi$ on $E$ and $r \in \Gamma(\Lambda, p^\ast M)$ is a trivialization such that $r(\chi) \otimes r(\chi') = r(\chi + \chi')$  for $\chi, \chi' \in \Lambda$. Note that the couple $(M, r)$ is unique only up to isomorphism of $\Lambda$-linearized line bundles, and the following two examples are worth to be kept in mind:

\begin{example} \label{Ex:LinearizationTrivialBundleUniversalCover} Suppose $M = \cO_B$. Then $r(\chi) \in K^\times$ and the line bundle $L$ if is trivial if and only if $r(\chi) = 1$ for all $\chi \in \Lambda$.
\end{example}

\begin{example} \label{Ex:LinearizationTautologicalBundleUniversalCover} Suppose $M = \cL_{\check{\chi}}$ for some $\check{\chi} \in \check{\Lambda}$ and $r(\chi) = \langle \chi, \check{\chi}\rangle$ for each $\chi \in \Lambda$. The neutral element of $E(K)$ induces a `universal' trivialization
\[ \langle \id, \check{\chi}\rangle \colon \cO_E \stackrel{\sim}{\too} p^\ast M.\]
Seeing $\chi \in \Lambda$ as a $K$-point of $E(K)$, the trivialization $\langle \chi, \check{\chi} \rangle$ is just the pull-back along $\chi$ of universal trivialization $\langle \id, \check{\chi}\rangle$. In other words, the isomorphism $\langle \id, \check{\chi}\rangle$ is $\Lambda$-equivariant with respect to the trivial action on $\cO_E$. It follows that the line bundle $L$ is trivial.
\end{example}

Let us recall here some results from \cite[3]{}. Let  $F$ be a finite-dimensional $K$-vector space, $\phi_A \colon \omega_{\check{A}} \to F$ a $K$-linear map and $(\cF_A)$ the short exact sequence obtained pushing-out the canonical extension $(\cU_A)$ along $\phi_A$. Consider the push-out $(\cF_B)$ of the canonical extension $(\cU_B)$ on $B$ along the $K$-linear map
\[ 
\begin{tikzcd}
\phi_B \colon \omega_{\check{B}} \ar[r, "\rd \check{p}"]& \omega_{\check{E}}  \ar[r, "\phi_E"] & F,
\end{tikzcd}
\]
with $\phi_E := \phi_A \circ (\rd \check{u})^{-1} \colon \omega_{\check{E}} \to F$,  where $\omega_{\check{B}}$ and $\omega_{\check{E}}$ are respectively the dual of the Lie algebra of $\check{B}$ and $\check{E}$,  and  $\rd \check{u} \colon \omega_{\check{A}} \to \omega_{\check{E}}$ is the isomorphism induced by the \'etale morphism $\check{u}$. By \cite[3.6]{}, there is a natural isomorphism of short exact sequences
\[ u^\ast (\cF_A) \cong p^\ast (\cF_B).\]
From now on the preceding isomorphism will be tacitly understood and $(\cF_E)$ will denote either sides of the above. When $(\cF_A) = (\cU_A)$ write $(\cU_E)$ for the corresponding extension on $E$ so that $(\cF_E)$ is the push-out of $(\cU_E)$ along $\phi_E$. For $X = A, B, E$ consider the affine bundle \[\bbA(\cF_X) := \bbP(\cF_X) \smallsetminus (\bbP(F) \times X),\] the natural projection $\pi_X \colon \bbA(\cF_X) \to X$ and the morphism $\Phi_X \colon X^\natural \to \bbA(\cF_X)$ given by definition of $(\cF_X)$ as push-out of $(\cU_X)$; here $E^\natural := A^\natural \times_A E$ in the case $X = E$. By construction,
\[ \bbA(\cF_E) = \bbA(\cF_A) \times_A E = \bbA(\cF_B) \times_B E.\]
Let $C_B : = \coker \phi_B$ and $\pr_{u, B} \colon \bbA(\cF_B) \to \bbV(C_B)$ the natural projection. Consider the unique map $\phi_T \colon \omega_{\check{T}} \to C_B$  fitting in the following commutative diagram
\begin{equation} \label{Eq:DecompositionUniversalCoverLinearMapMakingPushOut}
\begin{tikzcd}
0 \ar[r] & \omega_{\check{B}} \ar[r, "\rd \check{p}"] \ar[d, "\phi_B"]& \omega_{\check{E}} \ar[r] \ar[d, "\phi_E"] & \omega_{\check{T}} \ar[d, "\phi_T"] \ar[r]& 0\ \\
0 \ar[r] & \im \phi_B \ar[r] & F \ar[r] & C_B \ar[r] & 0.
\end{tikzcd}
\end{equation}
With this notation, the following diagram of $K$-analytic spaces is commutative:
\[ 
\begin{tikzcd}
E^\natural \ar[r] \ar[d, "\Phi_E"] & \bbV(\omega_{\check{T}}) \ar[d,"\phi_T" ] \\
\bbA(\cF_E) \ar[r, "\pr_{u, E}" ] & \bbV(C_B)
\end{tikzcd}
\]
where the upper horizontal arrow is analogous of the projection $\pr_{u, B}$ for $E^\natural$ and $\pr_{u, E}$ is the composition of $\pr_{u, B}$ with the projection $q \colon \bbA(\cF_E) \to \bbA(\cF_B)$.

The $K$-analytic space $E^\natural$ is the universal cover of $A^\natural$. Moreover the topological fundamental group of $A^\natural$ is naturally identified with the topological fundamental group $\Lambda$ of $A$. This furnishes an injection $\Lambda \to E^\natural(K)$ written $\chi \mapsto \chi^\natural$. Then
\begin{equation} \label{Eq:ProjectionFundamentalGroupOnUnipotentPartOfAffineBundle}
 \pr_{u, E}(\Phi_E(\chi^\natural)) = \phi_T(\theta_\Lambda(\chi)), \qquad \chi \in \Lambda,
\end{equation}
where $\theta_\Lambda \colon \Lambda \to \bbV(\omega_{\check{T}})$ is the universal vector hull of $\Lambda$ which is defined as follows. An element $\chi \in \Lambda$ can be thought as a character $\chi \colon \check{T} \to \Gm$ (because $\check{\Lambda}$ is by definition the group of characters of the torus $\check{T}$) and
\[ \theta_{\Lambda}(\chi) := \chi^\ast \tfrac{\rd z}{z}\]
where $z$ is the coordinate on $\Gm$.

\subsection{Proof in the general case} The compatibility of taking global sections with extension of scalars allows to replace the field $K$ by any finite extension. By doing so the abeloid variety may be supposed to admit a uniformization such as the one described in \Cref{sec:TateRaynaudUniformization}.  Adopt the monolith of notation therein introduced. Via the morphism $q \colon \bbA(\cF_E) \to \bbA(\cF_B)$, the $K$-analytic space $\bbA(\cF_E)$ is seen to be the principal $T$-bundle over $\bbA(\cF_B)$ associated with the group homomorphism 
\[
\check{\Lambda} \too \Pic(\bbA(\cF_B)), \qquad \check{\chi} \longmapsto \pi_B^\ast \cL_{\check{\chi}},
\]
where $\cL_{\check{\chi}}$ is the short for the line bundle $ \cL_{B, (\id_B, \check{c}(\check{\chi}))}$ on $B$. According to \Cref{Prop:FourierExpansionAnalyticFunctionToricBundle}, for a vector bundle $V$ on $\bbA(\cF_B)$, expanding global sections of $q^\ast V$ in Fourier series gives an isomorphism
\begin{equation} \label{Eq:FouriesSeriesVectorBundleAffineBundle}
\Gamma(\bbA(\cF_E), q^\ast V) \cong \hbigoplus_{\check{\chi} \in \check{\Lambda}} \Gamma(\bbA(\cF_B), V \otimes \pi_B^\ast \cL_{\check{\chi}}). 
\end{equation}
Let $M$ be a homogeneous line bundle on $B$ and consider the subset $\check{\Lambda}_M$ of $\check{\Lambda}$ made of those $\check{\chi}$ such that the line bundle $\cL_{\check{\chi}}$ is isomorphic to the dual of $M$:
\[ \check{\Lambda}_M := \{ \check{\chi} \in \check{\Lambda} : \cL_{\check{\chi}} \cong M^\vee\}.\]
According to \Cref{Prop:NonTrivialHomogeneousLineBundlesHaveNoNonzeroSections}, or rather its analogue for abeloid varieties, the line bundle $M \otimes \cL_{\check{\chi}}$ on $B$ has no non-zero global sections as soon as $\check{\chi} \not \in \check{\Lambda}_M$. The field $K$ being of characteristic $0$ and the abeloid variety $B$ having good reduction, it is possible to apply \Cref{Thm:AnalyticFunctionsOnVectorExt} (1) to the affine bundle $\bbA(\cF_B)$ and the line bundle $M \otimes \cL_{\check{\chi}}$ and obtain, for $\check{\chi} \in \check{\Lambda} \smallsetminus \check{\Lambda}_M$, the vanishing
\[ \Gamma(\bbA(\cF_B),  \pi_B^\ast (M \otimes \cL_{\check{\chi}})) = 0.\]
In particular, combining the previous considerations with $V = \pi_B^\ast M$ yields
\[
 \Gamma(\bbA(\cF_E),  \pi_E^\ast p^\ast M) \cong \hbigoplus_{\check{\chi} \in \check{\Lambda}_M} \Gamma(\bbA(\cF_B), \cO_{\bbA(\cF_B)}). 
 \]
Let $C_B$ be the cokernel of the $K$-linear map $\phi_B \colon \omega_{\check{B}} \to F$. \Cref{Thm:AnalyticFunctionsOnVectorExt} (2) in the good reduction case, applied to $\bbA(\cF_B)$, states that the precomposition with  the map $\pr_{u, B} \colon \bbA(\cF_B) \to \bbV(C_B)$ furnishes an isomorphism of $K$-algebras
\[ \Gamma(\bbV(C_B), \cO_{\bbV(C_B)}) \stackrel{\sim}{\too} \Gamma(\bbA(\cF_B), \cO_{\bbA(\cF_B)}). \]
Via this identification, the above decomposition becomes
\begin{equation}
\label{Eq:DecompositionGlobalSectionsLineBundleUniversalCoverAffineBundle}
 \Gamma(\bbA(\cF_E),  \pi_E^\ast p^\ast M) \cong \hbigoplus_{\check{\chi} \in \check{\Lambda}_M} \Gamma(\bbV(C_B), \cO_{\bbV(C_B)}),
\end{equation}
by means of which $f \in \Gamma(\bbA(\cF_E), \pi_E^\ast p^\ast M)$ can be expanded in Fourier series:
\[ f =\sum_{\check{\chi} \in \check{\Lambda}_M} \pr_{u, E}^\ast f_{\check{\chi}},\]
where, for $\check{\chi} \in \check{\Lambda}_M$,  $f_{\check{\chi}}$ is an analytic function on $\bbV(C_B)$ and
\[ \pr_{u, E} := \pr_{u, B} \circ q \colon \bbA(\cF_E) \too \bbV(C_B).\]
Let $L$ be a homogeneous line bundle on $A$ and $M$ the corresponding homogenous line bundle on $B$ such that $p^\ast M$ is endowed with a $\Lambda$-linearization $r$. By definition of $\check{\Lambda}_M$, there is a unique isomorphism of \emph{homogeneous} line bundles $M \cong \cL_{\check{\chi}}^\vee$ on $B$. Via this identification, for $\chi \in \Lambda$, the non-zero element $r(\chi)$ of the $K$-vector space $p(\chi)^\ast M$ can be seen as a linear form on $p(\chi)^\ast \cL_{\check{\chi}}$. Moreover, the section $\langle \chi, \check{\chi}\rangle$ is a non-zero element of $p(\chi)^\ast \cL_{\check{\chi}}$, thus it makes sense to evaluate $r(\chi)$ at $\langle \chi, \check{\chi} \rangle$: write
\[ \beta(\chi, \check{\chi}) := r(\chi).\langle \chi, \check{\chi}\rangle \in K^\times \]
for the result. After this lengthy preparation, \Cref{Thm:AnalyticFunctionsOnVectorExt} will be a consequence of next two lemmas:

\begin{lemma} \label{Lemma:FourierSeriesInvariantSection} With the notation above, for $f \in \Gamma(\bbA(\cF_E), \pi_E^\ast p^\ast M)$, the following are equivalent:
\begin{enumerate}
\item the section $f$ is $\Lambda$-invariant;
\item for each $\check{\chi} \in \check{\Lambda}_M$, the analytic function $f_{\check{\chi}}$ on $\bbV(C_B)$, satisfies, for $\chi \in \Lambda$,
\begin{equation} \label{Eq:LastInvarianceFunctionUnderLattice} \tr_{\chi}^\ast f_{\check{\chi}} = \beta(\chi, \check{\chi}) f_{\check{\chi}},\end{equation}
where $\tr_\chi$ is the translation by $\phi_T( \theta_\Lambda(\chi))$ on $\bbV(C_B)$.
\end{enumerate}
\end{lemma}

%Recall that $\theta_\Lambda \colon \Lambda \to \bbV(\omega_{\check{T}})$ is the universal vector hull of the group $\Lambda$ and $\phi_T \colon \omega_{\check{T}} \to C_B$ the $K$-linear map  induced by $\phi_E \colon \omega_{\check{E}} \to F$ (\emph{cf.} diagram \eqref{Eq:DecompositionUniversalCoverLinearMapMakingPushOut}).

For a fixed $\check{\chi} \in \Lambda_M$ the function $\beta(-, \check{\chi}) \colon \Lambda \to K^\times$ is a group homomorphism. This should explain the interest of the following result:
 
\begin{lemma} \label{Lemma:InvarianceCharacter} Let $X$ be a $K$-analytic space, $\Gamma$ a finitely generated free abelian group, $V$ a finite-dimensional $K$-vector space, $\sigma \colon \Gamma \to V$ and $\tau \colon \Gamma \to K^\times$ homomorphisms of abelian groups, and 
\[ \cI := \{ f \in \Gamma(\bbV(V) \times X, \cO_{\bbV(V) \times X}) : \tr_{\sigma(\gamma)}^\ast f = \tau(\gamma)f \textup{ for all } \gamma \in \Gamma\}\]
 where $\tr_{\sigma(\gamma)}$ is the translation by $\sigma(\gamma)$ on $\bbV(V) \times X$.
 
 \begin{enumerate} 
 
 \item If the character $\tau$ is not identically equal to $1$, then $\cI = 0$;
 \item If $\tau = 1$ and $W = \coker(\sigma \otimes \id \colon \Gamma \otimes_{\bbZ} K \to V)$, then precomposing with the projection $\bbV(V) \to \bbV(W)$ induces an isomorphism 
 \[ \Gamma(\bbV(W) \times X, \cO_{\bbV(W) \times X}) \stackrel{\sim}{\too} \cI. \]
 \end{enumerate}
\end{lemma}

The presence of the $K$-analytic space $X$ in the statement is just to make it better-behaved with regards to induction; at last, it will applied with $X$ being a $K$-rational point. Before proving this two statements, let us see how they permit to conclude the proof of \Cref{Thm:AnalyticFunctionsOnVectorExt}.

\begin{proof}[{Proof of \Cref{Thm:AnalyticFunctionsOnVectorExt}}] First of all, the affine bundle $\bbA(\cF_A)$ being the quotient of $\bbA(\cF_E)$ by the group $\Lambda$, global sections of the line bundle $\pi_A^\ast L$ on $\bbA(\cF_A)$ are identified with those of the line bundle $\pi_E^\ast p^\ast M$ on $\bbA(\cF_E)$ that are invariant under the action of $\Lambda$.  According to \Cref{Lemma:FourierSeriesInvariantSection}, and with the notation therein defined, understanding $\Lambda$-invariant sections of $\pi_E^\ast p^\ast M$ amounts to determining, for each $\check{\chi} \in \check{\Lambda}_M$, the analytic functions on $\bbV(C_B)$ undergoing the relation \eqref{Eq:LastInvarianceFunctionUnderLattice}. If the set $\check{\Lambda}_M$ is empty, then there is nothing do. Otherwise, for a fixed $\check{\chi} \in \check{\Lambda}_M$, this is done by applying \Cref{Lemma:InvarianceCharacter} with $X$ being a $K$-rational point and
\begin{align*}
V = C_B, && M = \Lambda,  && \sigma = \phi_T \circ \theta_\Lambda, &&\tau = \beta(-, \check{\chi}).
\end{align*}

 (1) Assume the line bundle $L$ to be non-trivial. What needs to be noticed is that the natural $\Lambda$-linearization carried by the (trivial, by definition  of $\check{\Lambda}_M$) line bundle $p^\ast (M \otimes \cL_{\check{\chi}})$ on $E$ is non-trivial. If it were not to be the case, then the line bundle $L$ would be isomorphic to the dual of the homogeneous line bundle on $A$ induced by the $\Lambda$-linearized line bundle $p^\ast \cL_{\check{\chi}}$ on $E$. The latter being trivial (see \Cref{Ex:LinearizationTautologicalBundleUniversalCover}), this would contradict the assumption of $L$ being non-trivial. As seen in \Cref{Ex:LinearizationTrivialBundleUniversalCover}, this forces the existence of $\chi \in \Lambda$ such that $\beta(\chi, \check{\chi}) \neq 1$. Thus \Cref{Lemma:FourierSeriesInvariantSection} (1) states that there are no non-zero analytic functions on $\bbV(C_B)$ satisfying \eqref{Eq:LastInvarianceFunctionUnderLattice}.

\medskip
   
 (2) Suppose $L = \cO_A$. In this case, the line bundle $M$ can be taken to be $\cO_B$ and the trivialization $r$ to take the value $r(\chi) = 1$ for each $\chi \in \Lambda$. Then, the subset $\Lambda_M$ is simply $\Ker \check{c}$ and, for $\chi \in \Lambda$,
 \[ \beta(\chi, \check{\chi}) = \langle \chi, \check{\chi}\rangle.\]
(Since $\check{\chi}$ lies in the kernel of $\check{c}$, the line bundle $\cL_{\check{\chi}}$ is tautologically trivial and, for $\chi \in \Lambda$, $\langle \chi, \check{\chi}\rangle$ a non-zero element of $K$.) If $\check{\chi}$ is non-zero, then there is $\chi \in \Lambda$ such that $\langle \chi, \check{\chi} \rangle \neq 1$, because the pairing $\langle -, -\rangle$ on $\Lambda \times \check{\Lambda}$ is nondegenerate. \Cref{Lemma:FourierSeriesInvariantSection} (1) implies again that there are no non-zero analytic functions on $\bbV(C_B)$ satisfying \eqref{Eq:LastInvarianceFunctionUnderLattice}. The only case left to consider is when $\check{\chi} = 0$. Remark that the cokernel of the $K$-linear map $\Lambda \otimes_{\bbZ} K \to C_B$ induced by $\phi_T \circ \theta_{\Lambda}$ can be identified with the cokernel $C$ of $\phi_A \colon \omega_{\check{A}} \to F$.  This is true for the following reasons: first, the image of $\theta_\Lambda$ clearly spans  $\omega_{\check{T}}$ as a $K$-vector space; second, the diagram \eqref{Eq:DecompositionUniversalCoverLinearMapMakingPushOut} is commutative; third, the $K$-linear map $\phi_E$ have same cokernel as $\phi_A$, being the composition of the latter with the isomorphism $\rd \check{u} \colon \omega_{\check{A}} \to \omega_{\check{E}}$. \Cref{Lemma:FourierSeriesInvariantSection} (2) states that analytic functions on $\bbV(C_B)$ satisfying \eqref{Eq:LastInvarianceFunctionUnderLattice} with $\check{\chi} = 0$ are exactly those coming from $\bbV(C)$ via the quotient map $\bbV(C_B) \to \bbV(C)$. This concludes the proof.
\end{proof}

\begin{proof}[{Proof of the \Cref{Lemma:FourierSeriesInvariantSection}}] To get started with, remark that the linear action of $\Lambda$ on the global sections of $\pi_E^\ast p^\ast M$ respects the decomposition \eqref{Eq:DecompositionGlobalSectionsLineBundleUniversalCoverAffineBundle}.   
The reason for that is the fonctoriality of the decomposition \eqref{Eq:FouriesSeriesVectorBundleAffineBundle}, from which \eqref{Eq:DecompositionGlobalSectionsLineBundleUniversalCoverAffineBundle} is deduced, with respect to automorphisms of $\bbA(\cF_E)$ commuting with the action of the torus $T$ (\Cref{Cor:FonctorialDecompositionWrtAutomorphisms}). This is meant to be applied,  for $\chi \in \Lambda$, to the translation $\tr_{\chi}$ by $\Phi_E(\chi^\natural)$ on $\bbA(\cF_E)$, the alluded commutation being satisfied because the analytic group $\bbA(\cF_E)$ is abelian. Now, by definition, a global section $f$ of the line bundle $\pi_E^\ast p^\ast M$ is invariant under the action of $\Lambda$ if and only if, for $\chi \in \Lambda$, 
\[ \tr_{\chi}^\ast f = f \otimes r(\chi).\]
Let $\sum_{\check{\chi} \in \check{\Lambda}_M} \pr_{u, E}^\ast f_{\check{\chi}}$ be the  Fourier expansion of $f$. The preservation of the decomposition \eqref{Eq:DecompositionGlobalSectionsLineBundleUniversalCoverAffineBundle} under the action of $\Lambda$ implies that $f$ is $\Lambda$-invariant if and only if $f_{\check{\chi}}$ is for each $\check{\chi} \in \check{\Lambda}_M$ is. In particular, it suffices to show that, for $\check{\chi} \in \check{\Lambda}_M$, the analytic function $f_{\check{\chi}}$ is $\Lambda$-invariant if and only if it fulfills the condition \eqref{Eq:LastInvarianceFunctionUnderLattice} for every $\chi \in \Lambda$.
In order to do so, begin with making explicit the action of $\Lambda$ on the factor of the right-hand side of \eqref{Eq:DecompositionGlobalSectionsLineBundleUniversalCoverAffineBundle} corresponding to some $\check{\chi} \in \check{\Lambda}_M$. For, the $\Lambda$-linearization of the line bundle $p^\ast \cL_{\check{\chi}}$ is the datum, for $\chi \in \Lambda$, of the isomorphism
\[
p^\ast \cL_{\check{\chi}} \too \tr_\chi^\ast p^\ast \cL_{\check{\chi}} , \qquad s \longmapsto s \otimes \langle \chi, \check{\chi}\rangle.
\]
Therefore, recalling the description of the $\Lambda$-linearization of the line bundle $p^\ast M$, the $\Lambda$-linearized line bundle $p^\ast (M \otimes \cL_{\check{\chi}})$ on $E$ is isomorphic to the line bundle $\cO_E$ together with the $\Lambda$-linearization given, for $\chi \in \Lambda$, by the isomorphism
\[ g \longmapsto  \beta(\chi, \check{\chi}) g.\]
The conclusion is reached by applying the above with $g = \pr_{u, E}^\ast f_{\check{\chi}}$ and making use of the identity \eqref{Eq:ProjectionFundamentalGroupOnUnipotentPartOfAffineBundle}.
\end{proof}

\begin{proof}[{Proof of \Cref{Lemma:InvarianceCharacter}}] Begin with a series of reductions. First, the kernel of $\sigma$ may be assumed to be contained in that of $\tau$. For (2) this is clear, as the latter is whole abelian group $\Gamma$; for (1), if $\gamma \in \ker \sigma$ is such that $\tau(\gamma) \neq 1$, then the invariance $f = \tau(\gamma) f$ implies $f = 0$. In view of this, the map $\sigma$ may be supposed injective and $\Gamma$ identified with its image in $V$.  Second, the $K$-subvector space $V'$ of $V$ generated by the lattice $\Gamma$ may be assumed to coincide with $V$. For, a choice of a section of the projection $V \to W$ yields an isomorphism $V \cong V' \oplus W$. Then, it suffices to replace $X$ by $X \times \bbV(W)$. Third, arguing by induction, the $K$-vector space $V$ may be assumed to be $1$-dimensional. Moreover, by choosing a generator $\gamma$ of the rank $1$ free abelian group $\Gamma$ and setting $u = \tau(\gamma)$, the set $\cI$ is identified with those of analytic functions $f$ on $X \times \bbA^1_K$ such that, for an integer $n \in \bbZ$,
\begin{equation} \label{Eq:LastDefinitionInvariantFunction} f(x + n) = u^n f(x).
\end{equation}
Fourth, the question being local on $X$, the $K$-analytic space $X$ may be assumed to be the Banach spectrum of a $K$-affinoid algebra $C$.
%\begin{remark} \label{Rmk:AccumulationImpliesConstant} When $X$ is a $K$-rational point, the desired result is a consequence of non-accumulation of zeros of an analytic function. Indeed, invertible analytic functions on $\bbA^1_K$ being constant, it suffices to study the case of $f \in \cI$ having a zero in some point $x_0$. The invariance implies that $f$ vanishes also in $x_0 + n$ for any $n \in \bbZ$. Pick a real number $r > 0$ such that the closed disk $D_r$ in $\bbA^1_K$ contains $x_0$. Then, by the non-Archimedean triangle inequality, for each integer $n \in \bbZ$, the point $x_0 + n$ belongs to $D_r$. A non-zero $K$-analytic function on a closed disc of $\bbA^1_K$ has only finitely many zeros, thus $f$ must vanish altogether.
%\end{remark}
Now, pick $f \in \cI$ and expand it in power series:
\[ f(x) = \sum_{i = 0}^\infty f_i x^i,\]
with $f_i \in C$ such that, for a real number $r > 0$, $\| f_i \|_C r^i \to 0$ as $i \to \infty$,  where $\| \cdot \|_C$ is a norm on the $K$-affinoid algebra $C$ defining its topology. With this notation, the expansion of $f(x + n)$ is
\[ f(x + n) = \sum_{i = 0}^\infty  \sum_{k \ge i} \binom{k}{i} f_k n^{k - i}  x^i. \]
Comparing the power series of the two sides of \eqref{Eq:LastDefinitionInvariantFunction} yields the equality, for an integers $n \neq 0$ and $i \ge 0$,
\begin{equation} \label{Eq:RelationForInvariantFunctionUnderCharacter} (u^n - 1) f_i = \sum_{k > i} \binom{k}{i} f_k n^{k - i} .
\end{equation}

(1) Applying the non-Archimedean triangle inequality to \eqref{Eq:RelationForInvariantFunctionUnderCharacter} with $n = 1$ gives
\[ \| f_i \|_C \le |u - 1|^{-1} \max_{k > i}   \| f_k \|_C, \]
because integers have always absolute value $\le 1$. A repeated use of such an inequality yields, for an integer $N \ge 0$,
\begin{equation} \label{Eq:HopefullyTheLastInequality} \| f_i \|_C \le |u - 1|^{-N} \max_{k \ge i + N}   \| f_k \|_C.
\end{equation}
If $|u - 1| \ge 1$, then the convergence of $f$ implies that the right-hand side of \eqref{Eq:HopefullyTheLastInequality} tends to $0$ as soon as $N$ tends to infinity, showing $f_i = 0$. Suppose $|u - 1| < 1$. For a real number $0 < r < 1$, the inequality \eqref{Eq:HopefullyTheLastInequality} entails
\[ \| f_i\|_C \le \left( \frac{r}{|u-1|} \right)^N \max_{k \ge i + N}  \{ \| f_k\|_C r^{-k} \}, \]
because $r^k \le r^{N}$ for $k \ge i + N$. In particular, for $r = |u - 1|$,
\[ \| f_i\|_C \le \max_{k \ge i + N} \{ \| f_k\|_C r^{-k} \}. \]
Again by the convergence of $f$, the right-hand side goes to $0$ as $N \to \infty$, eventually proving that $f_i$ vanishes.

(2) If $u = 1$, then dividing the identity \eqref{Eq:RelationForInvariantFunctionUnderCharacter} by $n$ and replacing $i + 1$ by $i$ implies, for an integer $i \ge 1$,
\[ i f_i = - \sum_{k \ge i + 1} \binom{k}{i - 1} f_k n^{k - i}.\]
If the valuation on $K$ is $p$-adic, then one concludes as in the proof of \Cref{Lemma:FormalFunctionsInvariantUnderLattice}: by taking $n = p^N$ and letting $N$ tend to infinity, the right-hand side of the previous identity tends to $0$, thus $f_i$ must vanish. When the valuation of $K$ induces the trivial one on $\bbQ$, the above formula together with the non-Archimedean triangle inequality for $\| \cdot \|_C$ implies
\[ \| f_i\|_C \le \max_{k > i + 1} \| f_k \|_C, \]
the absolute value of a non-zero integer being $1$. Applying it over and over again, yields, for each integer $N \ge 0$,
\[ \| f_i\|_C \le \max_{k \ge i + N} \| f_k \|_C. \]
The convergence of $f$ implies that the right-hand side of the above inequality tends to $0$ as soon as $N$ goes to infinity, whence $f_i = 0$.  The integer $i \ge 1$ being arbitrary, the analytic function $f$ coincides with $f_0$ and therefore belongs to $C$.
\end{proof}

\subsection{Proof of \Cref{Cor:AnalyticFunctionsVectExtSemiAbelian}} The proof of statement (1) does not need \Cref{Thm:AnalyticFunctionsOnVectorExt}. It is just a consequence of Fourier expansion for the toric bundle $G$ over $A$ (\Cref{Prop:FourierExpansionAnalyticFunctionToricBundle}) and the vanishing of global sections of a non-trivial homogeneous line bundles on $A$ (the rigid-analytic analogue of \Cref{Prop:NonTrivialHomogeneousLineBundlesHaveNoNonzeroSections}). For (2), the projection $G \to \bbA(\cF)$ permits to see the $K$-analytic space $G$ as the principal $T$-bundle over $\bbA(\cF)$ associated with the group homomorphism
\[ \Lambda \too \Pic(\bbA(\cF)), \qquad \chi \longmapsto p^\ast \cL_\chi, \]
where $p \colon \bbA(\cF) \to A$ and $\cL_\chi$ is the restriction to $A \times \{ c(\chi)\}$ of the Poincar\'e bundle $\cL$ on $A \times \check{A}$. Therefore, Fourier expansion yields an isomorphism (see \Cref{Prop:FourierExpansionAnalyticFunctionToricBundle}, with the notation therein defined),
\[ \Gamma(G, \cO_G) \cong \hbigoplus_{\chi \in \Lambda} \Gamma(\bbA(\cF), p^\ast \cL_\chi). \]
Since the field $K$ is of characteristic $0$, \Cref{Thm:AnalyticFunctionsOnVectorExt} (1) can be applied to obtain, for $\chi \not \in \ker c$,
\[  \Gamma(\bbA(\cF), p^\ast \cL_\chi)= 0.\]
In turn, by \Cref{Thm:AnalyticFunctionsOnVectorExt} (1),  precomposing with $\pr_u \colon \bbA(\cF) \to \bbV(C)$ induces an isomorphism of $K$-algebras
\[ \Gamma(\bbA(\cF), \cO_{\bbA(\cF)}) \cong \Gamma(\bbV(C), \cO_{\bbV(C)}). \]
Gathering all of this yields an isomorphism of $K$-algebras
\[ \Gamma(G, \cO_G) \cong \hbigoplus_{\chi \in \ker c} \Gamma(\bbV(C), \cO_{\bbV(C)}). \]
Applying again \Cref{Prop:FourierExpansionAnalyticFunctionToricBundle} to the toric bundle $T_0 \times \bbV(C) \to \bbV(C)$ permits to identify the right-hand side of the above equation with $\Gamma(\bbV(C) \times T_0, \cO_{\bbV(C) \times T_0})$, thus concluding the proof. \qed

\appendix

\section{Stein spaces} \label{sec:Stein}

%\begin{definition}
%An algebraic variety $X$ over $K$ is \emph{Stein} if the $K$-analytic space $X^\an$ is.
%\end{definition}

\begin{definition} 
A $K$-analytic space $X$ is \emph{Stein} if it is separated, countable at infinity and, for each complete valued extension $K'$ of $K$, higher coherent cohomology vanishes on the $K'$-analytic space $X'$ deduced from $X$ by extension of scalars; that is $\rH^q(X', F)$ for any coherent $\cO_{X'}$-module $F$ and any integer $q \ge 1$.

A \emph{Stein exhaustion} of a $K$-analytic space $X$ is a G-cover $\{ X_n \}_{n \in \bbN}$ made of Stein compact analytic domains $X_n$ of $X$ such that, for an integer $n \ge 0$, $X_n$ is contained in $X_{n + 1}$ and the restriction map $\Gamma(X_{n + 1}, \cO_X) \to \Gamma(X_n, \cO_X)$ has a dense image (\cite[Definition 1.1]{notions}).\footnote{A compact Stein $K$-analytic space is called a Liu space in \cite{notions}, so that a Stein exhaustion is nothing but a W-exhaustion by Liu domains in terminology of \emph{op.cit.}, Definition 1.10.}
\end{definition}

\begin{definition} A $K$-analytic space $X$ is
\begin{itemize}
\item \emph{holomorphically separated} if, for distinct points $x$ and $x'$ of $X$, there is a $K$-analytic function $f$ on $X$ such that $|f(x)| \neq |f(x')|$;
\item \emph{holomorphically convex} if, for a compact subset $C$ of $X$, its holomorphically convex hull $\{ x \in X : |f(x)| \le \| f\|_C, f \in \Gamma(X, \cO_X)\}$ is compact, where $\| \cdot \|_C$ is the sup norm on $C$.
\end{itemize}
\end{definition}

\begin{theorem} \label{Thm:KiehlTheorem} Let $X$ be a separated and countable at infinity  $K$-analytic space. Then the  following are equivalent:
\begin{enumerate}
\item $X$ is  Stein;
\item $X$ admits a Stein exhaustion;
\item $X$ is holomorphically separable, holomorphically convex and  $\cO_X$ is acyclic.
\end{enumerate}
\end{theorem}

\begin{proof} This is \cite[Theorem 1.11]{notions}. Note that, as the field $K$ is non-trivially valued, by Theorem A.5 in \emph{op.cit.}, the vanishing of the higher cohomology of $\cO_X$ implies that vanishing over any complete valued field extension of $K$.
\end{proof}

\begin{proposition} \label{Prop:ProductOfSteinIsStein} For $i = 1, 2$, let $X_i$ be a Stein $K$-analytic space and  $X_i \to S$ a separated morphism of $K$-analytic spaces.  Then $X_1 \times_S X_2$ is Stein.
\end{proposition}

\begin{proof} The fibered product $X_1 \times_S X_2$ is a closed analytic subspace of the product $X:= X_1 \times X_2$, reducing the matter to the case when $S$ is a $K$-rational point. When $X_1$ and $X_2$ are compact, let $K'$ be an algebraically closed complete valued extension of $K$ such that, for $i = 1, 2$, the $K'$-analytic space $X'_{i}$, deduced from $X_i$ by extension of scalars, is strict. Then, according to  Theorem 5.2 and allowing for Theorem A.5 in \cite{notions}, it suffices to show that the sheaf $\cO_{X}$ is acyclic and, given two distinct $K'$-points $x$, $x'$ of $X$, there is a $K'$-analytic function $f$ on $X' := X'_1 \times X_2'$ such that $|f(x)| \neq |f(x')|.$ The vanishing of higher cohomology of $\cO_X$ is granted by \emph{op.cit.}, Theorem A.6. For two distinct $K'$-points $x = (x_1, x_2)$, $x' = (x'_1, x'_2)$, there is $i  = 1, 2$ for which $x_i \neq x'_i$. Up to permuting the factors suppose $i = 1$. Then, the $K'$-analytic space $X_1'$ is Stein, the notion being insensible to scalar extensions, and thus holomorphically separable. Therefore, there is a $K'$-analytic function $f$ such that $|f(x_1)| \neq |f(x'_1)|$. In the general case, according to \Cref{Thm:KiehlTheorem}, it suffices to exhibit a Stein exhaustion of $X_1 \times X_2$. Now, if $\{ X_{i, a}\}_{a \in \bbN}$ is a Stein exhaustion of $X_i$ for $i = 1, 2$, the preceding case implies that the collection $\{ X_{1, a} \times X_{2, b}\}_{a, b \in \bbN}$ is a Stein exhaustion of $X$.
\end{proof}

\begin{lemma}[{\cite[Corollary 1.16]{notions}}] \label{Lemma:ClosedOfSteinIsStein}
A closed analytic subspace $Z$ of a Stein $K$-analytic space $X$ is Stein.
\end{lemma}

\begin{theorem} \label{Thm:ConsequenceSteinEmbedding} Let $X$ be a $K$-analytic space without boundary, countable at infinity and such that $\sup_{x \in X} \dim_{\cH(x)} \Omega^1_{X, x} < + \infty$. Then $X$ is Stein if and only if there is a closed immersion $X \to \bbA^{n, \an}_K$ for some integer $n \ge 0$.
\end{theorem}

\begin{proof} If $X$ is Stein, then the wanted embedding is given by \cite[Theorem 4.25]{LutkebohmertStein}, granted that the $K$-analytic space $X$ admits a Stein exhaustion made of affinoids domains (as opposed to compact Stein analytic spaces). Under the absence of boundary, the latter holds by  \Cref{Thm:KiehlTheorem} and \cite[Corollary 1.14]{notions}.

Conversely, the analytification of affine varieties are Stein spaces. \Cref{Lemma:ClosedOfSteinIsStein} permits to conclude.
\end{proof}

\begin{remark} \label{EquivalentDefinitionSteinAlgebraicVariety}
For an algebraic variety $X$, the $K$-analytic space $X^\an$ is without boundary and countable at infinity. Moreover, the bound $\sup_{x \in X^\an} \dim_{\cH(x)} \Omega^1_{X^\an, x} < + \infty$ is always satisfied. It follows from \Cref{Thm:ConsequenceSteinEmbedding} that the definition of Stein algebraic variety just given agrees with the one in the introduction of this paper.
\end{remark}

\begin{proposition} Let $E$ be a vector bundle on a Stein $K$-analytic space $X$. Then the total space $\bbV(E)$ is Stein.
\end{proposition}

\begin{proof} Write $V := \bbV(E)$. When $Y$ is compact, by the analogue of Cartan's Theorem A in the current framework (\cite[Proposition 2.1]{notions}), there exists an integer $N \ge 0$ and an epimorphism of $\cO_Y$-modules $\phi \colon \cO_Y^N \to E^\vee$. This furnishes a closed immersion $j \colon V \to \bbA^N_Y$, proving that $V$ is Stein. Before moving to the general case, remark that a Stein exhaustion on $V$ is constructed as follows. For a real number $r > 0$, set $V_{Y, \phi, r} := j^{-1}(D_r)$ where $D_r = \{ x \in \bbA^N_Y : |t_i(x)| \le r, i = 1, \dots, N \}$ is the disk of radius $r$ in $\bbA^N_Y$ and $t_1, \dots, t_N$ are the coordinate functions on $\bbA^N_Y$. Then,
\begin{enumerate}
\item the $K$-analytic space $V_{Y, \phi, r}$ is a closed subspace of $D_r$ thus Stein;
\item the restriction map $\Sym \Gamma(Y, E^\vee) \to \Gamma(V_{Y, \phi, r}, \cO_V)$ has a dense image.
\end{enumerate}
To see (2) notice that the polynomial $A$-algebra $A[t_1, \dots, t_N]$ is dense in \[\Gamma(\bbA^N_Y, \cO_{\bbA^N_Y}) = A \{t_1/r, \dots, t_N/r \},\] the restriction map $\rho \colon \Gamma(D_r, \cO_{\bbA^N_Y}) \to \Gamma(V_{Y, \phi, r}, \cO_{V})$ is surjective because $\bbA^N_Y$ is Stein, and the image of $A[t_1, \dots, t_N]$ via $\rho$ coincides with that of $\Sym \Gamma(Y, E^\vee)$. The sought-for Stein exhaustion is for instance $\{ V_{Y, \phi, r_\alpha} \}_{\alpha \in \bbN}$, where $(r_\alpha)_{\alpha \in \bbN}$ is an increasing unbounded sequence of positive real numbers. 

Now turn to the general case. Let $\{ Y_\alpha \}_{\alpha \in \bbN}$ be a Stein exhaustion of $Y$. Then, by \cite[Proposition 5.1]{notions} (itself based on \cite[Satz 2.4]{KiehlStein}), for a non-negative integer $\alpha \ge 0$, the restriction map  $\Gamma(Y, E^\vee) \to \Gamma(Y_\alpha, E^\vee)$ has a dense image. In particular, for an epimorphism of $\cO_{Y_\alpha}$-modules $\phi \colon \cO_{Y_\alpha}^N \to E^\vee_{\rvert Y_\alpha}$ and $r> 0$, the restriction map $\Sym \Gamma(Y, E^\vee) \to \Gamma(V_{Y_\alpha, \phi, r}, \cO_{V})$
has a dense image, where the notation above has been adopted. Construct inductively as follows a Stein exhaustion $\{ V_\alpha \}$ for $V$ where each $V_\alpha$ is of the form $V_{Y_\alpha, \phi_\alpha, r_\alpha}$ for some $r_\alpha > 0$ and a surjective homomorphism of $\cO_{Y_\alpha}$-modules $\phi_\alpha \colon \cO_{Y_\alpha}^{N_\alpha} \to E_{\rvert Y_\alpha}$.  For $\alpha = 0$, the positive real number $r_0$ and the epimorphism $\phi_0$ can be picked freely. For $\alpha \ge 1$, fix $\phi_\alpha$ arbitrarily and choose a real number $\rho_\alpha > 1$ so that the compact analytic domain $V_{Y_\alpha, \phi_\alpha, \rho_\alpha}$ contains $V_{\alpha -1}$; finally, set $r_\alpha = 2 \rho_\alpha$. The $K$-analytic space $V_\alpha$ is Stein and, by the arguments above, the image of the restriction map  $\Gamma(V_\alpha, \cO_{V}) \to \Gamma(V_{\alpha-1}, \cO_{V})$
is dense, as it contains $\Sym \Gamma(Y, E^\vee)$.  This shows that the so-obtained collection of compact analytic domains $\{ V_\alpha \}_{\alpha \in \bbN}$ is a Stein exhaustion of $V$, concluding the proof.
\end{proof}

\begin{proposition} \label{Prop:ToricBundleStein} Let $X$ be a $K$-analytic space without boundary and countable at infinity. Let $L$ be a line bundle on $X$ and $P$ its total space deprived of it zero section. Then $X$ is Stein if and only if $P$ is.
\end{proposition}

\begin{proof} Suppose that $P$ is Stein. By definition, showing that $X$ is Stein amounts to proving the vanishing of higher coherent cohomology on $X \times_K K'$ for any given valued field extension $K'$ of $K$. Up to extending scalars, one reduces to the case $K = K'$. Now, for a coherent $\cO_X$-module $F$ and an integer $q \ge 1$, the cohomology group $\rH^q(P, p^\ast F)$ vanishes, thus $\rH^q(X, F)= 0$ by \Cref{Prop:CohomologyPrincipalBundle} (2). 

Suppose that $X$ is Stein. Consider the duality pairing $\beta \colon \bbV(L) \times_X \bbV(L^\vee) \to \bbA^1$ and the `diagonal' morphism $\delta \colon \bbV(L)^\times \to \bbV(L)^\times \times_X \bbV(L^\vee)^\times $ sending a trivialization $s$ to $(s, \check{s})$ where $\check{s}$ is the unique linear form on $L$ whose value on $s$ is $1$. The map $\delta$ induces an isomorphism of $P$ with the fiber at $1$ of $\beta$. \Cref{Prop:ProductOfSteinIsStein} implies that the $K$-analytic space $\bbV(L) \times_X \bbV(L^\vee)$ is Stein, thus $P = \beta^{-1}(1)$ is Stein too.
\end{proof}

\small

\bibliography{./../biblio}

\newcommand{\etalchar}[1]{$^{#1}$}
\providecommand{\bysame}{\leavevmode\hbox to3em{\hrulefill}\thinspace}
\providecommand{\MR}{\relax\ifhmode\unskip\space\fi MR }
\providecommand{\MRhref}[2]{%
  \href{http://www.ams.org/mathscinet-getitem?mr=#1}{#2}
}
\providecommand{\href}[2]{#2}
\begin{thebibliography}{FGI{\etalchar{+}}05}

\bibitem[Abb10]{AbbesEGR}
A.~Abbes, \emph{\'{E}l\'{e}ments de g\'{e}om\'{e}trie rigide. {V}olume {I}},
  Progress in Mathematics, vol. 286, Birkh\"{a}user/Springer Basel AG, Basel,
  2010, Construction et \'{e}tude g\'{e}om\'{e}trique des espaces rigides.,
  With a preface by Michel Raynaud.

\bibitem[Ber93]{BerkovichIHES}
V.G. Berkovich, \emph{\'{E}tale cohomology for non-{A}rchimedean analytic
  spaces}, Inst. Hautes \'Etudes Sci. Publ. Math. (1993), no.~78, 5--161
  (1994).

\bibitem[BL91]{BoschLutkebohmertDegenerating}
S.~Bosch and W.~L\"{u}tkebohmert, \emph{Degenerating abelian varieties},
  Topology \textbf{30} (1991), no.~4, 653--698.

\bibitem[BLR90]{NeronModels}
S.~Bosch, W.~L\"{u}tkebohmert, and M.~Raynaud, \emph{N\'{e}ron models},
  Ergebnisse der Mathematik und ihrer Grenzgebiete (3), vol.~21,
  Springer-Verlag, Berlin, 1990.

\bibitem[Bri09]{BrionAntiAffine}
M.~Brion, \emph{Anti-affine algebraic groups}, J. Algebra \textbf{321} (2009),
  no.~3, 934--952.

\bibitem[Col84]{ColemanHodgeTate}
R.F. Coleman, \emph{Hodge-{T}ate periods and {$p$}-adic abelian integrals},
  Invent. Math. \textbf{78} (1984), no.~3, 351--379.

\bibitem[Col91]{ColemanBiextension}
\bysame, \emph{The universal vectorial bi-extension and {$p$}-adic heights},
  Invent. Math. \textbf{103} (1991), no.~3, 631--650.

\bibitem[Col98]{ColemanFourier}
\bysame, \emph{Duality for the de {R}ham cohomology of an abelian scheme}, Ann.
  Inst. Fourier (Grenoble) \textbf{48} (1998), no.~5, 1379--1393.

\bibitem[Con02]{ConradChevalley}
B.~Conrad, \emph{A modern proof of {C}hevalley's theorem on algebraic groups},
  J. Ramanujan Math. Soc. \textbf{17} (2002), no.~1, 1--18.

\bibitem[Day21]{DayliesDescent}
M.~Daylies, \emph{Descente fid\`element plate et alg\'ebrisation en
  g\'eom\'etrie de {B}erkovich}, ar{X}iv, 2021,
  \url{https://arxiv.org/abs/2103.10490}.

\bibitem[DG70]{DemazureGabriel}
M.~Demazure and P.~Gabriel, \emph{Groupes alg\'{e}briques. {T}ome {I}:
  {G}\'{e}om\'{e}trie alg\'{e}brique, g\'{e}n\'{e}ralit\'{e}s, groupes
  commutatifs}, Masson \& Cie, \'{E}diteur, Paris; North-Holland Publishing
  Co., Amsterdam, 1970, Avec un appendice {{\i}t Corps de classes local} par
  Michiel Hazewinkel.
  
\bibitem[EGA III$_1$]{EGAIII1}
A.~Grothendieck, \emph{\'{E}l\'{e}ments de g\'{e}om\'{e}trie alg\'{e}brique.
  {III}. \'{E}tude cohomologique des faisceaux coh\'{e}rents. {I}}, Inst.
  Hautes \'{E}tudes Sci. Publ. Math. (1961), no.~11, 167.  

\bibitem[Fal87]{FaltingsHodgeTate}
G.~Faltings, \emph{Hodge-{T}ate structures and modular forms}, Math. Ann.
  \textbf{278} (1987), no.~1-4, 133--149.

\bibitem[FGI{\etalchar{+}}05]{FGAExplained}
B.~Fantechi, L.~G\"{o}ttsche, L.~Illusie, S.~L. Kleiman, N.~Nitsure, and
  A.~Vistoli, \emph{Fundamental algebraic geometry}, Mathematical Surveys and
  Monographs, vol. 123, American Mathematical Society, Providence, RI, 2005,
  Grothendieck's FGA explained.

\bibitem[FK18]{FujiwaraKato}
K.~Fujiwara and F.~Kato, \emph{Foundations of rigid geometry. {I}}, EMS
  Monographs in Mathematics, European Mathematical Society (EMS), Z\"{u}rich,
  2018.


\bibitem[Kie67]{KiehlStein}
R.~Kiehl, \emph{Theorem {A} und {T}heorem {B} in der nichtarchimedischen
  {F}unktionentheorie}, Invent. Math. \textbf{2} (1967), 256--273.

\bibitem[Lau96]{LaumonFourierGen}
G.~Laumon, \emph{Transformation de {F}ourier generalis\'ee}, ar{X}iv, 1996,
  \url{https://arxiv.org/abs/alg-geom/9603004}.

\bibitem[L{\"u}t73]{LutkebohmertStein}
W.~L{\"u}tkebohmert, \emph{Steinsche {R}\"aume in der nichtarchimedischen
  {F}unktionentheorie}, Schr. Math. Inst. Univ. M\"unster (2) (1973), no.~6,
  ii+55.

\bibitem[L{\"{u}}t16]{LutkebohmertCurves}
W.~L{\"{u}}tkebohmert, \emph{Rigid geometry of curves and their {J}acobians},
  Ergebnisse der Mathematik und ihrer Grenzgebiete. 3. Folge. A Series of
  Modern Surveys in Mathematics, vol.~61, Springer, Cham, 2016.

\bibitem[Mac22]{RigidAnalyticUniversalVectorExtension}
M.~Maculan, \emph{The universal vector extension of an abeloid variety},
  ar{X}iv, 2022.

\bibitem[MM60]{MatsushimaMorimoto}
Y.~Matsushima and A.~Morimoto, \emph{Sur certains espaces fibr\'{e}s
  holomorphes sur une vari\'{e}t\'{e} de {S}tein}, Bull. Soc. Math. France
  \textbf{88} (1960), 137--155.

\bibitem[MM74]{MazurMessing}
B.~Mazur and W.~Messing, \emph{Universal extensions and one dimensional
  crystalline cohomology}, Lecture Notes in Mathematics, Vol. 370,
  Springer-Verlag, Berlin-New York, 1974.

\bibitem[MP21]{notions}
M.~Maculan and J.~Poineau, \emph{{N}otions of {S}tein spaces in
  non-{A}rchimedean geometry}, J. Algebraic Geom. \textbf{30} (2021), 287--330.

\bibitem[Nee88]{Neeman}
A.~Neeman, \emph{Steins, affines and {H}ilbert's fourteenth problem}, Ann. of
  Math. (2) \textbf{127} (1988), no.~2, 229--244.

\bibitem[Oda69]{OdaASENS}
T.~Oda, \emph{The first de {R}ham cohomology group and {D}ieudonn\'{e}
  modules}, Ann. Sci. \'{E}cole Norm. Sup. (4) \textbf{2} (1969), 63--135.

\bibitem[Ser59]{SerreGroupesAlgebriques}
J.-P. Serre, \emph{Groupes alg\'{e}briques et corps de classes}, Publications
  de l'institut de math\'{e}matique de l'universit\'{e} de Nancago, VII.
  Hermann, Paris, 1959.

\bibitem[SGA 1]{SGA1}
\emph{Rev\^{e}tements \'{e}tales et groupe fondamental ({SGA} 1)}, Documents
  Math\'{e}matiques (Paris), vol.~3, Soci\'{e}t\'{e} Math\'{e}matique de
  France, Paris, 2003, S\'{e}minaire de g\'{e}om\'{e}trie alg\'{e}brique du
  Bois Marie 1960--61., Directed by A. Grothendieck, With two papers by M.
  Raynaud, Updated and annotated reprint of the 1971 original.

\bibitem[SGA 3]{SGA3}
Philippe Gille and Patrick Polo (eds.), \emph{Sch\'{e}mas en groupes ({SGA} 3).
  {T}ome {I}. {P}ropri\'{e}t\'{e}s g\'{e}n\'{e}rales des sch\'{e}mas en
  groupes}, Documents Math\'{e}matiques (Paris), vol.~7, Soci\'{e}t\'{e}
  Math\'{e}matique de France, Paris, 2011, S\'{e}minaire de G\'{e}om\'{e}trie
  Alg\'{e}brique du Bois Marie 1962--64., A seminar directed by M. Demazure and
  A. Grothendieck with the collaboration of M. Artin, J.-E. Bertin, P. Gabriel,
  M. Raynaud and J-P. Serre, Revised and annotated edition of the 1970 French
  original.

\bibitem[Stacks]{stacks-project}
The {S}tacks~project authors, \emph{The {S}tacks project},
  \url{https://stacks.math.columbia.edu}, 2019.

\bibitem[SW13]{ScholzePDivisibleGroups}
P.~Scholze and J.~Weinstein, \emph{Moduli of {$p$}-divisible groups}, Camb. J.
  Math. \textbf{1} (2013), no.~2, 145--237.

\end{thebibliography}

\bibliographystyle{amsalpha}

\end{document}